\documentclass[a4paper]{amsart}

\usepackage{amscd, amssymb, calc, ifthen, mathrsfs}
\usepackage[matrix,arrow,curve,cmtip]{xy}
\usepackage{amsxtra}
\usepackage{hyperref}
\usepackage{microtype}
\usepackage{ellipsis}

\DeclareMathAlphabet{\mathpzc}{OT1}{pzc}{m}{it}

\let \: \colon
\newcommand{\spacesubsec}[1]{\medskip\subsec{{#1}}\medskip}
\newcommand{\id}{{\mathrm{id}}}

\newcommand{\vbl}{-}

\newcommand{\tn}{\otimes}           

\newcommand{\mathbold}{\bf}

\newcommand{\red}[1]{{#1}_{\mathrm{red}}}
\def\longisomap{{\,\buildrel \sim\over\longrightarrow\,}} 
\newcommand{\pr}{{\mathrm{pr}}}
\newcommand{\subsec}[1]{{\begin{trivlist}\item\em\large#1\end{trivlist}}}

\newcommand{\longlabelmap}[1]{{\,\buildrel #1\over\longrightarrow\,}}

\newcommand{\longlabelmaps}[2]{{\rightrightarrows}}
\newcommand{\xright}[1]{\xrightarrow{\;#1\;}}

\newcommand{\longmap}{{\,\longrightarrow\,}}

\newcommand{\Hom}{{\mathrm{Hom}}}

\DeclareMathOperator{\Spec}{Spec}
\DeclareMathOperator{\underSpec}{\Spec}
\DeclareMathOperator{\Spf}{Spf}

\newcommand{\sO}{{\mathcal{O}}}

\newcommand{\sB}{{\mathcal{B}}}

\newcommand{\bF}{{\mathbold F}}
\newcommand{\bA}{{\mathbold A}}
\newcommand{\bP}{{\mathbold P}}
\newcommand{\bQ}{{\mathbold Q}}

\newcommand{\bZ}{{\mathbold Z}}

\newcommand{\bN}{{\mathbold N}}

\newcommand{\m}{{\mathfrak{m}}}

\newcommand{\p}{\mathfrak{p}}
\newcommand{\q}{\mathfrak{q}}

\def\longisomap{{\,\buildrel \sim\over\longrightarrow\,}} 

\newcommand{\comment}[1]{}

\renewcommand{\leq}{\leqslant}
\renewcommand{\geq}{\geqslant}

\DeclareMathOperator{\colimm}{colim}
\DeclareMathOperator{\limm}{lim}
\newcommand{\colim}{\mathop{\colimm}}
\newcommand{\coeq}{\mathrm{coeq}}

\newcommand{\smcoprod}{{\,\scriptstyle\amalg\,}}
\newcommand{\smsmcoprod}{{\,\scriptscriptstyle\amalg\,}}

\newcommand{\eps}{\varepsilon}

\DeclareMathOperator{\depth}{depth}

\newcommand{\labeleq}[1]{{\,\buildrel #1\over=\,}}

\newcommand{\rightlabelxyarrows}[2]{{\ar@<0.7ex>^-{#1}[r]\ar@<-0.7ex>_-{#2}[r]}}

\newcommand{\displaylabelfork}[6]{{	\entrymodifiers={+!!<0pt,\fontdimen22\textfont2>}
	\def\objectstyle{\displaystyle}
\xymatrix{{#1} \ar^-{#2}[r] & {#3} \ar@<0.7ex>^-{#4}[r]\ar@<-0.7ex>_-{#5}[r] & {#6}}}}

\newcommand{\displaylonglabelcofork}[6]{{	\entrymodifiers={+!!<0pt,\fontdimen22\textfont2>}
	\def\objectstyle{\displaystyle}
\xymatrix@C=40pt{{#1} \ar@<0.7ex>^-{#2}[r]\ar@<-0.7ex>_-{#3}[r] & {#4} \ar^-{#5}[r] & {#6}}}}

\newcommand{\displaylabelcofork}[6]{{	\entrymodifiers={+!!<0pt,\fontdimen22\textfont2>}
	\def\objectstyle{\displaystyle}
\xymatrix{{#1} \ar@<0.7ex>^-{#2}[r]\ar@<-0.7ex>_-{#3}[r] & {#4} \ar^-{#5}[r] & {#6}}}}

\newcommand{\displaycofork}[3]{{\displaylabelcofork{#1}{}{}{#2}{}{#3}}}
\newcommand{\displayfork}[3]{{\displaylabelfork{#1}{}{#2}{}{}{#3}}}

\newcommand{\displaylabelrightarrows}[4]{{\entrymodifiers={+!!<0pt,\fontdimen22\textfont2>}
	\def\objectstyle{\displaystyle}
\xymatrix{{#1} \ar@<0.7ex>^-{#2}[r]\ar@<-0.7ex>_-{#3}[r] & {#4}}}}

\DeclareMathOperator{\Gr}{Gr}

\newcommand{\bcp}{\odot}
\newcommand{\ptst}{E}
\newcommand{\pind}{\alpha}

\newcommand{\gh}[1]{w_{#1}}
\newcommand{\rgh}[1]{{\bar{w}_{#1}}}
\newcommand{\cgh}[1]{{\kappa_{#1}}}
\newcommand{\rcgh}[1]{{\bar{\kappa}_{#1}}}
\newcommand{\aff}{\mathsf{Aff}}
\newcommand{\affrel}{\mathsf{AffRel}}

\newcommand{\Space}{\mathsf{Sp}}
\newcommand{\AlgSp}{\mathsf{AlgSp}}

\newcommand{\nset}{{\bN^{(\ptst)}}}

\newcommand{\inclmap}[1]{r_{{#1}}}
\newcommand{\projmap}[1]{s_{{#1}}}

\newcommand{\wnus}[1]{{W_{#1}^{*}}}
\newcommand{\wnls}[1]{W_{{#1*}}}
\newcommand{\wius}{{\wnus{}}}
\newcommand{\wils}{{\wnls{}}}

\newcommand{\cat}{{\mathsf{C}}}

\newtheoremstyle{mythm}{}{}%
  {\itshape}
  {}
  {\bfseries}
  {}
  { }
  {\thmnumber{#2.\hspace{1.5mm}}\thmname{#1}\thmnote{ #3}.}

\newtheoremstyle{intro}{}{}%
  {\itshape}
  {}
  {\bfseries}
  {}
  { }
  {\thmname{#1}\thmnumber{ #2}\thmnote{ #3}.}

\newtheoremstyle{myrmk}{}{}%
  {}
  {}
  {}
  {}
  { }
  {{\bfseries\thmnumber{#2.\hspace{1.5mm}}}{\itshape\thmname{#1}}\thmnote{ #3}.}

\numberwithin{equation}{subsection}

\theoremstyle{mythm}

\newtheorem{theorem}[subsection]{Theorem}

\newtheorem{proposition}[subsection]{Proposition}
\newtheorem{lemma}[subsection]{Lemma}

\newtheorem{corollary}[subsection]{Corollary}

\theoremstyle{myrmk}

\theoremstyle{intro}
\newtheorem*{thmintro}{Theorem}

\theoremstyle{plain}
\newtheorem*{prop*}{Proposition}
\newtheorem*{cor*}{Corollary}
\newtheorem*{conj*}{Conjecture}

\theoremstyle{definition}

\makeatletter
\def\@seccntformat#1{\@ifundefined{#1@cntformat}%
{\csname the#1\endcsname\quad}
{\csname #1@cntformat\endcsname}
}
\def\section@cntformat{\thesection.\enspace}
\def\subsection@cntformat{\thesubsection.}
\makeatother

\begin{document}

\title[The basic geometry of Witt vectors, II]{The basic geometry of Witt vectors, II\\ Spaces}
\author[J.~Borger]{James Borger}
\address{Australian National University}

\email{james.borger@anu.edu.au}
\thanks{This work was partly supported by Discovery Project DP0773301, 
a grant from the Australian Research Council.}

\begin{abstract}
This is an account of the algebraic geometry of Witt vectors and arithmetic jet spaces. The 
usual, ``$p$-typical'' Witt vectors of $p$-adic schemes of finite type are already reasonably
well understood. The main point here is to generalize this theory in two ways. We allow
not just $p$-typical Witt vectors but those taken with respect to any set of
primes in any ring of integers in any global field, for example. This includes the ``big'' Witt
vectors. We also allow not just $p$-adic schemes of finite type but arbitrary algebraic spaces over
the ring of integers in the global field. We give similar generalizations of Buium's formal
arithmetic jet functor, which is dual to the Witt functor. We also give concrete 
geometric descriptions of Witt spaces and arithmetic jet spaces and investigate whether a number
of standard geometric properties are preserved by these functors.
\end{abstract}

\maketitle

\section*{Introduction}
Let $p$ be a prime number. For any integer $n\geq 0$ and any (commutative) ring $A$,
let $W_n(A)$ denote the ring of $p$-typical Witt vectors of length $n$ with entries in $A$.
This construction gives a functor $W_n$ from the category of rings to itself.
It is an important tool in number theory,
especially in the cohomology of varieties over $p$-adic fields.
For example, it is used in the definition of Fontaine's period 
rings~\cite{Fontaine:Corps-des-periodes-p-adiques}
and in the definition of the de~Rham--Witt complex, which is an explicit complex that computes
crystalline cohomology~\cite{Illusie:dRW-1016}. 

The functor $W_n$ has a left adjoint, which we denote by $A\mapsto \Lambda_{n}\bcp A$:
	\begin{equation}
		\label{eq:intro-1}
		\Hom(\Lambda_n\bcp A, B) \cong \Hom(A,W_n(B)).
	\end{equation}
This adjunction was first considered by Greenberg \cite{Greenberg:I}\cite{Greenberg:II}, but he
restricted himself to the case where $B$ is an $\bF_p$-algebra, and so he only constructed the
special fiber $\bF_p\tn_{\bZ}(\Lambda_n\bcp A)$. The construction of the full functor had to wait
until Joyal~\cite{Joyal:Witt} and, independently, Buium~\cite{Buium:p-jets}. It also has
applications in number theory, most notably in the study of $p$-adic points on
varieties. For example, see Buium \cite{Buium:p-jets} and Buium--Poonen
\cite{Buium-Poonen:Special-points-II} (as well as Buium's earlier
work~\cite{Buium:Intersections-in-jet-spaces} for applications of analogous constructions
in complex algebraic
geometry). These two adjoint constructions see different sides of the arithmetic of $A$---the
ring $W_n(A)$ sees certain maps into $A$, and the ring $\Lambda_n\bcp A$ sees certain maps out of 
it.

This paper is part of a general program to analyze varieties over global fields using global
analogues of these functors, such as the ``big'' Witt functors. The first issue one
faces is that, even in the $p$-typical case above, the schemes $\Spec W_n(A)$ and $\Spec
\Lambda_n\bcp A$ are not familiar geometric constructions, and it is important that we be able to
handle them with ease. The purpose of this paper is to demonstrate that
this is possible. The first part~\cite{Borger:BGWV-I} developed the affine theory, and this part 
extends it to arbitrary schemes and algebraic spaces.

\bigskip

Let us go over the contents in more detail. We will work throughout with certain generalizations of
the classical $p$-typical and big Witt functors. These are the $\ptst$-typical Witt functors
$W_{R,\ptst,n}$ defined in \cite{Borger:BGWV-I}. These functors depend on a ring $R$, a set $\ptst$
of finitely presented maximal ideals $\m$ of $R$ with the property that each localization $R_{\m}$
is discrete valuation ring with finite residue field, and an element $n\in\nset=\bigoplus_E \bN$.
Then $W_{R,\ptst,n}$ is a functor from the category of $R$-algebras to itself. We recover the
$p$-typical Witt functor when $\ptst$ consists of the single maximal ideal $p\bZ$ of $\bZ$, and
we recover the big Witt functor when $\ptst$ consists of all the maximal ideals of $\bZ$. When
$\ptst$ consists of the maximal ideal of the valuation ring of a local field, we recover a variant
of the $p$-typical Witt functor due to Drinfeld~\cite{Drinfeld:symmetric-domains} and to
Hazewinkel~\cite{Hazewinkel:book}, (18.6.13). While the general $\ptst$-typical functors are
necessary for future applications, all phenomena in this paper occur already in the $p$-typical
case. This is because for foundational questions, the general methods of \cite{Borger:BGWV-I}
usually allow one to reduce matters to the case where $\ptst$ consists of a single principal ideal,
and in this case, the classical $p$-typical functor is a representative example.

As in the $p$-typical case above, the functor $W_{R,\ptst,n}$
has a left adjoint, which in general we will denote by $A\mapsto \Lambda_{R,\ptst,n}\bcp A$.
But let us write $\Lambda_n\bcp A=\Lambda_{R,\ptst,n}\bcp A$ and $W_n=W_{R,\ptst,n}$, for short.
The first concern of this paper is to extend both of these functors to the category of algebraic 
spaces over $R$, including for example all $R$-schemes. 
In fact, as explained in Grothendieck--Verdier (SGA 4, exp.\ III \cite{SGA4.1}), there is 
a general way of doing this---we only need to verify that $W_n$ satisfies certain 
properties. The method is as follows.
Let $\aff_S$ denote the category of affine schemes over $S=\Spec R$. 
Then $W_n$ induces a functor $\aff_S\to\aff_S$, which we also denote by $W_n$.
So we have $W_n(\Spec A) = \Spec W_n(A)$.
This functor has two important properties. First, if $U\to X$ and $V\to X$ are
\'etale maps in $\aff_S$, then the induced map
	\[
	W_n(U\times_X V) \longmap W_n(U)\times_{W_n(X)} W_n(V)
	\]
is an isomorphism. Second, if $(U_i\to X)_{i\in I}$ is a covering family of \'etale maps, then so
is the induced family $\bigl(W_n(U_i)\to W_n(X)\bigr)_{i\in I}$. 
Both of these are consequences of van der
Kallen's theorem for $\ptst$-typical Witt functors,
which says that $W_n$ preserves \'etale maps of $R$-algebras (\cite{Borger:BGWV-I}, theorem B). 

It then follows from general sheaf theory
that if $X$ is a sheaf of sets on $\aff_S$ in the \'etale topology, then
so is the functor $\wnls{n}(X)=X\circ W_n$, thus giving a functor $\wnls{n}$ from the category
$\Space_{S}$ of sheaves of sets on $\aff_{S}$ to itself. By another general theorem, 
this functor $\wnls{n}\:\Space_S\to\Space_S$ has a left adjoint $\wnus{n}$ satisfying
 	\[
	\wnus{n}(X) = \colim_{U} W_n(U),
	\]
where $U$ runs over the category of affine schemes equipped with a map to $X$ and where
we identify the affine scheme $W_n(U)$ with the object of $\Space_S$ it represents. These
functors extend the affine functors $W_n$ and $\Lambda_n\bcp\vbl$ to $\Space_S$:
	\begin{equation}
		\label{eq:intro-extension-of-affine-functors}
	\wnus{n}(\Spec A) = \Spec W_n(A), \quad\quad \wnls{n}(\Spec A) = \Spec \Lambda_n\bcp A.
	\end{equation}
They are the extensions we will consider. In fact, by the discussion above, they are the unique 
extensions satisfying certain natural properties.

\begin{thmintro}[A]\label{thm:intro-B}
	If $X\in\Space_{S}$ is an algebraic space, then so are $\wnus{n}(X)$
	and $\wnls{n}(X)$. If $X$ is a scheme, then so are $\wnus{n}(X)$ and $\wnls{n}(X)$.
\end{thmintro}

We call $\wnus{n}(X)$ the $\ptst$-typical Witt space of $X$ of length $n$, and we call
$\wnls{n}(X)$ the $\ptst$-typical arithmetic jet space of $X$ of length $n$. In certain cases, they
have been constructed before. In their appendix, Langer and Zink~\cite{Langer-Zink:dRW} constructed
the $p$-typical Witt space of a general $\bZ_p$-scheme $X$. For earlier work see
Bloch~\cite{Bloch:dRW}, Lubkin~\cite{Lubkin:bounded-Witt}, and Illusie~\cite{Illusie:dRW}.
Buium~\cite{Buium:p-jets} has constructed the $p$-typical arithmetic jet space of a formal
$\bZ_p$-scheme, extending Greenberg's construction of the special
fiber~\cite{Greenberg:I}\cite{Greenberg:II}. When $R$ is $\bZ$ and $\ptst$ is arbitrary, Buium and
Simanca have constructed the arithmetic jet spaces for affine schemes and have constructed certain
approximations to it for general schemes~\cite{Buium-Simanca:Arithmetic-Laplacians} (Defintion
2.16).

For the reader who does not have a mind for abstract sheaf theory, let us reinterpret theorem A in
the language of covers. The most obvious way of defining the Witt space of a separated scheme $X$
is to choose an affine open cover $(U_i)_{i\in I}$ of $X$ and to define $\wnus{n}(X)$ to be the
result of gluing the affine schemes $W_n(U_i)$ along the affine schemes $W_n(U_i\times_X U_j)$. It
is not hard to check that this gives a scheme which is independent of the cover. (If $X$ is
arbitrary, then $U_i\times_X U_j$ is separated, and so we can define $\wnus{n}(X)$ in general by
doing this procedure twice.) This is in Langer--Zink~\cite{Langer-Zink:dRW} in the $p$-typical
case, and the general $\ptst$-typical case is no harder. When $X$ is an algebraic space and
$(U_i)_{i\in I}$ is an \'etale cover, we need to know that $\coprod_{i,j}W_n(U_i\times_X U_j)$ is
an \'etale equivalence relation on $\coprod_i W_n(U_i)$. This requires van der Kallen's theorem and
a more sophisticated gluing argument, but the principle is the same. Instead the approach of this
paper is to define $\wnus{n}(X)$ as an object of $\Space_S$ and to prove later that it is a scheme
or an algebraic space. If one cares about $\wnus{n}$ only for schemes and algebraic spaces, then
the difference is mostly a matter of organization.

This method does not work as well with $\wnls{n}$, because it is rarely the case that the
$\wnls{n}(U_i)$ cover $\wnls{n}(X)$. Indeed, generically over $\Spec R$, the space $\wnls{n}(X)$
agrees with a certain cartesian power $X^N$, and of course one cannot usually construct $X^N$ by
gluing the $U_i^N$ together. For $p$-adic formal schemes in the $p$-typical case, the generic fiber
is empty and this method does actually work, but in general it does not. Instead we must use the
total space $U=\coprod_i U_i$ of the cover. We will prove below that $\wnls{n}(U\times_X U)$ is an
\'etale equivalence relation on $\wnls{n}(U)$, and the quotient is $\wnls{n}(X)$. If $X$ is
quasi-compact and separated, we can assume $U$ and $U\times_X U$ are affine, and then
$\wnls{n}(X)$ becomes the quotient of a known affine scheme by a known affine \'etale equivalence
relation. And so we could avoid abstract sheaf theory for such $X$ by
taking this to be the definition of $\wnls{n}(X)$, although it would still take a small argument to
prove that $\wnls{n}$ is the right adjoint of $\wnus{n}$ and that it sends schemes to
schemes, rather than just algebraic spaces. It would also take some work to remove the assumption
that $X$ is quasi-compact and separated, but of course it could be done. Instead
we will define $\wnls{n}(X)$ in one stroke as an object of $\Space_S$ and then prove the
representability properties later.

Another benefit to working with the whole category $\Space_S$ is that is allows us 
to make the infinite-length constructions
	\begin{equation}
		\label{eq:intro-infinite-length-spaces}
	\wius(X)=\colim_n \wnus{n}(X), \quad\quad \wils(X)=\lim_n\wnls{n}(X).
	\end{equation}
These constructions are ind-algebraic spaces (resp.\ pro-algebraic spaces) but are generally not
algebraic spaces. While it would be possible to remain in the category of schemes or algebraic 
spaces by treating them as inductive systems
(resp.\ projective systems), it is convenient to be able to pass to the limit in $\Space_S$.
We will only consider the finite-length constructions in this
paper, but it is in fact the infinite-length ones that are of ultimate interest. Further, we will
eventually want to consider iterated constructions, such as $\wius\wius(X)$, and so
it is convenient to have $\wius(X)$ defined when $X$ ind-algebraic, and 
to have $\wils(X)$ defined when $X$ is pro-algebraic. At this point, it becomes easier just to
let $X$ be any object of $\Space_S$.

\begin{table}[t]
\begin{tabular}{|p{43mm}|c|c|}
\hline
Property of algebraic spaces	& Preserved 			& Reference or \\
								& by $\wnus{n}$?		& counterexample \\
\hline\hline
quasi-compact					& yes	& \ref{pro:W-preserves-qcom} \\
quasi-separated					& yes	& \ref{pro:wnus-preserves-separatedness} \\
\hline
affine							& yes	& \ref{pro:W_n-of-affine-is-affine} \\
a scheme						& yes	& \ref{cor:W-preserves-open-immersions}\\
of Krull dimension $d$			& yes	& \ref{pro:W-preserves-etale-local-properties}	 \\
separated						& yes	& \ref{pro:wnus-preserves-separatedness} \\
reduced and flat over $S$		& yes	& \ref{pro:W-preserves-etale-local-properties} \\
reduced							& no	& $W_1(\bF_p)$ \\
regular, normal					& no	& $W_1(\bZ)$ \\
\hline
(locally) noetherian 			& $\text{yes}^b$
								& \ref{subsec:W-does-not-preserve-relative-fg} +
								\ref{pro:W-preserves-etale-local-properties}  \\
$\text{S}_k$ (Serre's property)	& $\text{yes}^b$	& \ref{pro:W-preserves-depth-bound}\\
Cohen--Macaulay					& $\text{yes}^b$	& \ref{pro:W-preserves-depth-bound}\\
Gorenstein						& no	& $W_2(\bZ)$ \\
local complete intersection		& no	& $W_2(\bZ)$ \\
\hline
\end{tabular}
\vspace{\abovecaptionskip}
\caption[Absolute properties]{This table indicates whether the given property of algebraic spaces
$X$ over $S$ is preserved by $\wnus{n}$ in general. The superscript \emph{b} means that $X$ is
assumed to be locally of finite type over $S$ and that $S$ is assumed to be noetherian. 
In the counterexamples, $W_n$
denotes the $p$-typical Witt vectors over $\bZ$ of length $n$ (traditionally denoted
$W_{n+1}$).}
\end{table}

\spacesubsec{Preservation of properties by $\wnus{n}$}

We will spend some time looking at whether common properties of algebraic spaces and maps
are preserved by $\wnus{n}$. Rather than state the results formally, I have arranged them into
tables 1 and 2. (Note that we use normalized indexing throughout. So our $p$-typical Witt functor
$W_n$ is what is traditionally denoted $W_{n+1}$. The reasons for this are explained
in~\cite{Borger:BGWV-I},~2.5.)

Several results in the $p$-typical case are folklore or have appeared elsewhere. See,
for example, Bloch~\cite{Bloch:dRW}, Illusie~\cite{Illusie:dRW}, or
Langer--Zink~\cite{Langer-Zink:dRW}. Perhaps the most interesting of them is that
while smoothness over $S$ and regularity are essentially never preserved by $\wnus{n}$, being
Cohen--Macaulay always is. As with the work of Ekedahl and Illusie on $p$-typical Witt vectors of
$\bF_p$-schemes \cite{Ekedahl:dRW-I}\cite{Ekedahl:dRW-II}\cite{Illusie:dRW-1016}, this has
implications for Grothendieck duality and de\ Rham--Witt theory, but we will not consider them 
here.

\begin{table}[t]
\begin{tabular}{|p{3.7cm}|c|c|c|c|}
\hline
Property $P$ of maps 			& \multicolumn{2}{c|}{Must $\wnus{n}(f)$ have}
								& \multicolumn{2}{c|}{When $Y=S$, must} \\
$f\:X\to Y$ of algebraic 		& \multicolumn{2}{c|}{property $P$?}
								& \multicolumn{2}{c|}{$\wnus{n}(X)\to S$ have} \\
spaces 							& \multicolumn{2}{c|}{}
								& \multicolumn{2}{c|}{property $P$?} \\
\hline\hline 
\'etale							& yes & \ref{thm:W-main-geometric-finite-length(2)}
 								& no & $\bZ$ \\
an open immersion				& yes & \ref{cor:W-preserves-open-immersions}	
								& no & $\bZ$ \\

\hline 
quasi-compact					& yes & \ref{pro:W-preserves-target-local-properties} 
								& yes & \ref{cor:qcom-and-ftype-over-S} \\
quasi-separated					& yes & \ref{pro:W-preserves-target-local-properties}	
								& yes & \ref{pro:wnus-preserves-separatedness} \\

\hline 
affine							& yes & \ref{pro:W-preserves-affine-local-properties}	
								& yes & + \ref{pro:W(S)-properties} \\
integral						& yes & \ref{pro:W-preserves-affine-local-properties}	
								& yes & + \ref{pro:W(S)-properties} \\
a closed immersion				& yes & \ref{pro:W-preserves-affine-local-properties}
								& no  & $\bZ$ \\
finite \'etale					& yes & \ref{pro:W-preserves-affine-local-properties}
								& no  & $\bZ$ \\

\hline 
separated						& yes & \ref{pro:W-preserves-target-local-properties}	
								& yes & \ref{pro:wnus-preserves-separatedness} \\
surjective						& yes & \ref{pro:W-preserves-target-local-properties}	
								& yes & + \ref{pro:W(S)-properties} \\
universally closed				& yes & \ref{pro:W-preserves-target-local-properties}	
								& yes & + \ref{pro:W(S)-properties} \\

\hline 
locally of finite type		& $\text{yes}^a$
								& \ref{pro:W-preserves-relative-fin-type}	+
								\ref{subsec:W-does-not-preserve-relative-fg}
								& yes & \ref{pro:W-preserves-etale-local-properties} \\
of finite type					& $\text{yes}^a$
								& \ref{pro:W-preserves-relative-fin-type}	+
								\ref{subsec:W-does-not-preserve-relative-fg}
								& yes & \ref{cor:qcom-and-ftype-over-S} \\
finite							& $\text{yes}^a$
								& \ref{pro:W-preserves-relative-fin-type}	+
								\ref{subsec:W-does-not-preserve-relative-fg}
								& yes & \ref{pro:W-preserves-target-local-rel-to-S} \\
proper							& $\text{yes}^a$
								& \ref{pro:W-preserves-relative-fin-type} +
								\ref{subsec:W-does-not-preserve-properness}
								& yes & + \ref{pro:W(S)-properties} \\

\hline 
flat							& no & $\bZ[x]$
								& yes & \ref{pro:W-preserves-etale-local-properties} \\
faithfully flat					& no & $\bZ[x]$
								& yes & \ref{pro:W-preserves-target-local-rel-to-S}\\
Cohen-Macaulay					& no & $\bZ[x]$
								& $\text{yes}^b$ & \ref{pro:W-preserves-depth-bound}\\
$\text{S}_k$ (Serre's property)		& no & $\bZ[x]$
								& $\text{yes}^b$ & \ref{pro:W-preserves-depth-bound}\\
smooth							& no & $\bZ[x]$
								& no & $\bZ$\\
finite flat						& no & $\bZ[\sqrt{p}]$	
								& yes & \ref{pro:W-preserves-target-local-rel-to-S} \\

\hline
\end{tabular}
\vspace{\abovecaptionskip}

\caption[Relative properties]{This table indicates whether the given property $P$ of morphisms 
of algebraic spaces over $S$ is preserved by $\wnus{n}$ in general. The
central two columns indicate whether $P$ is preserved by $\wnus{n}$ and give either a
reference to the main text or a counterexample. The right columns indicate whether
the structure map $\wnus{n}(X)\to S$ must satisfy $P$ when the structure map $X\to S$
does. The superscript \emph{a} means that $X$ and $Y$ are assumed to be locally of finite type over 
$S$; and \emph{b} means that also $S$ is assumed to be noetherian.
The counterexamples are for $\wnus{1}$, the $p$-typical Witt functor of length 1,
with $X$ the spectrum of the given ring and $Y=\Spec\bZ$.}

\end{table}

\spacesubsec{Preservation of properties by $\wnls{n}$}

Preservation results for $\wnls{n}$ are typically easier to establish. This is because many common
properties of morphisms are naturally expressed in terms of the functor of points, and the functor
of points of $\wnls{n}(X)$ is described simply in terms of that of $X$. For the same reason, many
of these results extend readily beyond algebraic spaces to the category $\Space_S$; this is unlike
with $\wnus{n}$, where we usually need to make representability assumptions.

A number of the results are displayed in table 3. Because we have $\wnls{n}(S)=S$, the preservation
of properties relative to $S$ is a special case of the preservation of properties of morphisms.
This is unlike the case with $\wnus{n}$, where we have the right-hand pair of columns in table 2.
I have mostly ignored whether absolute properties, such as regularity, are preserved by $\wnls{n}$.
This is because such properties are usually not preserved by products over $S$, and in that case
they would fail to be preserved by $\wnls{n}$ for the trivial reason that $\wnls{n}$ is a product
functor away from the ideals of $\ptst$. This is like the case with $\wnus{n}$: properties that are
not preserved by disjoint unions, such as connectedness, are not listed in table 1.

\spacesubsec{Geometric descriptions}

As explained above, both $\wnus{n}(X)$ and $\wnls{n}(X)$ can be described 
in terms of the case where $X$ is affine by using charts. 
But under some flatness restrictions on $X$, it is
possible to construct $\wnus{n}(X)$ and $\wnls{n}(X)$ in purely geometric terms without mentioning
Witt vectors or arithmetic jet spaces at all. I will give the descriptions here in the $p$-typical
case when $n=1$; the general case is in the body of the paper.

Let us first consider the Witt space $\wnus{1}(X)$. Assume that $X$ is flat over $\bZ$ locally at 
$p$. Let $X_0$ denote the special fiber $X\times_{\Spec\bZ}\Spec\bF_p$. Then 
the theorem is that $\wnus{1}(X)$ is the 
coequalizer in the category of algebraic spaces of the two maps
	\[
		\displaylabelrightarrows{X_0}{i_1\circ F}{i_2}{X\smcoprod X,}
	\]
where $i_j\:X_0\to X \smcoprod X$ denotes the canonical closed immersion into the $j$-th component
of $X\smcoprod X$ and where $F$ is the absolute Frobenius endomorphism of $X_0$. For general $n$,
the space $\wnus{n}(X)$ can be constructed by gluing $n+1$ copies of $X$ together in a similar but
more complicated way along their fibers modulo $p$,\dots,$p^n$. See
\ref{thm:presentation-of-W-of-alg-space}.

For the arithmetic jet space $\wnls{1}(X)$, we need to assume that $X$ is smooth over $\bZ$ locally
at $p$. Let $I$ denote the ideal sheaf on $X\times X$ defining the graph of the Frobenius map on
the special fiber $X_0$, and let $\sB$ denote the sub-$\sO_{X\times X}$-algebra of $\sO_{X\times
X}[1/p]$ generated by the subsheaf $p^{-1}I$. Then the theorem is that $\wnls{1}(X)$ is naturally
isomorphic to the relative spectrum $\underSpec (\sB)$ over $X\times X$. (One might hope that it is
also worth studying the full blow up of $X\times X$ along $I$.) In particular, the map
$\wnls{1}(X)\to X\times X$ is affine and is an isomorphism outside the fiber over $p$. For general
$n$, the space $\wnls{n}(X)$ can be constructed by taking a similar but more complicated affine
modification of $X^{n+1}$. See \ref{thm:blow-up-description-of-w-lower-s}.

\begin{table}[t]
\begin{tabular}{|p{5cm}|c|c|}
\hline
Property of maps of				& Preserved 		& Reference or\\								
algebraic spaces 				& by $\wnls{n}$?	& counterexample\\

\hline\hline 
(formally) \'etale, smooth, unram.	& yes & \ref{pro:wnls-preserves-formally-etale-etc} \\
a monomorphism					& yes	& \ref{pro:wnls-preserves-epis-and-monos} \\
an open immersion				& yes	& \ref{pro:wnls-preserves-formally-etale-etc} +
											\ref{pro:wnls-preserves-epis-and-monos}\\

\hline 
quasi-compact					& yes	& \ref{pro:wnls-preserves-qcom-qsep} \\
quasi-separated					& yes	& \ref{pro:wnls-preserves-qcom-qsep} \\
epimorphism in $\Space_S$		& yes	& \ref{pro:wnls-preserves-epis-and-monos} \\

\hline 
affine							& yes	& \ref{pro:wnls-preserves-target-local-properties}\\
a closed immersion				& yes	& \ref{pro:wnls-preserves-target-local-properties}\\
integral, finite				& no 	& $\bZ\times\bZ$ \\
finite \'etale, finite flat		& no 	& $\bZ\times\bZ$ \\

\hline 
(locally) of finite type/pres.	& yes	& \ref{pro:wnls-preserves-target-local-properties}\\
separated						& yes	& \ref{pro:wnls-preserves-target-local-properties}\\
smooth and surjective			& yes	& \ref{pro:wnls-preserves-target-local-properties}\\
surjective						& no	& $\bZ[\sqrt{p}]$ \\
proper, universally closed		& no	& $\bZ\times\bZ$ \\
smooth and proper				& no	& $\bZ\times\bZ$ \\

\hline 
flat							& no	& $\bZ[x]/(x^2-px)$ \\
faithfully flat					& no	& $\bZ[x]/(x^2-px)$ \\
Cohen-Macaulay					& no	& $\bZ[x]/(x^2-px)$ \\
$\text{S}_k$ (Serre's property)	& no	& $\bZ[x]/(x^2-px)$ \\

\hline
\end{tabular}
\vspace{\abovecaptionskip}

\caption[Relative properties]{This table indicates whether the given property of maps of algebraic
spaces over $S$ is preserved by $\wnls{n}$ in general. The counterexamples are for the $p$-typical 
jet functor $\wnls{1}$ applied to the map $\Spec A\to \Spec\bZ$, where $A$ is
the given ring. See \ref{subsec:counterexamples-for-preservation-by-wnls}.}

\end{table}

\spacesubsec{Absolute algebraic geometry}

Let us end with a few words on how the Witt and jet functors relate
to the philosophy of absolute algebraic geometry. The first hope of this
philosophy is that there exists a category whose relationship to the category of schemes over $\bZ$
is analogous to the relationship of $\bF_p$ to 
$\bF_p[t]$. It is sometimes called the category of absolute schemes, or schemes over $\bF_1$. The
second hope is that this category would suggest ways of transporting results in algebraic
geometry over $\bF_p(t)$ to $\bQ$.

There are a number of proposed definitions of this category. One of the general themes is that an
absolute scheme could be defined to be a scheme together with some additional structure, which
should be interpreted as descent data from $\bZ$ to $\bF_1$. One precise proposal for this
structure is a so-called $\Lambda$-structure~\cite{Borger:LRFOE}. If $X$ is a flat scheme over
$\bZ$, then a $\Lambda$-structure is equivalent to a commuting family of maps $\psi_p\:X\to X$,
where $p$ runs over the prime numbers, such that each $\psi_p$ agrees with the Frobenius map on the
fiber of $X$ over $p$. And if $X$ is affine, then a $\Lambda$-structure is equivalent to a
(special) $\lambda$-ring structure on the corresponding ring, in the sense of Grothendieck's
Riemann--Roch theory \cite{Grothendieck:Chern}.

From this point of view, the functor that forgets the $\Lambda$-structure should be thought of as
base change from $\bF_1$ to $\bZ$. Therefore its left adjoint should be thought of as the
base-forgetting functor, and its right adjoint the Weil restriction of scalars. In fact, it is
possible to say explicitly what these adjoints are. Let $\ptst$ be the set of all maximal ideals of
$\bZ$. Then for any space $X\in\Space_{\bZ}$, the infinite-length Witt and jet spaces $\wius(X)$
and $\wils(X)$ of (\ref{eq:intro-infinite-length-spaces}) carry natural $\Lambda$-structures, and
hence give functors from spaces over $\bZ$ to those over $\bF_1$. The first is the left adjoint of
base change and the second is the right adjoint. Thus it is natural to interpret the Witt space
$\wius(X)\in\Space_S$ as $X\times_{\bF_1}\Spec\bZ$ and the arithmetic jet space
$\wils(X)\in\Space_S$ as the base change to $\bZ$ of the Weil restriction of scalars of $X$ to
$\bF_1$. One would interpret the truncated versions $\wnus{n}(X)$ and $\wnls{n}(X)$ as
approximations.

This theme is discussed in more detail in the preprint~\cite{Borger:LRFOE} and will
developed in forthcoming work.

\bigskip

I thank William Messing, Amnon Neeman, and Martin Olsson for some helpful
discussions on technical points, and Alexandru Buium and Lance Gurney for making some helpful
comments on earlier versions of this paper. In particular, Buium's questions led
to~\ref{subsec:relation-to-Greenberg-and-Buium} and all of
section~\ref{sec:geometry-of-arithmetic-jet-spaces}.

\tableofcontents

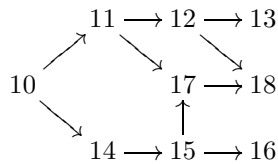
\begin{figure}[h]
\[
\xymatrix@C=30pt@R=25pt@!0{
	&			
	& \text{\ref{sec:sheaf-theoretic-properties-of-W-lower-star}} \ar[r]\ar[dr]
	& \text{\ref{sec:wnls-and-algebraic-spaces}} \ar[r]\ar[dr]
	& \text{\ref{sec:preservation-of-properties-for-wnls}}\\
	& \text{\ref{sec:Basic-functoriality}} \ar[ur]\ar[dr]
	&
	& \text{\ref{sec:ghost-descent-and-the-geometry-of-Witt-spaces}} \ar[r]
	& \text{\ref{sec:geometry-of-arithmetic-jet-spaces}} \\
	&			
	& \text{\ref{sec:the-inductive-lemma-for-wnus}} \ar[r]
	& \text{\ref{sec:wnus-and-algebraic-spaces}} \ar[r]	\ar[u]
	& \text{\ref{sec:geometric-properties-of-W-upper-star}}  \\
}
\]
\caption[Dependencies]{Dependence between sections}
\end{figure}

\section*{Conventions}
This paper is a continuation of~\cite{Borger:BGWV-I}. When we need to refer to results
in~\cite{Borger:BGWV-I}, we will generally not mention the paper itself and instead simply refer to
the subsection or equation number. There is no risk of confusion because the numbering of
this paper is a continuation of the numbering of~\cite{Borger:BGWV-I} (see Fig.\ 1).
We will also keep the general conventions of~\cite{Borger:BGWV-I}.

\setcounter{section}{9}

\section{Sheaf-theoretic foundations} \label{sec:Basic-functoriality}

The purpose of this section is to set up the basic global definitions.
The approach is purely sheaf theoretic in the style of SGA 4~\cite{SGA4.1}.

\subsection{} \emph{Spaces.}
Let $\aff$ denote the category of affine schemes equipped with the \'etale topology: a family
$(X_i\to X)$ is a covering family if each $X_i\to X$ is \'etale and their images cover $X$. Let
$\Space$ denote the category of sheaves of sets on $\aff$. We will call its objects
\emph{spaces}.\footnote{As always with large sites, there are set-theoretic subtleties. So,
precisely, let $\Upsilon$ be a universe containing the universe of discourse. The term
\emph{presheaf} will mean a functor from $\affrel_S$ to the category of $\Upsilon$-small sets, and
the term \emph{sheaf} will mean a presheaf satisfying the sheaf condition. Because $\affrel_S$ is
an $\Upsilon$-small category, we can sheafify presheaves. On the other hand, the categories of
sheaves and presheaves are not true categories because their hom-sets are not necessarily true
sets, but only $\Upsilon$-small sets. A possible way of avoiding set-theoretic issues would be to
consider only sheaves subject to certain set-theoretic smallness conditions, but to my knowledge,
no one has pursued this.} Any scheme represents a contravariant set-valued functor on $\aff$, and
this functor is a sheaf. In this way, the category of schemes can be identified with a full
subcategory of $\Space$.

For any object $S\in\Space$, let $\Space_S$ denote the subcategory of objects over $S$ and let
$\aff_S$ denote the full subcategory of $\Space_S$ consisting of objects $X$ over $S$, where $X$ is
affine. When $S$ is a scheme, define $\affrel_S$ to be the full subcategory of $\aff_S$ consisting
of objects $X$ whose structure map $X\to S$ factors through an affine open subscheme of $S$.
Observe that the inclusion $\affrel_S\to \Space_S$ induces an equivalence between $\Space_S$ and
the category of sheaves of sets on $\affrel_S$, and for convenience, we will typically identify the
two. (The reason for using the site $\affrel_S$, rather than the more common one $\aff_S$, is that
the $\ptst$-typical Witt functors $W_{R,\ptst,n}$ are defined in terms
of the base $R$; so it is more 
convenient to use a generating site in which the objects have an affine base available.)
In the important special case where $S$ itself is affine, $S=\Spec R$, we will often write
$\affrel_R$ and $\Space_R$, and of course we have $\affrel_S=\aff_S$.

\subsection{} \emph{Supramaximal ideals.}
\label{subsec:supramaximal-def-global-base-S}
For the rest of this paper, $S$ will denote a separated scheme.
(In all applications, $S$ will be an arithmetic curve. The extra generality we work in
will not create any more work.) 
Define a supramaximal ideal
on $S$ to be a finitely presented (\cite{EGA-no.4}, 0 (5.3.1))
ideal sheaf $\m$ in $\sO_S$ corresponding to either
	\begin{enumerate}
		\item 
			a closed point whose local
			ring $\sO_{S,\m}$ is a discrete valuation ring with finite residue field, or
		\item the empty subscheme.
	\end{enumerate}
When $S$ is affine, this agrees with the earlier definition 
in~1.2. 

We will generally fix the following notation: Let $(\m_\pind)_{\pind\in\ptst}$ denote a family of
supramaximal ideals of $S$ which are pairwise coprime, that is, for all $\alpha,\beta\in\ptst$ with
$\alpha\neq\beta$, we have $\m_{\alpha}+\m_{\beta} = \sO_S$. For each $\pind$, let $q_\pind$ be the
cardinality of the ring $\sO_{S,\m_\pind}/\m_\pind$. Finally, $n$ will be an element of 
$\nset=\bigoplus_{\ptst}\bN$.

\subsection{} {\em Definition of $\wnus{S,\ptst,n}(X)$ and $\wnls{S,\ptst,n}(X)$.}
\label{subsec:def-of-wnus-and-wnls}
Let $X=\Spec A$ be an object of $\affrel_S$, and 
let $\Spec R$ be an affine open subscheme of $S$ which contains the image of 
the structure map $X\to S$. Then
$W_{R,\ptst,n}(A)$ is independent of $R$, up to a coherent family of canonical isomorphisms. 
(Here we abusively conflate $\ptst$ and $n$ with their restrictions to $R$.)
Indeed, let $\Spec R'$ be another
such subscheme of $S$. Since $S$ is separated, we can assume $\Spec R'\subseteq \Spec R$
and can then apply~(2.6.2). Thus we can safely define 
	\[
	W_{S,\ptst,n}(X)= \Spec W_{R,\ptst,n}(A).
	\]

Now we will pass from $X\in\affrel_S$ to $X\in\Space_S$.
The functor $W_{S,\ptst,n}$ preserves \'etale fiber products.
Indeed, let $f\:\Spec A'\to \Spec A$ and $g\:\Spec A''\to \Spec A$ be \'etale maps in 
$\affrel_S$, and let $\Spec R$
be an affine open subscheme of $S$ containing the image of $\Spec A$, and hence those of $\Spec A'$
and $\Spec A''$; then by~9.4, we have
	\[
	W_{R,\ptst,n}(A'\tn_{A} A'') = W_{R,\ptst,n}(A')\tn_{W_{R,\ptst,n}(A)} W_{R,\ptst,n}(A'').
	\]
Similarly, $W_{S,\ptst,n}$ preserves covering families,
by~9.2 and~6.9.
It  follows from general sheaf theory (see the footnote to SGA 4 III 1.6~\cite{SGA4.1}, say) 
that for any sheaf $X$, the presheaf $X\circ W_{S,\ptst,n}$ is a sheaf. Let us write 
	\[
	\wnls{S,\ptst,n}\:\Space_S\longmap\Space_S
	\] 
for the functor $X\mapsto X\circ W_{S,\ptst,n}$. 
Again by general sheaf theory (SGA 4 III 1.2~\cite{SGA4.1}), the functor
$\wnls{S,\ptst,n}$ has a left adjoint 
	\[
	\wnus{S,\ptst,n}\:\Space_S\longmap\Space_S
	\]
constructed in the usual way.
For any affine open subscheme $\Spec R$ of $S$ and any $R$-algebra $A$, it satisfies 
	\begin{equation} \label{eq:wnus-of-affine-is-W}
		\wnus{S,\ptst,n}(\Spec A) = W_{S,\ptst,n}(\Spec A) = \Spec W_{R,\ptst,n}(A).
	\end{equation}
By the adjunction between $W_n$ and $\Lambda_n\bcp\vbl$, we further have 
	\begin{equation} \label{eq:wnls-of-affine-is-lambda-circle}
		\wnls{S,\ptst,n}(\Spec A) = \Spec (\Lambda_{R,\ptst,n}\bcp A),	
	\end{equation}
for $R$ and $A$ as above.

We call $\wnus{S,\ptst,n}(X)$ the {\em $\ptst$-typical Witt space of $X$ of length $n$} and
$\wnls{S,\ptst,n}(X)$ the {\em $\ptst$-typical (arithmetic) jet space of $X$ of length $n$}. We
will often use shortened forms such as $\wnus{S,n}$, $\wnus{n}$, and so on.

(Note that $\wnus{S,\ptst,n}$ does not generally commute with finite products. For example,
see~9.5. So, despite the notation,
$\wnus{S,\ptst,n}$ is essentially never the inverse-image functor in a map of toposes.)

\subsection{} \emph{Restriction of $S$.}
\label{subsec:change-of-S}
Let $j\:S'\to S$ be a flat monomorphism of schemes (especially an open immersion
or a localization at a point).
There are certain isomorphisms of functors
	\begin{align}
			\label{eq:change-of-S-1}
		\wnus{S',n}\circ j^* &\longisomap j^* \circ \wnus{S,n} \\
			\label{eq:change-of-S-2}
		\wnls{S,n}\circ j_* &\longisomap j_* \circ \wnls{S',n} \\ 
			\label{eq:change-of-S-3}
		\wnus{S,n}\circ j_! &\longisomap j_! \circ \wnus{S',n} \\
			\label{eq:change-of-S-4}
		\wnls{S',n} \circ j^* &\longisomap j^* \circ \wnls{S,n} \\
			\label{eq:change-of-S-5}
		j_! \circ \wnls{S',n} &\longisomap \wnls{S,n}\circ j_!
	\end{align}
which we will find useful. The map~(\ref{eq:change-of-S-1}) restricted to the site $\affrel_S$ was
constructed in~(2.6.3); and it was shown to be an isomorphism in
6.1. It therefore induces an isomorphism~(\ref{eq:change-of-S-2}) on
the whole sheaf category $\Space_{S'}$ and, by adjunction,~(\ref{eq:change-of-S-1}) on $\Space_S$.
Similarly,~(\ref{eq:change-of-S-3}) was constructed on $\affrel_{S'}$
in~(2.6.2); and this induces~(\ref{eq:change-of-S-4}) on
the sheaf category $\Space_S$ and, by adjunction,~(\ref{eq:change-of-S-3}) on $\Space_{S'}$.

Finally (\ref{eq:change-of-S-5}) is defined to be the composition
	\[
	j_!\circ\wnls{S',n} \longisomap j_!\circ \wnls{S',n}\circ j^*\circ j_! 
		\xright{(\ref{eq:change-of-S-1})} j_!\circ j^*\circ \wnls{S,n}\circ j_! 
		\longmap \wnls{S,n} \circ j_!
	\]
where the first map is induced by the unit of the adjunction $j_! \dashv j^*$ and the last by the
counit. Let us show the last map isomorphism. It is enough to show that, for any $X'\in\Space_{S'}$,
the structure map $\wnls{S,n}(j_!(X'))\to S$ factors through $S'$. To do this, it is enough to
assume $X'=S'$, and in this case, we will show $\wnls{S,n}(S')=j_!(\wnls{S',n}(S'))$. It suffices to
show this locally on $S$, by~(\ref{eq:change-of-S-4}), and so we can assume $S$ is affine. Since $S$
is separated, $S'$ is also affine, in which case we can
apply~(2.6.4).

It will be convenient to refer to the following simplified expressions of the isomorphisms above:
	\begin{align}
			\label{eq:localization-of-S-commutes-with-wnus}
		\wnus{S',n}(S'\times_S X) &= S'\times_S \wnus{S,n}(X) \\
		\wnls{S,n}\bigl( j_*(X')\bigr) &= j_* \bigl(\wnls{S',n}(X')) \\ 
			\label{eq:wnus-is-independent-of-S}
		\wnus{S,n}(X') &= \wnus{S',n}(X') \\
			\label{eq:global-base-jet-localization}
		\wnls{S',n}(S'\times_S X) &= S'\times_S \wnls{S,n}(X) \\
			\label{eq:wnls-is-independent-of-S}
		\wnls{S',n}(X') &= \wnls{S,n}(X')
	\end{align}
for $X\in\Space_S$, $X'\in\Space_{S'}$.

\subsection{} \emph{Restriction of $\ptst$.}
Observe that if we let $\ptst'$ denote the support in $\ptst$ of $n\in\nset$,
then we have $W_{S,\ptst',n}=W_{S,\ptst,n}$, and hence $\wnus{S,\ptst',n}=\wnus{S,\ptst,n}$ and 
$\wnls{S,\ptst',n}=\wnls{S,\ptst,n}$. So without loss of generality, we can assume
that $\ptst$ equals the support of $n$ and hence that $\ptst$ is finite.
(This is no longer true in the infinite-length case, but that does not appear in this paper.)

\subsection{} {\em Natural maps.}
\label{subsec:natural-maps}
Natural transformations between Witt vector functors for rings
extend naturally to natural transformations of their sheaf-theoretic variants and,
by adjunction, of the arithmetic jet spaces.

For example, for any partition $\ptst=\ptst'\smsmcoprod\ptst''$, the natural 
isomorphism~(5.4.2) induces natural
isomorphisms
	\begin{gather}
			\label{eq:geom-witt-iterate-upper-s}
		\wnus{S,\ptst'',n''}(\wnus{S,\ptst',n'}(X)) \longisomap \wnus{S,\ptst,n}(X), \\
			\label{eq:geom-witt-iterate-lower-s}
		\wnls{S,\ptst,n}(X) \longisomap \wnls{S,\ptst',n'}(\wnls{S,\ptst'',n''}(X)),
	\end{gather}
for any $X\in\Space_S$.

Similarly, the natural projections $W_{n+i}(A)\to W_{n}(A)$, induced by the inclusion 
$\Lambda_{n}\subseteq\Lambda_{n+i}$, induce natural maps
	\begin{align}
		\wnus{n}(X) &\xright{\inclmap{n,i}} \wnus{n+i}(X),
			\label{map:wnus-change-of-n-inclusion-map} \\
		\wnls{n+i}(X) &\xright{\projmap{n,i}} \wnls{n}(X),
			\label{map:wnls-change-of-n-projection-map}
	\end{align}
which we usually just call the natural inclusion and projection;
and the natural transformations $\psi_i$ of~(2.4.8) induce
natural maps
	\begin{align}
		\wnls{n+i}(X) &\xright{\psi_i} \wnls{n}(X),\\
		\wnus{n}(X) &\xright{\psi_i} \wnus{n+i}(X).
	\end{align}

The affine ghost maps $\gh{i}\:W_n(A) \to A$ (for $i=0,\dots,n$) 
and $\gh{\leq n}\:W_n(A)\to A^{[0,n]}$
of~(2.4.3) and~(2.4.4) induce general
ghost maps
	\begin{align} 
			\label{eq:individual-ghost-map-geometric}
		X &\xright{\gh{i}} \wnus{n}(X), \\
			\label{eq:total-ghost-map-geometric}
		\coprod_{[0,n]} X &\xright{\gh{\leq n}} \wnus{n}(X),
	\end{align}
and, by adjunction, the co-ghost maps 
	\begin{align} 
			\label{eq:individual-co-ghost-map-geometric}
		\wnls{n}(X) &\xright{\cgh{i}} X, \\
			\label{eq:total-co-ghost-map-geometric}
		\wnls{n}(X) &\xright{\cgh{\leq n}} X^{[0,n]}.
	\end{align}
Observe that if every ideal in $\ptst$ is the unit ideal,
then $\gh{\leq n}$ and $\cgh{\leq n}$ are isomorphisms,
simply because they are induced by isomorphisms between the site maps,
by for example~2.7.

When $\ptst$ consists of a single ideal $\m$,
the reduced affine ghost maps $\rgh{n}\:W_n(A)\to A/\m^{n+1}A$ 
of~(4.6.1) extend similarly to natural maps
	\begin{equation} \label{map:general-reduced-ghost-map}
		S_n \times_S X \xright{\rgh{n}} \wnus{n}(X),
	\end{equation}
where $S_n=\Spec \sO_S/\m^{n+1}$.
Indeed, both sides commute with colimits in $X$, and so, since every $X\in\Space_S$
is the colimit of the objects of $\affrel_S$ mapping to it, the maps in the affine
case naturally induce maps in general.

Let $\rcgh{n}$ denote the \emph{reduced co-ghost map}
	\begin{equation} \label{map:general-reduced-co-ghost-map}
		\rcgh{n}\:S_n \times_S \wnls{n}(X) \xright{\rgh{n}}
		 	\wnus{n}\wnls{n}(X) \xright{\varepsilon} X,
	\end{equation}
where $\varepsilon$ is the counit of the evident adjunction.

Finally, we have natural plethysm and co-plethysm maps
	\begin{align} 
			\label{map:geometric-plethysm}
		\wnus{m}\wnus{n}(X) &\xright{\mu_X} \wnus{m+n}(X),\\
			\label{map:geometric-coplethysm}
		\wnls{m+n}(X) &\longmap \wnls{n}\wnls{m}(X),
	\end{align}
which are induced by~(2.4.5).

\begin{proposition}\label{pro:W_n-of-affine-is-affine}
	If $X$ is affine, then so is $\wnus{n}(X)$.
\end{proposition}
\begin{proof}
	First observe that when $S$ is affine, this was established already 
	in~(\ref{eq:wnus-of-affine-is-W}).
	
	For general $S$, we will apply Chevalley's theorem (in the final form, due to 
	Rydh~\cite{Rydh:Noetherian-approximation},
	Theorem (8.1)), to the ghost map $\gh{\leq n}$ of~(\ref{eq:total-ghost-map-geometric}).
	To apply it, it is enough to verify that $\coprod_{[0,n]}X$ is affine,
	that $\wnus{n}(X)$ is a scheme, and that $\gh{\leq n}$ is integral and surjective.
	
	The first statement is clear. Let us check the second.
	Let $(S_i)_{i\in I}$ be an open affine cover of $S$.
	Since $\wnus{n}(X)$ is the quotient of $\coprod_i S_i\times_S \wnus{n}(X)$ by the equivalence
	relation $\coprod_{j,k}S_j\times_S S_k\times_S \wnus{n}(X)$, it is enough to show that
	each of the summands in each expression
	is affine. And since $S$ is separated, it is enough to show the single statement that, 
	for any affine open subscheme $S'$ of $S$, the space $S'\times_S \wnus{n}(X)$ is affine.
	By~(\ref{eq:localization-of-S-commutes-with-wnus}), we have
	$S'\times_S\wnus{n}(X)=\wnus{S',n}(S'\times_S X)$.
	Further, because $X$ and $S'$ are affine and $S$ is separated, $S'\times_S X$ is affine. 
	Therefore $\wnus{S',n}(S'\times_S X)$ is affine, by the case mentioned in the beginning,
	and hence so is $S'\times_S\wnus{n}(X)$.
	
	Now let us check that $\gh{\leq n}$ is integral and surjective.
	It is enough to show this for each base-change 
	map $S_i\times_S \gh{\leq n}$. By~(\ref{eq:localization-of-S-commutes-with-wnus}) again,
	this map can be identified with the ghost map
	$\coprod_{[0,n]}(S_i\times_S X) \to \wnus{S_i,n}(S_i\times_S X)$.
	In other words, we may assume $S$ is affine, in which case we can conclude by 
	applying~\ref{lem:gh-integral-and-surj-multi-prime} below.
\end{proof}

\begin{lemma}\label{lem:gh-integral-and-surj-multi-prime}
	Suppose $S$ is affine, $S=\Spec R$. For any $R$-algebra $A$,
	the map $\gh{\leq n}\:W_{R,n}(A)\to A^{[0,n]}$ is integral, and its kernel $I$
	satisfies $I^{2^N}=0$, where $N=\sum_{\m}n_{\m}$.
\end{lemma}
\begin{proof}
	We may assume that 
	$\ptst$ equals the support of $n$, which is finite, and then reason by induction
	on the cardinality of $\ptst$. 
	When $\ptst$ is empty, $\gh{\leq n}$ is an isomorphism.
	Now suppose $\ptst$ contains an element $\m$.  Write $\ptst'=\ptst-\{\m\}$
	and let $n'$ denote the restriction of $n$ to $\ptst'$.
	Then the map $\gh{\leq n}$ factors as follows:
		\[
		\xymatrix{
		W_{\ptst,n}(A) \ar^{\sim}_{(5.4.2)}[d] 
				\ar^{\gh{\leq n}}[rr]
			& & \bigl(A^{[0,n_{\m}]}\bigr)^{[0,n']}\\
		W_{\m,n_{\m}}(W_{\ptst',n'}(A)) \ar^-{\gh{\leq n_{\m}}}[r] 
			& W_{\ptst',n'}(A)^{[0,n_{\m}]} \ar^{\sim}[r]
			& W_{\ptst',n'}(A^{[0,n_{\m}]}). \ar^{\gh{\leq n'}}[u] \\
		}
		\]
	So it is enough to show $\gh{\leq n_{\m}}$ and $\gh{\leq n'}$ are integral and 
	their kernels are nilpotent of the appropriate degree.
	For $\gh{\leq n'}$, it follows by induction on $\ptst$. For $a$, it follows by induction on 
	the integer $n_{\m}$, 
	using~8.1(a)--(b).
\end{proof}

\subsection{} {\em Algebraic spaces.}
\label{subsec:algebraic-spaces}
We will define the category of algebraic spaces to be the smallest full subcategory of
$\Space_{\bZ}$ which contains $\aff_{\bZ}$ and which is closed under arbitrary (set indexed)
disjoint unions and quotients by \'etale equivalence relations. This obviously exists. It also
agrees with the category of algebraic spaces as defined in
To\"en--Vaqui\'e~\cite{Toen-Vaquie:algebrisation}, section 2. This follows from
To\"en--Vezzosi~\cite{Toen-Vezzosi:HAGII}, Corollary 1.3.3.5, as does the fact that this category
is closed under finite limits. (Note that~\cite{Toen-Vezzosi:HAGII} is written in the homotopical
language of higher stacks, but it is possible to translate the arguments by substituting the word
\emph{space} for \emph{stack} and so on.) Indeed, as mentioned in their Remark 2.6, their category 
has the defining closure property of ours.

It will be convenient to use their concept of an algebraic space $X$ being $m$-geometric,
$m\in\bZ$. For $m\leq -1$, the space $X$ is $m$-geometric if and only if it is affine, it is
$0$-geometric if and only if its diagonal map is affine, and every algebraic space is $m$-geometric
for $m\geq 1$. (Again, see~\cite{Toen-Vaquie:algebrisation}, Remark 2.6. Note that they require
$m\geq -1$, but it will be convenient for us to allow $m<-1$.) In particular, for $m\geq 0$, every
$m$-geometric algebraic space is the quotient of a disjoint union of affine schemes by an \'etale
equivalence relation which is a disjoint union of $(m-1)$-geometric algebraic spaces. (Note however
that the converse is not true: over a field of characteristic $0$, the quotient group
$\mathbold{G}_{\mathrm{a}}/\bZ$ is $1$-geometric but not $0$-geometric.) Let $\AlgSp_m$ denote the
full subcategory of $\Space_{\bZ}$ consisting of all disjoint unions of $m$-geometric algebraic
spaces. A map $X\to Y$ of spaces is said to be $m$-representable if for every affine scheme $T$ and
every map $T\to Y$, the pull back $X\times_Y T$ is an $m$-geometric algebraic space.

This definition of algebraic space also agrees with that of Raynaud--Gruson~\cite{Raynaud-Gruson}.
Indeed, the category of Raynaud--Gruson algebraic spaces contains affine schemes and is closed
under disjoint unions and quotients by \'etale equivalence relations. (See Conrad--Lieblich--Olsson
\cite{Conrad-Lieblich-Olsson}, A.1.2.) And conversely, any algebraic space in the sense of
Raynaud--Gruson is the quotient of a disjoint union of affine schemes by an \'etale equivalence
relation which is a scheme, necessarily separated; therefore it is $1$-geometric.

The difference between these two approaches is that Raynaud--Gruson~\cite{Raynaud-Gruson} use
schemes as the intermediate class of algebraic spaces, and To\"en--Vaqui\'e use algebraic spaces
with affine diagonal. The advantage of the second approach is that the two steps---going
from affine schemes to the intermediate category, and going from that to general
algebraic spaces---are two instances of a single procedure. Thus we can prove results by induction
on the geometricity $m$, and so we do not need to consider the intermediate case separately. Note
however that the induction is rather meager in that it terminates after two steps.

Finally, a space is algebraic in the sense of Knutson~\cite{Knutson:Alg-spaces} if and only if it
is quasi-separated and algebraic in the sense above.

\subsection{} {\em Smoothness properties of maps.}
\label{subsec:smoothness-properties-of-maps}
Let us say a map $f\:X\to Y$ in $\Space$ is formally \'etale (resp.\ formally unramified, resp.\
formally smooth) if the usual nilpotent lifting properties (EGA IV 17.1.2 (iii)~\cite{EGA-no.32})
hold: for all nilpotent closed immersions $\bar{T}\to T$ of affine schemes, the induced map
	\begin{equation} \label{map:def-formally-etale}
		X(T) \longmap X(\bar{T})\times_{Y(\bar{T})} Y(T)
	\end{equation}
is a bijection (resp.\ injection, resp.\ surjection).

Also, let us say that $f$ is locally of finite presentation if
for any cofiltered system $(T_i)_{i\in I}$ of $Y$-schemes, each of which is affine, the induced map
	\begin{equation} \label{map:def-loc-of-finite-pres}
		\colim_i \Hom_Y(T_i,X) \longmap \Hom_Y(\lim_i T_i, X)
	\end{equation}
is a bijection.
This definition agrees with the usual one if $X$ and $Y$ are schemes
(EGA IV 8.14.2.c~\cite{EGA-no.28}).

We call a map \'etale (resp.\ unramified, resp.\ smooth) if it is locally of finite presentation
and formally \'etale (resp.\ formally unramified, resp.\ formally smooth).
When $X$ and $Y$ are schemes, all these definitions agree with the usual ones.

\subsection{} \emph{$\ptst$-flat and $\ptst$-smooth algebraic spaces.}
\label{subsec:definition-of-E-flat}
Let us say that an algebraic space $X$ over
$S$ is $\ptst$-flat (resp.\ $\ptst$-smooth) if,
for all maximal ideals $\m\in\ptst$, the algebraic space
$X\times_S\Spec\sO_{S,\m}$ is flat (resp.\ smooth)
over $\Spec\sO_{S,\m}$, where $\sO_{S,\m}$ is the local ring
of $S$ at $\m$. If $\ptst=\{\m\}$, we will often write $\m$-flat instead of $\ptst$-flat.

\section{Sheaf-theoretic properties of $\wnls{n}$}
\label{sec:sheaf-theoretic-properties-of-W-lower-star}

We continue with the notation of~\ref{subsec:supramaximal-def-global-base-S}.

\begin{proposition}\label{pro:wnls-preserves-formally-etale-etc}
	The functor $\wnls{n}\:\Space_S\to\Space_S$ preserves the following properties of maps:
		\begin{enumerate}
			\item locally of finite presentation,
			\item formally \'etale, formally unramified, formally smooth,
			\item \'etale, unramified, smooth.
		\end{enumerate}
\end{proposition}

\begin{proof}
	(a): 
	Let  $(T_i)_{i\in I}$ be a cofiltered system in $\aff_S$ mapping to $\wnls{n}(Y)$,
	as in~(\ref{map:def-loc-of-finite-pres}).
	The following chain of equalities, which we will justify below, constitutes the proof:
		\begin{align*}
			\Hom_{\wnls{n}(Y)}(\limm_i T_i, \wnls{n}(X)) &= \Hom_Y(\wnus{n}(\limm_i T_i),X) \\
				&\labeleq{\text{1}} \Hom_Y(\limm_i \wnus{n}(T_i),X) \\
				&\labeleq{\text{2}} \colimm_i \Hom_Y(\wnus{n}(T_i),X) \\
				&= \colimm_i \Hom_{\wnls{n}(Y)} (T_i, \wnls{n}(X)).
		\end{align*}
	Equality 2 holds because each $\wnus{n}(T_i)$ is affine (\ref{pro:W_n-of-affine-is-affine})
	and because $f\:X\to Y$ is locally of finite presentation.	
		
	To show equality 1, it is enough to show 
		\begin{equation} \label{eq:W-commutes-with-affine-limits}
			\wnus{n}(\lim_i T_i)=\lim_i \wnus{n}(T_i).
		\end{equation}
	Let $(S_j)_{j\in J}$ be an affine open cover of $S$. 
	Then it is enough to show~(\ref{eq:W-commutes-with-affine-limits}) after applying each
	functor $S_j\times_S\vbl$.
	We can then reduce to the case $S=S_j$, by 
	using~(\ref{eq:localization-of-S-commutes-with-wnus}), 
	the fact that $S_j\times_S\vbl$ commutes with limits,
	and the fact that each $S_j\times_S T_i$ is affine ($S$ being separated).
	In other words, we may assume $S$ is affine,
	in which case~(\ref{eq:W-commutes-with-affine-limits}) follows 
	from~6.10 and~(\ref{eq:wnus-of-affine-is-W}).
	
	(b): By definition,	the map $\wnls{n}(X) \to \wnls{n}(Y)$ is formally \'etale 
	(resp.\ formally unramified, resp.\ formally smooth) if 
	for any closed immersion $\bar{T}\to T$ of affine schemes defined by
	a nilpotent ideal sheaf, the induced map
	\[
		\wnls{n}(X)(T) \longmap
		 \wnls{n}(X)(\bar{T})\times_{\wnls{n}(Y)(\bar{T})} \wnls{n}(Y)(T).
	\]
	is bijective (resp.\ injective, resp.\ surjective).
	But this map is the same, by adjunction, as the map
	\[
		X(\wnus{n}(T)) \longmap X(\wnus{n}(\bar{T}))\times_{Y(\wnus{n}(\bar{T}))}Y(\wnus{n}(T)).
	\]
	Because $f\:X\to Y$ is formally \'etale (resp. \dots), to show this map is bijective 
	(resp. \dots), it is enough to check that the induced map $W_n(\bar{T})\to W_n(T)$ is also a 
	nilpotent immersion 
	of affine schemes. Affineness follows from~\ref{pro:W_n-of-affine-is-affine}.
	On the other hand, to show nilpotence, we may work locally. So 
	by~(\ref{eq:localization-of-S-commutes-with-wnus}),
	we may assume $S$ and $T$ are affine (since $S$ is separated). 
	We can then apply~6.4. 

	(c): This follows from (a) and (b) by definition.
\end{proof}

\begin{proposition}\label{pro:wnls-mod-p}
	Suppose $\ptst$ consists of one ideal $\m$, and set $S_0=\Spec\sO_S/\m$.
 	Let $f\:X\to Y$ be a map in $\Space_{S}$ which
	is formally \'etale (resp.\ formally unramified, resp.\ formally smooth).
	Then the map
	\begin{equation} \label{eq:wnls--mod-p}
		\xymatrix@C=75pt{
		S_0\times_S \wnls{n}(X) \ar^-{\id_{S_0}\times(\wnls{n}(f),\cgh{0})}[r] 
			& S_0 \times_S \wnls{n}(Y) \times_{\cgh{0},Y} X		
		}
	\end{equation}
	is an isomorphism (resp.\ monomorphism, resp.\ presheaf epimorphism).
\end{proposition}

\begin{proof}
	Let $Z$ be an affine $S_0$-scheme. Then $Z$ is an object of $\affrel_S$.
	We have the following commutative diagram:
	\[
	\xymatrix{
		\Hom_S(Z,\wnls{n}(X)) \ar^{\sim}_{c}[d]\ar^-a[r]
			& \Hom_S(Z, \wnls{n}(Y)\times_Y X) \ar^{\sim}_{d}[d] \\
		\Hom_S(\wnus{n}(Z),X) \ar^-b[r]
			& \Hom_S(\wnus{n}(Z), Y) \times_{\Hom_S(Z, Y)} \Hom_S(Z,X), \\
	}
	\]
	where $a$ is induced by~(\ref{eq:wnls--mod-p}), and $c$ and $d$ are given
	by the evident universal properties.
	The map $\gh{0}\:Z\to W_n(Z)$ is a closed immersion defined by a nilpotent ideal,
	by~6.8 and~(\ref{eq:wnus-of-affine-is-W}).
	Therefore, since $f$ is formally \'etale (resp.\ formally unramified, resp.\ formally smooth),
	$b$ is bijective (resp.\ injective, resp.\ surjective), and hence so is $a$.
	In other words, the map in question is an isomorphism (resp.\ monomorphism, resp.\
	presheaf epimorphism).
\end{proof}

\begin{corollary}\label{cor:naive-gluing-ok-in-char-p-2}
	Let $\ptst$ and $S_0$ be as in~\ref{pro:wnls-mod-p}.
	Let $(U_i\to X)_{i\in I}$ be an epimorphic family of \'etale maps in $\Space_S$.
	Then the induced family 
		\[
		\Bigl(S_0\times_S\wnls{n}(U_i)\longmap S_0\times_S\wnls{n}(X)\Bigr)_{i\in I}
		\] 
	is an epimorphic family of \'etale maps.
	
\end{corollary}
\begin{proof}
	By~\ref{pro:wnls-preserves-formally-etale-etc}, we only need to show
	the family is epimorphic. By~\ref{pro:wnls-mod-p}, each square
	\[
	\xymatrix{
	S_0\times_S \wnls{n}(U_i)\ar[r]\ar^{\cgh{0}\circ\pr_2}[d]
		& S_0\times_S \wnls{n}(X) \ar^{\cgh{0}\circ\pr_2}[d] \\
	U_i \ar[r]
		& X
	}
	\] 
	is cartesian. The lower arrows are assumed to form a covering family indexed by $i\in I$, and
	hence so do the upper arrows.
\end{proof}

\begin{proposition}\label{pro:wnls-preserves-epis-and-monos}
	The functor $\wnls{n}\:\Space_S\to\Space_S$ preserves epimorphisms
	and monomorphisms.
\end{proposition}

\begin{proof}
	The statement about monomorphisms follows for general reasons from the
	fact that $\wnls{n}$ has a left adjoint.
	
	Let us now consider the statement about epimorphisms. By~(\ref{eq:geom-witt-iterate-lower-s}),
	it suffices to assume $\ptst$ consists of one maximal ideal $\m$.
	Let $f\:X\to Y$ be an epimorphism in $\Space_S$.

	First consider the case where $X$ and $Y$ lie in $\affrel_S$ and $f$ is \'etale.
	The map $\wnls{n}(f)\:\wnls{n}(X)\to\wnls{n}(Y)$ is also an \'etale map between objects
	of $\affrel_S$, by~\ref{pro:wnls-preserves-formally-etale-etc} 
	and~(\ref{eq:wnls-of-affine-is-lambda-circle}).
	Being epimorphic is then equivalent to being surjective, which can be checked after 
	base change to $s=\Spec\sO_S/\m$ and to $S-s$.	For $S-s$, the base change
	of $\wnls{n}(X)\to\wnls{n}(Y)$ agrees, by~(\ref{eq:global-base-jet-localization}), with that of 
	$X^{[0,n]}\to Y^{[0,n]}$, which is surjective.
	On the other hand, over $s$ the map $\wnls{n}(f)\:\wnls{n}(X)\to\wnls{n}(Y)$ is
	a base change of the map $f\:X\to Y$, by~\ref{pro:wnls-mod-p}, which is also surjective.

	Now consider the general case, where $f\:X\to Y$ is any epimorphism of spaces.
	It is enough to show that for any object $V\in\affrel_S$,
	any given map $a\:V\to \wnls{n}(Y)$ lifts locally on $V$ to $\wnls{n}(X)$.
	Since $f$ is an epimorphism and
	$\wnus{n}(V)$ is affine, and hence quasi-compact, there exists a commutative
	diagram
		\begin{equation}\label{eq:internal-diagram-3}
		\xymatrix@C=40pt{
		X \ar^f[d]			& Z \ar_b[l]\ar@{->>}^c[d] \\
		Y 					& \wnus{n}(V), \ar_a[l]
		}			
		\end{equation}
	where $Z\in\affrel_S$, $c$ is an \'etale cover, and $b$ is some map.
 	Consider the commutative diagram
		\begin{equation}
		\xymatrix@C=40pt{
		\wnls{n}(X) \ar_{\wnls{n}(f)}[d]			
			& \wnls{n}(Z) \ar_-{\wnls{n}(b)}[l]\ar@{->}^{\wnls{n}(c)}[d] 
			& \wnls{n}(Z)\times_{\wnls{n}\wnus{n}(V)} V \ar_-{\pr_1}[l]\ar^{\pr_2}[d] \\
		\wnls{n}(Y) 					
			& \wnls{n}\wnus{n}(V) \ar_-{\wnls{n}(a)}[l]
			& V, \ar_-{\eta}[l]
		}						
		\end{equation}
	where $\eta$ is the unit of the evident adjunction.
	By the argument above, $\wnls{n}(c)$ is an \'etale epimorphism
	between objects of $\affrel_S$.
	Therefore $\pr_2$ is the same, and hence the map $V\to\wnls{n}(Y)$ lifts
	locally to $\wnls{n}(X)$.  
\end{proof}

\begin{corollary}\label{cor:W-lower-s-preserves-equ-relations}
	Let $(\pi_1,\pi_2)\:\Gamma\to X\times_S X$ 
	be an equivalence relation on a space $X\in\Space_{S}$.  
	Then the map
		\[
		(\wnls{n}(\pi_1),\wnls{n}(\pi_2))\:\wnls{n}(\Gamma)\to\wnls{n}(X)\times\wnls{n}(X)
		\] 
	is an equivalence relation on $\wnls{n}(X)$, and
	the induced map
		\[
		\wnls{n}(X)/\wnls{n}(\Gamma) \longmap \wnls{n}(X/\Gamma)
		\]
	is an isomorphism.
\end{corollary}
\begin{proof}
	By~\ref{pro:wnls-preserves-epis-and-monos}, 
	the map $\wnls{n}(X)\to\wnls{n}(X/\Gamma)$ is an epimorphism
	of spaces.  On the other hand, since $\wnls{n}$ has a left adjoint, we have
		\[
		\wnls{n}(X)\times_{\wnls{n}(X/\Gamma)} \wnls{n}(\Gamma) = \wnls{n}(X\times_{X/\Gamma} X)
			= \wnls{n}(\Gamma),
		\]
	and so the equivalence relation inducing the quotient $\wnls{n}(X/\Gamma)$
	is $\wnls{n}(\Gamma)$.
\end{proof}

\subsection{} {\em Remark.} 
\label{subsec:wnls-chart}
These results allow us to present $\wnls{n}(X)$ using charts, but not in the sense that might first
come to mind. For while $\wnls{n}$ preserves covering maps
(by \ref{pro:wnls-preserves-epis-and-monos}), it does not generally preserve covering
families. That is, if $(U_i)_{i\in I}$ is an \'etale covering family of $X$, 
then the space $\wnls{n}(\coprod_i U_i)$ covers $\wnls{n}(X)$, but it is usually
not true that $\coprod_i \wnls{n}(U_i)$ covers it. For example, consider the $p$-typical case with
$n=1$. On the generic fiber, $\wnls{1}(X)$ is just $X\times X$; and of course $\coprod_i U_i\times
U_i$ does not generally cover $X\times X$. In particular, $\wnls{n}(X)$ cannot be constructed using
charts by gluing the spaces $\wnls{n}(U_i)$ together along the overlaps $\wnls{n}(U_{j}\times_X
U_k)$. This is just a general property of products and not a particular property of Frobenius lifts.

On the other hand, if $X$ is an algebraic space over $\Spec \sO_S/\m^j$, for some integer $j\geq
0$, this naive gluing method does work. This is because for any \'etale cover $(U_i)_{i\in I}$ of
$X$, the family $\bigl(\wnls{n}(U_i)\bigr)_{i\in I}$ is an \'etale cover of $\wnls{n}(X)$. This is
true by \ref{pro:wnls-preserves-formally-etale-etc} and \ref{cor:naive-gluing-ok-in-char-p-2}.
See~\ref{subsec:relation-to-Greenberg-and-Buium} for the implications this has for Buium's $p$-jet
spaces.

\begin{proposition}\label{pro:wnls--commutes-with-filtered-colimits}
	The functor $\wnls{n}\:\Space_S\to\Space_S$ commutes with filtered colimits.
\end{proposition}
\begin{proof}
	By adjunction,
	this is equivalent to the statement that for any filtered system $(X_i)_{i\in I}$
	and any $T\in\affrel_S$ the map
		\[
		\colimm_i \Hom\bigl(\wnus{n}(T),X_i\bigr) \longmap \Hom\bigl(\wnus{n}(T),\colimm_i X_i\bigr)
		\]
	is an isomorphism.  Because $\wnus{n}(T)$ is affine (\ref{eq:wnus-of-affine-is-W}),
	it is quasi-compact and quasi-separated,
	and so the proposition follows from SGA 4 VI 1.23(ii)~\cite{SGA4.2}.
\end{proof}

\begin{lemma}\label{lem:wnls-good-cover}
	Let $(S_t)_{t\in T}$ be an open cover of $S$,
	let $X$ be an object of $\Space_{S}$, and let $(U_i)_{i\in I}$ be a cover of $X$
	by objects of $\Space_S$.
	Let $J$ denote the set of finite subsets of $I$, and
	for each $j\in J$, write $V_j=\coprod_{i\in j} U_i$.
	Then $(S_t\times_S V_j)_{(t,j)\in T\times J}$ is a cover of $X$ with the property
	that $\bigl(\wnls{n}(S_t\times_S V_j)\bigr)_{(t,j)\in T\times J}$ is a cover of $\wnls{n}(X)$.
\end{lemma}
\begin{proof}
	It is clear that $(S_t\times_S V_j)_{t,j}$ is a cover of $X$. Let us show that
	$\bigl(\wnls{n}(S_t\times_S V_j)\bigr)_{t,j}$ is a cover of $\wnls{n}(X)$.
	By~(\ref{eq:global-base-jet-localization}) and (\ref{eq:wnls-is-independent-of-S}),
	it is enough to consider the case where $(S_t)_{t\in T}$ is the trivial cover consisting
	of $S$ itself. Thus it is enough to	show that $\bigl(\wnls{n}(V_j)\bigr)_{j}$ is a cover of 
	$\wnls{n}(X)$.

	Observe that we have a natural isomorphism
		\begin{equation} \label{eq:filterization-colimit}
			\coprod_{i\in I} U_i = \colimm_{j\in J} V_j;
		\end{equation}
	indeed, both sides have the same universal property.	
	Now consider the commutative diagram
			\[
			\entrymodifiers={+!!<0pt,\fontdimen22\textfont2>}
			\def\objectstyle{\displaystyle}
			\xymatrix@C=115pt@R=60pt@!0{
				\coprod_{j\in J} \wnls{n} (V_j) \ar^-a[rr]\ar@{->>}[d] 	
					&
					& \wnls{n} (X) \\
				\colimm_{j\in J}\wnls{n} (V_j) 
	\ar^-{\sim}_-{\text{(\ref{pro:wnls--commutes-with-filtered-colimits})}}[r]
					& \wnls{n}\bigl(\colimm_{j\in J}V_j\bigr)
						\ar^-{\sim}_-{\text{(\ref{eq:filterization-colimit})}}[r]  
					& \wnls{n} \big(\coprod_{i\in I} U_i\big),  \ar@{->}_b[u]
			}
			\]
	where $a$ and $b$ are the maps induced by the covering maps $V_j\to X$ and $U_i\to X$.
	Then $a$ is an epimorphism because $b$ is,
	which is true by~\ref{pro:wnls-preserves-epis-and-monos}. 
\end{proof}

\subsection{} \emph{$\wnls{n}$-stable covers.} 
\label{rem:wnls-good-cover}
It is useful to have covers $(X_k)_{k\in K}$ of $X$ that are $\wnls{n}$-stable, meaning that
$(\wnls{n}(X_k))_{k\in K}$ is a cover of $\wnls{n}(X)$. While general covers are not
$\wnls{n}$-stable (see \ref{subsec:wnls-chart}), some are. Any singleton cover is,
by~\ref{pro:wnls-preserves-epis-and-monos}, but it is not always enough to have this because there
can fail to be singleton covers with desirable properties. For instance, if $X$ is not
quasi-compact, it cannot be covered by a single affine scheme. But it often suffices to know only
that $\wnls{n}$-stable covers with certain desirable properties exist, and we can sometimes use
\ref{lem:wnls-good-cover} to make them. For instance, we can produce a $\wnls{n}$-stable cover with
each $X_k$ affine by taking $K=T\times J$ and $X_{(t,j)}=S_t\times_S V_j$ in
\ref{lem:wnls-good-cover}, where the $U_i$ and $S_t$ are affine. If we refine the cover
$(S_t)_{t\in T}$ so that each ideal in the support of $n$ is principal on each $S_t$, then we
further have that the image of each $X_k$ in $S$ is contained in an affine open subscheme of $S$ on
which each ideal in the support of $n$ is principal. If $X$ is an algebraic space, we can even
further arrange for each $X_k$ to be \'etale over $X$ by taking $(U_i)_{i\in I}$ to be an \'etale
cover of $X$.

\begin{proposition}\label{pro:wnls-preserves-qcom-qsep}
	The functor $\wnls{n}\:\Space_S\to\Space_S$ preserves
   \begin{enumerate}
	   	\item quasi-compactness of objects,
	   	\item quasi-separatedness of objects,
	   	\item quasi-compactness of maps,
		\item quasi-separatedness of maps.
   \end{enumerate}
\end{proposition}

\begin{proof}
	Let $X$ be an object in $\Space_S$.
	
	(a): Suppose $X$ is quasi-compact. Let
	$(U_i)_{i\in I}$ be a finite family in $\affrel_S$ which covers $X$. 
	(Such a family exists, because the large family of all morphisms from objects
	of $\affrel_S$ to $X$ covers
	to $X$, and therefore has a finite subcover, because $X$ is quasi-compact.)
	Then the space $U=\coprod_{i\in I}U_i$ is affine.
   	By~\ref{pro:wnls-preserves-epis-and-monos}, 
	the map $\wnls{n} (U)\to\wnls{n} (X)$ is an epimorphism.
	Since $\wnls{n} (U)$ is affine, it is quasi-compact.
	It follows that $\wnls{n} (X)$ is quasi-compact (SGA 4 VI 1.3\cite{SGA4.2}).
   
	(b): Suppose $X$ is quasi-separated. Then for any cover $(U_i)_{i\in I}$ of $X$, with each
	$U_i\in\affrel_S$, each space $U_i\times_X U_j$ is 
	quasi-compact. Therefore, by (a), each space
	$\wnls{n} (U_i) \times_{\wnls{n} (X)} \wnls{n} (U_j) = \wnls{n}(U_i\times_X U_j)$ is 
	quasi-compact.  By SGA 4 VI 1.17~\cite{SGA4.2}, this implies that $\wnls{n}(X)$ is 
	quasi-separated
	as long as we can choose the cover $(U_i)_{i\in I}$ such that $(\wnls{n}(U_i))_{i\in I}$
	is a cover of $\wnls{n}(X)$. This is possible by~\ref{rem:wnls-good-cover}.

	(c): Let $f\:X\to Y$ be a quasi-compact map of spaces.
	As above, by~\ref{rem:wnls-good-cover}, there exists a cover $(U_i)_{i\in I}$ of $Y$, with each 
	$U_i\in\affrel_S$, such that $(\wnls{n} (U_i))_{i\in I}$ is an affine cover of $\wnls{n} (Y)$. 
	It is then enough
	to show that each $\wnls{n} (U_i) \times_{\wnls{n} (Y)} \wnls{n} (X)$ is quasi-compact
	(SGA 4 VI 1.16~\cite{SGA4.2}), but this agrees with $\wnls{n}(U_i\times_Y X)$.  Now apply (a).

	(d): Let $f\:X\to Y$ be a quasi-separated map of spaces. By definition, its diagonal map
	$\Delta_f$ is quasi-compact. By (c), so is the map $\wnls{n}(\Delta_f)$, and this
	agrees with the diagonal map of $\wnls{n}(f)$.
\end{proof}

\section{$\wnls{n}$ and algebraic spaces}
\label{sec:wnls-and-algebraic-spaces}

We continue with the notation of~\ref{subsec:supramaximal-def-global-base-S}.

\begin{theorem}\label{thm:wnls--alg-space}
	Let $X$ be an algebraic space over $S$. Then $\wnls{n}(X)$ is an algebraic space.
	If $X$ is a scheme, then so is $\wnls{n}(X)$.
\end{theorem}

For the proof, see~\ref{subsec:proof-of-wnls-of-alg-space} below. Observe that when $X$ is
quasi-compact and has affine diagonal (e.g.\ is separated), as is often the case in applications,
the algebraicity of $\wnls{n}(X)$ follows immediately
from~\ref{cor:W-lower-s-preserves-equ-relations} and~\ref{pro:wnls-preserves-formally-etale-etc}.
Thus, for the part of the theorem asserting that $\wnls{n}(X)$ is an algebraic space, all the work
below is in removing these assumptions.

\begin{proposition}\label{pro:wnls-of-a-disjoint-union-of-pairs}
	For any spaces $X,Y\in\Space_S$, the diagram 
		\[
		\entrymodifiers={+!!<0pt,\fontdimen22\textfont2>}
		\def\objectstyle{\displaystyle}
		\xymatrix@C=100pt@R=45pt@!0{
		\wnls{n}(X) \ar^-{\wnls{n}(j)}[r]\ar^{\cgh{\leq n}}[d] 
			& \wnls{n}(X\smcoprod Y) \ar^{\cgh{\leq n}}[d] \\
		X^{[0,n]} \ar^-{j^{[0,n]}}[r] 
			& (X\smcoprod Y)^{[0,n]},
		}
		\]
	where $j\:X\to X\smcoprod Y$ denotes the canonical summand inclusion, is cartesian. 
\end{proposition}

\begin{proof}
	It is enough to show that for any object $T\in\affrel_S$, the functor $\Hom(T,\vbl)$
	takes the diagram above to a cartesian diagram.
	By adjunction, this is equivalent to the existence
	of a unique dashed arrow making the diagram 
		\[
		\entrymodifiers={+!!<0pt,\fontdimen22\textfont2>}
		\def\objectstyle{\displaystyle}
		\xymatrix@C=80pt@R=45pt@!0{
		\coprod_{[0,n]} T \ar^{\gh{\leq n}}[r]\ar[d] 
			& \wnus{n}(T) \ar[d]\ar_{\exists!?}@{-->}[dl] \\
		X \ar^j[r] & X \smcoprod Y
		}
		\] 
	commute, for any given vertical arrows making the square commute.
	It is therefore enough to show $\wnus{n}(T)\times_{X\smsmcoprod Y} Y=\emptyset$.
	
	To do this, we will show that if there exists a map
	$U\to\wnus{n}(T)\times_{X\smsmcoprod Y} Y$, where $U$ is an affine scheme, then $U$
	is empty. Pulling back such a map by $\gh{\leq n}$, we get a map
		\[
		\bigl(\coprod_{[0,n]}T\bigr)\times_{\wnus{n}(T)} U \longmap
		\bigl(\coprod_{[0,n]}T\bigr)\times_{\wnus{n}(T)}\wnus{n}(T)
			\times_{(X\smsmcoprod Y)} Y.
		\]
	By the commutativity of the square above, the right-hand side is empty.
	Therefore the left-hand side is empty. But since $\gh{\leq n}$ is a surjective
	map of (affine) schemes (by \ref{lem:gh-integral-and-surj-multi-prime} and 
	(\ref{eq:wnls-of-affine-is-lambda-circle})), $U$ must be empty.
\end{proof}

\subsection{} \emph{Remark.} It follows from~\ref{pro:wnls-of-a-disjoint-union-of-pairs}
that $\coprod_i\wnls{n}(X_i)$ is a summand of $\wnls{n}(\coprod_i X_i)$. In all but the most
trivial cases, the two will not be equal. See~\ref{subsec:wnls-chart}, for example.

\begin{lemma}\label{lem:wnls-of-componentwise-affine}
	If $X$ is a disjoint union in $\Space_S$ of objects in $\affrel_S$,
	then so is $\wnls{n}(X)$.
\end{lemma}
\begin{proof}
	Let $(X_i)_{i\in I}$ be a family of objects in $\affrel_S$ such 
	that $X\cong \coprod_{i\in I} X_i$.
	For any function $h\:[0,n]\to I$, write
		\[
		X^h = \prod_{m\in [0,n]} X_{h(m)}.
		\]
	Then we have 
		\[
		X^{[0,n]} = \coprod_{h} X^h,
		\]
	where $h$ runs over all maps $[0,n]\to I$. Therefore it is enough to show that the pre-image of 
	each $X^h$ under the map $\cgh{\leq n}\: \wnls{n}(X)\to X^{[0,n]}$
	is a disjoint union in $\Space_S$ of objects in $\affrel_S$.
	
	Since this preimage lies over $X^h$, and hence over $X_{h(0)}$, it lies over an
	affine open subscheme $S'$ of $S$. Therefore it is enough to show that this preimage is affine.
	
	Let us first do this when $I$ is finite. Because $X^h$ lies over $S'$, we have
		\begin{equation} \label{eq:5395}
		X^h \times_{X^{[0,n]}}\wnls{n}(X) = X^h\times_{X^{[0,n]}}\bigl(\wnls{n}(X)\times_S S'\bigr).
		\end{equation}
	Since $I$ is finite, $X$ is affine, and hence so is $X^h$.
	On the other hand, $\wnls{n}(X)\times_S S'$ is affine, 
	by~(\ref{eq:global-base-jet-localization}) and~(\ref{eq:wnls-of-affine-is-lambda-circle}).
	Therefore the left-hand side of (\ref{eq:5395}) is affine.

	Now suppose $I$ is arbitrary.
	Let $J$ denote the image of $h$, and write $Y=\coprod_{i\in J} X_i$.
	Then the map $X^h\to X^{[0,n]}$ factors through the map
	$j^{[0,n]}\:Y^{[0,n]}\to X^{[0,n]}$ induced by the summand inclusion 
	$j\:Y \to X$.
	Therefore by~\ref{pro:wnls-of-a-disjoint-union-of-pairs}, the right-hand
	square in the digram
		\[
		\entrymodifiers={+!!<0pt,\fontdimen22\textfont2>}
		\def\objectstyle{\displaystyle}
		\xymatrix@C=90pt@R=45pt@!0{
		X^h\times_{Y^{[0,n]}}\wnls{n}(Y) \ar^-{\pr_2}[r]\ar^{\pr_1}[d] 
			& \wnls{n}(Y) \ar^{\wnls{n}(j)}[r]\ar^{\cgh{\leq n}}[d] 
			& \wnls{n}(X) \ar^{\cgh{\leq n}}[d] \\
		X^h \ar[r]
			& Y^{[0,n]} \ar^{j^{[0,n]}}[r] 
			& X^{[0,n]}
		}
		\]
	is cartesian. Thus $X^h\times_{X^{[0,n]}} \wnls{n}(X)$ agrees with
	$X^h\times_{Y^{[0,n]}}\wnls{n}(Y)$, which is affine by the case proved above
	and the fact that $Y$ is affine.
\end{proof}

\subsection{} \emph{Proof of~\ref{thm:wnls--alg-space}.}
\label{subsec:proof-of-wnls-of-alg-space}
	It is enough to show that $S'\times_S\wnls{n}(X)$ is an algebraic space, or a scheme
	when $X$ is,
	for all sufficiently small affine open subschemes $S'$ of $S$. 
	Therefore by~(\ref{eq:global-base-jet-localization}), we 
	may assume that $S=\Spec R$ for some ring $R$,
	and that the ideal $\m$ of $R$ is generated by a single element $\pi$. 
	
	Let us first show that $\wnls{n}(X)$ is an algebraic space.
	We will show by induction on $m$ that if $X$ is $m$-geometric, then $\wnls{n}(X)$ is 
	an algebraic space. If $m=-1$, then $X$ is affine and so we can	
	apply~(\ref{eq:wnus-of-affine-is-W}). Now assume $m\geq 0$.
	
	Let $(U_i\to X)_{i\in I}$ be an affine \'etale cover for which
	each map $U_i\to X$ is $(m-1)$-representable. Write 
	$U=\coprod_{i\in I} U_i$. Consider the diagrams
		\[
		\displaycofork{U\times_X U}{U}{X.}
		\]	
	and
		\[
		\displaycofork{\wnls{n}(U\times_X U)}{\wnls{n}(U)}{\wnls{n}(X).}
		\]
	By~\ref{cor:W-lower-s-preserves-equ-relations} 
	and~\ref{pro:wnls-preserves-formally-etale-etc}, the space $\wnls{n}(U\times_X U)$
	is an \'etale equivalence relation on $\wnls{n}(U)$ with quotient $\wnls{n}(X)$.
	Since the category of algebraic spaces is closed under quotients by \'etale
	equivalence relations, it is sufficient to show that $\wnls{n}(U)$ and $\wnls{n}(U\times_X U)$
	are algebraic spaces. This holds because of two facts.
	First, $U$ (resp.\ $U\times_X U$) is a disjoint union of 
	$-1$-geometric (resp.\ $(m-1)$-geometric) algebraic spaces. Second,
	for $k\leq m-1$, the functor $\wnls{n}$ applied to a disjoint union of
	$k$-geometric algebraic spaces is an algebraic space.
	Indeed, if $k=-1$, this follows from~\ref{lem:wnls-of-componentwise-affine}; if $k\geq 0$, then
	a disjoint union of $k$-geometric algebraic spaces is itself
	a $k$-geometric algebraic space, and so it follows by induction.

	Now suppose $X$ is a scheme. To show that $\wnls{n}(X)$ is a scheme,
	we may assume, by (\ref{eq:geom-witt-iterate-lower-s}), that
	$\ptst$ consists of a single ideal $\m$.
	Since $\wnls{n}(X)$ is an algebraic space, it is enough to show it has an affine open cover.
	
	Let $(V_i)_{i\in I}$ be an affine open cover of $X$. 
	For any $S$-space $Y$, write $Y'=Y\times_S \Spec R[1/\pi]$. Then each $V'_i$ is an affine
	open subscheme of $X$. Further, the schemes $V'_{i_0}\times_S\cdots\times_S V'_{i_n}$
	cover the product $(X')^{[0,n]}=X'\times_S\dots\times_S X'$, which agrees with
	$\wnls{n}(X')$ and hence $\wnls{n}(X)'$ 
	(by~(\ref{eq:global-base-jet-localization})--(\ref{eq:wnls-is-independent-of-S})).
	So all that remains is to show the fiber of $\wnls{n}(X)$ over $\m$ can be covered
	by open subspaces that are affine schemes.

	Since each $V_i$ is affine, each $\wnls{n}(V_i)$ is 
	affine, by \ref{eq:wnls-of-affine-is-lambda-circle}. Since each
	map $V_i\to X$ is an \'etale monomorphism, each
	map $\wnls{n}(V_i)\to \wnls{n}(X)$ is an \'etale monomorphism, 
	by \ref{pro:wnls-preserves-formally-etale-etc} and
	\ref{pro:wnls-preserves-epis-and-monos}. Therefore these maps
	are open immersions, and by \ref{cor:naive-gluing-ok-in-char-p-2} they cover the fiber of 
	$\wnls{n}(X)$ over $\m$.
\qed

\begin{corollary}\label{cor:smooth-implies-wnls-flat}
	Let $X$ be an $\ptst$-smooth (\ref{subsec:definition-of-E-flat}) algebraic space over $S$.
	Then $\wnls{n}(X)$ is an $\ptst$-smooth algebraic space over $S$.  In particular, it is 
	$\ptst$-flat.
\end{corollary}
\begin{proof}
	By~\ref{thm:wnls--alg-space}, we know $\wnls{n}(X)$ is an algebraic space.
	Now let us show it is smooth locally at all maximal ideals of $\ptst$.
	By~(\ref{eq:global-base-jet-localization}), we may assume that $X$ is smooth over $S$.
	Then apply~\ref{pro:wnls-preserves-formally-etale-etc} and the
	fact that smoothness for a map of algebraic	spaces implies flatness.
\end{proof}

\subsection{} \emph{$\wnls{n}$ does not generally preserve flatness.}
\label{subsec:wnls-does-not-preserve-flatness}
For example, consider the $p$-typical jets of length $1$.
Let $A= \bZ[x]/(x^2-px)$, which is flat over $\bZ$ (and happens to be isomorphic to $W_1(\bZ)$).
Then
	\[
	\bZ[1/p]\tn_{\bZ}(\Lambda_1\bcp A) = \Lambda_1\bcp(\bZ[1/p]\tn A)
		= (\bZ[1/p]\tn A)^{\tn 2} = \bZ[1/p]^{2\times 2}.
	\]
But a short computation using 3.4 shows
	\[
	\Lambda_1\bcp A = \bZ[x,\delta]/(x^2-px,2x^p\delta+p\delta^2-x^p-p\delta+p^{p-1}x^p),
	\]
and hence
	\[
	\bF_p\tn_{\bZ}(\Lambda_1\bcp A) = \bF_p[x,\delta]/(x^2).
	\]
So $\Spec(\Lambda_1\bcp A)$ has one irreducible component lying over $\Spec\bF_p$.  In particular,
it is not flat locally at $p$.

It would be interesting to find a reasonable condition on $X$ that is
weaker than smoothness over $S$ but still implies flatness of $\wnls{n}(X)$ over $S$.

\subsection{} \emph{Relation to Greenberg's and Buium's spaces.} 
\label{subsec:relation-to-Greenberg-and-Buium}
In the case $S=\Spec\bZ_p$ and $\ptst=\{p\bZ_p\}$,
our arithmetic jet space is closely related to 
previously defined spaces, the Greenberg transform and Buium's $p$-jet space.

Let $X$ be a scheme locally of finite type over $\bZ/p^{n+1}\bZ$. Then the Greenberg transform
$\Gr_{n+1}(X)$ is a scheme over $\bF_p$. (See~\cite{Greenberg:I}\cite{Greenberg:II}, or for a
summary in modern language,~\cite{BLR:Neron-Models} p.\ 276.) It is related to
$\wnls{n}(X)$ by the formula
\begin{equation}
	\Gr_{n+1}(Y) = \wnls{n}(Y)\times_{\Spec\bZ_p}\Spec\bF_p.	
\end{equation}
This is simply because they represent, almost by definition, the same functor.

On the other hand, for smooth schemes $Y$ over the completion $\tilde{R}$ of the maximal
unramified extension of $\bZ_p$, Buium has defined \emph{$p$-jet spaces} $J^n(Y)$, which
are formal schemes over $\tilde{R}$. (See \cite{Buium:p-jets}, section 2, or
\cite{Buium:Arithmetic-diff-equ}, section 3.1.) His jet space is related to ours by the
formula
\begin{equation} \label{eq:general-buium-wnls-relation}
	J^n(Y) = \wnls{n}(\hat{Y}), 
\end{equation}
where $\hat{Y}$ denotes the colimit of the schemes $Y\times_{\Spec\bZ}\Spec\bZ/p^m\bZ$, taken over
$m$ and in the category $\Space_S$. Indeed it is true when $Y$ is affine,
by~3.4, and it holds when $Y$ is any smooth scheme over
$\tilde{R}$ by gluing. As mentioned in~\ref{subsec:wnls-chart}, the gluing methods used to define
$\wnls{n}$ and $J^n$ are not the same in general, but here $\hat{Y}$ is $p$-adically formal; so they
agree by the discussion in \ref{subsec:wnls-chart}.

The following consequence of~(\ref{eq:general-buium-wnls-relation}) is also worth
recording: if $X$ is a smooth scheme over $\bZ_p$, then we have
\begin{equation}
		J^n(X\times_{\Spec\bZ_p}\Spec\tilde{R}) 
			= \wnls{n}(\hat{X})\times_{\Spf\bZ_p}\Spf\tilde{R}.
\end{equation}

\section{Preservation of geometric properties by $\wnls{n}$}
\label{sec:preservation-of-properties-for-wnls}

We continue with the notation of~\ref{subsec:supramaximal-def-global-base-S}.

\subsection{} \emph{\'Etale-local properties.}
Recall that a property $P$ of algebraic spaces over $S$ is said to be \emph{\'etale-local} if the
following hold: whenever $X$ satisfies $P$, then so does any algebraic space $Y$ which admits an
\'etale map to $X$; and if $(U_i)_{i\in I}$ is an \'etale cover of $X$ such that each $U_i$
satisfies $P$, then so does $X$.

A property $P$ of maps of algebraic spaces is said to be \emph{\'etale-local on the
target} if for any map $f\:X\to Y$ the following hold: whenever $f$ satisfies $P$, then so does any
base change $f_V\:X\times_Y V\to V$ with $V\to Y$ \'etale; and if $(V_j)_{j\in J}$ is an \'etale
cover of $Y$ such that each base change $f_{V_j}$ satisfies $P$, then so does $f$.

Such a property is said to be \emph{\'etale-local on the source} if, in addition, the following
hold: whenever $f$ satisfies $P$, then so does any composition $U\to X\to Y$ with $U\to X$ \'etale;
and if $(U_i)_{i\in I}$ is an \'etale cover of $X$ such that each composition $U_i\to X\to Y$
satisfies $P$, then so does $f$.

\begin{proposition}\label{pro:maps-local-permanence-for-wnls}
	Let	$P$ be a property of maps $f\:X\to Y$ of algebraic spaces which is \'etale-local on the 
	target. For $\wnls{n}$ to preserve property $P$, it is sufficient that it do so
	when $\ptst$ consists of one principal ideal, $S$ is affine, and $Y$ is affine. 
	
	If property $P$ is also \'etale-local on the source, then 
	we may further restrict to the case where $X$ is affine.
\end{proposition}

The argument is the same as the one given below for \ref{pro:maps-local-permanence}, except that
one takes affine \'etale covering families of the kind given by~\ref{rem:wnls-good-cover},
and one uses the easy fact that $\wnls{n}$
preserves fiber products instead of the more
difficult~\ref{thm:W-main-geometric-finite-length(2)}(c). Since the details are given
in~\ref{pro:maps-local-permanence}, let us omit them here.

\begin{proposition}\label{pro:wnls-preserves-target-local-properties}
	The following properties of maps (\'etale-local on the target) are preserved by $\wnls{n}$:
		\begin{enumerate}
			\item \label{part:affine:wnls-target-local} affine,
			\item \label{part:cl-imm:wnls-target-local} a closed immersion,
			\item \label{part:lft:wnls-target-local} locally of finite type,
			\item \label{part:lfp:wnls-target-local} locally of finite presentation,
			\item \label{part:fin-type:wnls-target-local} of finite type,
			\item \label{part:fin-pres:wnls-target-local} of finite presentation,
			\item \label{part:separated:wnls-target-local} separated,
			\item \label{part:smooth-surj:wnls-target-local} smooth and surjective.
		\end{enumerate}
\end{proposition}
\begin{proof}
	Let $f\:X\to Y$ be such a map.
	By~\ref{pro:maps-local-permanence-for-wnls}, we may assume $S=\Spec R$, $Y=\Spec A$.
	
	(\ref{part:affine:wnls-target-local})--(\ref{part:cl-imm:wnls-target-local}): 
	These are affine properties; so we have $X=\Spec B$, for some $A$-algebra $B$.
	By (\ref{eq:wnls-of-affine-is-lambda-circle}), we have $\wnls{n}(X)=\Spec\Lambda_n\bcp B$, 
	which is affine. This proves (\ref{part:affine:wnls-target-local}).
	If the structure map $A\to B$ is surjective, then so is the induced map
	$\Lambda_n\bcp A\to\Lambda_n\bcp B$, which proves (\ref{part:cl-imm:wnls-target-local}).
	
	(\ref{part:lft:wnls-target-local})--(\ref{part:lfp:wnls-target-local}):
	These properties are  \'etale-local on the source, and so we may assume
	$X=\Spec B$, where $B$ is a finitely generated $A$-algebra.	
	Take an integer $m\geq 0$ such that there
	exists a surjection $A[x]^{\tn m}\to B$. Then the induced map
		\[
		(\Lambda_{n}\bcp A) \tn_R (\Lambda_{n}^{\tn m}) 
			= \Lambda_{n}\bcp (A\tn_R R[x_1,\dots,x_m]) \longmap \Lambda_{n} \bcp B
		\]
	is surjective. Therefore it is enough to show that $\Lambda_n$ is finitely generated
	as an $R$-algebra. This was proved in 6.10.
	
	Now suppose that $B$ is finitely presented.
	Then there exist finitely generated $A$-algebras $A'$ and $A''$ and a coequalizer diagram
		\[
		\displaycofork{A''}{A'}{B.}
		\]
	This then induces a coequalizer diagram
		\[
		\displaycofork{\Lambda_n \bcp A''}{\Lambda_n\bcp A'}{\Lambda_n\bcp B.}
		\]
	By (\ref{part:lft:wnls-target-local}), the first two terms are finitely generated $A$-algebras;
	therefore the last term is finitely presented. This proves (\ref{part:lfp:wnls-target-local}).
	
	(\ref{part:fin-type:wnls-target-local})--(\ref{part:fin-pres:wnls-target-local}):
	These follow from (\ref{part:lft:wnls-target-local})--(\ref{part:lfp:wnls-target-local})
	and \ref{pro:wnls-preserves-qcom-qsep}(c), by definition.
	
	(\ref{part:separated:wnls-target-local}): 
	Since the diagonal map $\Delta_f\:X\to X\times_Y X$ is a closed immersion,
	(\ref{part:cl-imm:wnls-target-local}) implies $\wnls{n}(\Delta_f)$ is a closed immersion.
	This map can be identified with the diagonal map
	$\wnls{n}(X)\to \wnls{n}(X)\times_{\wnls{n}(Y)}\wnls{n}(X)$, and so the result follows.
	
	(\ref{part:smooth-surj:wnls-target-local}): By~\ref{pro:wnls-preserves-formally-etale-etc},
	the map $\wnls{n}(f)$ is 
	smooth. By~\ref{pro:wnls-mod-p}, it is surjective over $S_0=\Spec \sO/\m$, and
	by (\ref{eq:wnus-is-independent-of-S}), it can be 
	identified away from $S_0$ with the product map $X^{[0,n]}\to Y^{[0,n]}$,
	which is surjective. Therefore $\wnls{n}(f)$ is, too.
\end{proof}

\subsection{} \emph{Counterexamples.}
\label{subsec:counterexamples-for-preservation-by-wnls}
Consider the $p$-typical case: $R=\bZ$, $\ptst=\{p\bZ\}$, where $p$ is a prime number.
Then a short computation using 3.4 shows
	\[
	\Lambda_1\bcp(\bZ\times\bZ) \cong \bZ \times \bZ \times \bZ[1/p] \times \bZ[1/p].
	\]
So $\wnls{n}$ does not generally preserve any property which the map $\bZ\to\bZ\times\bZ$
has and which is at least as strong as integrality: integral, 
finite, finite flat, finite \'etale,\dots.

Also, $\wnls{n}$ does not generally preserve surjectivity (of schemes), because we have
	\begin{align*}
		\Lambda_1\bcp \bZ_p[x]/(x^2-p) &= \bZ_p[x,\delta]/(x^2-p, p\delta^2+2x^p\delta+p^{p-1}-1)\\
			&\cong \bQ_p(\sqrt{p})\times\bQ_p(\sqrt{p})
	\end{align*}
(by 3.4 again) but also $\Lambda_{1}\bcp \bZ_p=\bZ_p$.

Finally, as shown in \ref{subsec:wnls-does-not-preserve-flatness}, the map $\bZ\to\bZ[x]/(x^2-px)$
becomes non-flat after the application of $\Lambda_1\bcp\vbl$. So $\wnls{n}$ does not
generally preserve any property which the map $\bZ\to\bZ[x]/(x^2-px)$ has and which
is at least as strong as flatness: flat, faithfully flat, Cohen--Macaulay, $\text{S}_k$,\dots.

\section{The inductive lemma for $\wnus{n}$}
\label{sec:the-inductive-lemma-for-wnus}

We continue with the notation of~\ref{subsec:supramaximal-def-global-base-S},
but we restrict to the case where $\ptst$ consists of only one ideal $\m$.
The purpose of this section is to establish the following lemma:

\begin{lemma}\label{lem:main-geometric-finite-length}
Let $X$ be an object of $\AlgSp_m$, with $m\in\bZ$. 
	\begin{enumerate}
		\item \label{lem-mgfl-part1}
			$\wnus{n}(X)$ is an algebraic space; and for any $(m-1)$-representable \'etale 
			surjection $g\:U\to X$, where $U$ is a disjoint union of affine schemes, the space
			$\wnus{n}(U\times_X U)$	is an \'etale equivalence relation on $\wnus{n}(U)$ with
			respect to the map
						\begin{equation} \label{map:main-lemma-1}
							(\wnus{n}(\pr_1),\wnus{n}(\pr_2))\: \wnus{n}(U\times_X U) \longmap
								\wnus{n}(U)\times_{\wnus{n}(X)} \wnus{n}(U),
						\end{equation}
					and the induced map
							\[
							\wnus{n}(U)/\wnus{n}(U\times_X U) \longmap \wnus{n}(X)
							\]
					is an isomorphism. In particular, (\ref{map:main-lemma-1})
					is an isomorphism.
		\item \label{lem-mgfl-part1.5}
			For any map $g$ as in (\ref{lem-mgfl-part1}), the diagram
					\[
					\entrymodifiers={+!!<0pt,\fontdimen22\textfont2>}
					\def\objectstyle{\displaystyle}
					\xymatrix@C=40pt{
						U \ar^{g}[r]\ar_{\gh{0}}[d]		& X \ar^{\gh{0}}[d] \\
						\wnus{n}(U) \ar^{\wnus{n}(g)}[r]			& \wnus{n}(X)
						}
					\]
					is cartesian.
		\item \label{lem-mgfl-part2}
			The map 
			\begin{equation} \label{map:main-geometric-finite-length-1}
				\xymatrix@C=33pt{
					X\times_S S_0 \ar^-{\gh{0}\times \id}[r]
						& \wnus{n}(X)\times_S S_0
				}
			\end{equation}
			is a closed immersion defined by a square-zero ideal sheaf,
			where $S_0$ denotes $\Spec \sO_S/\m$.
		\item \label{lem-mgfl-part3}
			For any object $X'\in\AlgSp_{m}$ and any \'etale map $f\:X'\to X$, the map 
				\[
				\wnus{n}(f)\:\wnus{n} (X') \to \wnus{n} (X)
				\]
			is \'etale; and for any algebraic space $Y$ over $X$, the map
				\begin{equation} \label{map:main-lemma-2}
	   			(\wnus{n}(\pr_1),\wnus{n}(\pr_2))\: \wnus{n}(X'\times_X Y) 
	   				\longlabelmap{} \wnus{n}(X')\times_{\wnus{n}(X)}\wnus{n}(Y)
				\end{equation}
	   		is an isomorphism.
	\end{enumerate}
\end{lemma}

\begin{proof} 
	We will prove all parts at once by induction on $m$. For clarity, write
	$\text{(\ref{lem-mgfl-part1})}_{m}$ for the	statement $\text{(\ref{lem-mgfl-part1})}$ above,
	and so on. 
	
	First consider the case where $m\leq -1$. Here
	we use the fact that $\wnus{n}$ preserves coproducts
	together with the analogous affine results:
	$\text{(\ref{lem-mgfl-part1})}_{m}$	follows from~(\ref{eq:wnus-of-affine-is-W}),
	9.2, 6.9,
	and~9.4;
	$\text{(\ref{lem-mgfl-part2})}_{m}$ follows from~6.8; and
	$\text{(\ref{lem-mgfl-part3})}_{m}$ follows from 9.2
	and~9.4. 	
	
	It remains to prove $\text{(\ref{lem-mgfl-part1.5})}_{m}$.
	It is enough to assume $U$ and $X$ are affine.
	Consider the map
		\[
		a=(g,\gh{0})\: U\longmap X\times_{\wnus{n}(X)}\wnus{n}(U).
		\]
	By assumption, the source is \'etale over $X$, and by~9.2
	so is the target. Therefore $a$ itself is \'etale, and so
	to show it is an isomorphism,
	it is enough by~\ref{lem:isom-criterion-for-etale-maps} below to show that the maps 
	$\red{(S'\times_S a)}$ and $\red{(S_0\times_S a)}$ are isomorphisms, where $S'=S-S_0$.
	On the one hand, $S'\times_S a$ agrees,
	by~(\ref{eq:localization-of-S-commutes-with-wnus}),
	with $S'\times_S\vbl$ applied to the evident map
		\[
		U\longmap X \times_{(\coprod_{[0,n]} X)} \coprod_{[0,n]} U,
		\]
	which is an isomorphism. On the other hand, by~6.8,
	the map $\red{(S_0\times a)}$ agrees with $\red{b}$, where $b$ is the evident map 
	$S_0\times_S U \to S_0\times_S (X\times_X U)$.
	Since $b$ is an isomorphism, so is $\red{(S_0\times a)}$.
	This proves $\text{(\ref{lem-mgfl-part1.5})}_{m}$ and hence the lemma for $m\leq -1$.

	From now on, assume $m\geq 0$.  

	$\text{(\ref{lem-mgfl-part1})}_m$:
	First, observe that it follows from the rest of (\ref{lem-mgfl-part1}) that
	$\wnus{n}(X)$ is an algebraic space. Indeed, because $X$ is in $\AlgSp_m$, 
	there exists a map $g$ as in (\ref{lem-mgfl-part1}). Assuming
	the rest of (\ref{lem-mgfl-part1}), we have
		\[
		\wnus{n}(X) \cong \wnus{n}(U)/\wnus{n}(U\times_X U),
		\]
	and so it is enough
	to show that $\wnus{n}(U)$ and $\wnus{n}(U\times_XU)$ are algebraic spaces.
	This follows from $\text{(\ref{lem-mgfl-part1})}_{m-1}$ because 
	$U,U\times_X U\in \AlgSp_{m-1}$.
	
	The diagram
		\[
		\displaylabelcofork{U\times_X U}{\pr_1}{\pr_2}{U}{f}{X}
		\]
	is a coequalizer diagram and, since
	$\wnus{n}$ commutes with colimits, so is
		\begin{equation} \label{diag:internal-110}
		\displaylabelcofork{\wnus{n}(U\times_X U)}{\wnus{n}(\pr_1)}{\wnus{n}(\pr_2)}
			{\wnus{n}(U)}{\wnus{n}(f)}{\wnus{n}(X).}
		\end{equation}
	Thus, all that remains is to show that $\wnus{n}(U\times_X U)$ is an \'etale equivalence 
	relation on $\wnus{n}(U)$ under the structure map (\ref{map:main-lemma-1}).
	By $\text{(\ref{lem-mgfl-part3})}_{m-1}$, the projections $\wnus{n}(\pr_i)$ in 
	(\ref{diag:internal-110}) are \'etale.
	Let us now show that $\wnus{n}(U\times_X U)$ is an equivalence relation.

	Let $t$ denote the map (\ref{map:main-lemma-1}).
	Let us first show that $t$ is a monomorphism.
	We can view this as a map of algebraic spaces
	over $\wnus{n}(U)$ by projecting onto the first factor, say.
	Then since $\wnus{n}(U\times_X U)$ is \'etale over $\wnus{n}(U)$, 
	it is enough, by~\ref{lem:monomorphism-criterion-for-etale-spaces} below, 
	to show that $t\times_S S'$ and $\red{(t\times_S S_0)}$,
	are monomorphisms.
	For $t\times_S S'$, we may assume $\m$ is the unit ideal, 
	by~(\ref{eq:localization-of-S-commutes-with-wnus});
	then $t$ agrees with the evident map
		\[
		\coprod_{[0,n]}U\times_X U \longmap  \Bigl(\coprod_{[0,n]}U\Bigr) 
			\times_{S} \Bigl(\coprod_{[0,n]} U\Bigr),
		\]
	which is a monomorphism. On the other hand, by $\text{(\ref{lem-mgfl-part2})}_{m-1}$,
	the map $\red{(t\times_S S_0)}$ can be identified with $\red{u}$, where
	$u$ is the evident map
		\[
		U\times_X U\times_S S_0 \longmap U\times_S U\times_S S_0,
		\]
	which is a monomorphism. 
	This proves $t$ is a monomorphism.

 	Now let us show that $\wnus{n}(U\times_X U)$ 
	is an equivalence relation on $\wnus{n}(U)$. It is reflexive and symmetric, simply because
	$\wnus{n}$ is a functor and $U\times_X U$ is reflexive and symmetric relation on $U$.
	Let us show transitivity. 
	
	Write 
		\[
		\Gamma_1=U\times_X U \quad\text{and} \quad 
			\Gamma_2=\wnus{n}(U)\times_{S}\wnus{n}(U).
		\]
	Then by definition, $\wnus{n}(\Gamma_1)$ is transitive if and only if there exists a map $c'$
	making the right-hand square in the diagram
		\begin{equation*}\label{diag:interal-113}
			\xymatrix{
			\wnus{n}(\Gamma_1\times_U\Gamma_1) \ar^-{t'}[r] \ar_{\wnus{n}(c_1)}[dr]
				& \wnus{n}(\Gamma_1)\times_{\wnus{n}(U)}\wnus{n}(\Gamma_1) 
					\ar^-{t\times t}[r] \ar@{-->}^{c'}[d]
				& \Gamma_2 \times_{\wnus{n}(U)}\Gamma_2 \ar^{c_2}[d] \\
				& \wnus{n}(\Gamma_1) \ar@{>->}^{t}[r]
				& \Gamma_2
			}
		\end{equation*}
	commute,
	where each $c_i$ is the transitivity map for the equivalence relation $\Gamma_i$.
	If we define $t'=(\wnus{n}(\pr_1),\wnus{n}(\pr_2))$,
	then the perimeter commutes. Therefore it is enough to show that
	$t'$ is an isomorphism. 
	This follows from $\text{(\ref{lem-mgfl-part3})}_{(m-1)}$, which we can
	apply because we have $\Gamma_1,U\in\AlgSp_{m-1}$, as discussed above.

	$\text{(\ref{lem-mgfl-part1.5})}_{m}$:
	To show that the map
		\[
		(g,\gh{0})\:U \longmap X\times_{\wnus{n}(X)}\wnus{n}(U)
		\]
	is an isomorphism, it suffices to do so after applying $U\times_X \vbl$.
	We can do that as follows:
		\begin{align*}
			U\times_X U &\labeleq{1} U\times_{\wnus{n}(U)} \wnus{n}(U\times_X U) \\
						&\labeleq{2} U\times_{\wnus{n}(U)} 
							\wnus{n}(U)\times_{\wnus{n}(X)}\wnus{n}(U)\\
						&= U\times_{\wnus{n}(X)} \wnus{n}(U).
		\end{align*}
	Equality 2 follows from $\text{(\ref{lem-mgfl-part1})}_{m}$.
	Thus it suffices to show equality 1.

	Let $h\:V\to U\times_X U$ be an $(m-2)$-representable \'etale cover, 
	where $V\in\AlgSp_{-1}$; this exists because $U\times_XU\in\AlgSp_{m-1}$.
	Consider the following diagram:
		\[
		\xymatrix@C=40pt{
		V \ar^-h[r]\ar^{\gh{0}}[d] &
			U\times_X U \ar^-{\pr_1}[r] \ar^{\gh{0}}[d] &
			U \ar^{\gh{0}}[d] \\
		\wnus{n}(V) \ar^-{\wnus{n}(h)}[r] &
			\wnus{n}(U\times_X U) \ar^-{\pr_1}[r] &
			\wnus{n}(U).
		}
		\]
	By $\text{(\ref{lem-mgfl-part1.5})}_{m-1}$, 
	the left-hand square is cartesian, since $U\times_X U\in\AlgSp_{m-1}$.
	Further, the perimeter is cartesian; this is because $U$ and $V$ are disjoint
	unions of affine schemes, and so we can apply 
	$\text{(\ref{lem-mgfl-part1.5})}_{-1}$ on each component.
	Therefore the induced map
		\begin{equation} \label{map:internal-sdkds}
			U\times_X U\longmap U\times_{\wnus{n}(U)}\wnus{n}(U\times_XU)
		\end{equation}
	becomes an isomorphism when we apply the functor 
	$\vbl\times_{\wnus{n}(U\times_X U)}\wnus{n}(V)$. But the map 
	$\wnus{n}(V)\to\wnus{n}(U\times_XU)$ is an \'etale cover, by 
	$\text{(\ref{lem-mgfl-part1})}_{m-1}$. So this implies (\ref{map:internal-sdkds}) 
	is an isomorphism.
	
	$\text{(\ref{lem-mgfl-part2})}_m$:	
	Let $g\:U\to X$ be an $(m-1)$-representable	\'etale cover, where $U\in\AlgSp_{-1}$.
	Then $\wnus{n}(g)$ is an \'etale cover, by $\text{(\ref{lem-mgfl-part1})}_{m}$. 
	Therefore it is enough to show that (\ref{map:main-geometric-finite-length-1})
	becomes a closed immersion defined by a square-zero ideal after base change
	from $\wnus{n}(X)$ to $\wnus{n}(U)$---indeed, this is an \'etale-local
	property. 
	But by $\text{(\ref{lem-mgfl-part1.5})}_{m}$, this map can be identified with
		\[
		\gh{0}\times\id_{S_0}\:U\times_S S_0 \longlabelmap{} \wnus{n}(U)\times_S S_0,
		\]
	which has the required property by $\text{(\ref{lem-mgfl-part2})}_{-1}$.

	$\text{(\ref{lem-mgfl-part3})}_m$:	
	Let $u'\:U'\to X'$ and $u\:U\to X$ be $(m-1)$-representable
	\'etale covers, where $U',U\in\AlgSp_{-1}$,
	such that the map $f\:X'\to X$ lifts to a map $h\:U'\to U$,	necessarily \'etale.
	Then we have a commutative diagram
		\[
		\xymatrix@C=35pt{
		\wnus{n}(U') \ar@{->}^{\wnus{n}(u')}[r]\ar_{\wnus{n}(h)}[d] 
				& \wnus{n}(X') \ar^{\wnus{n}(f)}[d] \\
		\wnus{n}(U') \ar@{->}^{\wnus{n}(u)}[r] 
				& \wnus{n}(X).
		}
		\]
	By $\text{(\ref{lem-mgfl-part1})}_m$, all the spaces in this diagram are algebraic,
	and the horizontal maps are \'etale covers.
	By $\text{(\ref{lem-mgfl-part3})}_{-1}$, the map $\wnus{n}(h)$ is \'etale.
	Therefore $\wnus{n}(f)$ is \'etale.

	Let us now show that~(\ref{map:main-lemma-2}) is an isomorphism.
	The map 
		\[
		\pr_2\:\wnus{n}(X')\times_{\wnus{n}(X)} \wnus{n}(Y)\longmap\wnus{n}(Y)
		\]
	is \'etale, because it is a base change of $\wnus{n}(f)$.
	The map	
		\[
		\wnus{n}(\pr_2)\:\wnus{n}(X'\times_X Y)\longmap\wnus{n}(Y)
		\]
	is also \'etale. Indeed, by the above, we only
	need to verify $X'\times_X Y\in\AlgSp_m$. This holds because
	$\AlgSp_m$ is stable under pull back, by~\cite{Toen-Vezzosi:HAGII}, Corollary 1.3.3.5.
	
	Therefore, we can view~(\ref{map:main-lemma-2}), which we will denote $t$,
	as a map of \'etale algebraic spaces over $\wnus{n}(Y)$.
	So, to show $t$ is an isomorphism, it is enough by~\ref{lem:isom-criterion-for-etale-maps}
	below to show $\red{(t\times_S S')}$ and $\red{(t\times_S S_0)}$ are isomorphisms.
	This can be done as in the proof of $\text{(\ref{lem-mgfl-part1.5})}$: for $t\times_S S'$, 
	use~(\ref{eq:localization-of-S-commutes-with-wnus}) 
	to reduce the question to one about ghost
	components; for $\red{(t\times_S S_0)}$, use $\text{(\ref{lem-mgfl-part2})}_m$.
\end{proof}

\begin{lemma}\label{lem:monomorphism-criterion-for-etale-spaces}
	Consider a commutative diagram of algebraic spaces 
		\[
		\xymatrix{
		X \ar^{f}[rr] \ar^{g}[dr] & & Y\ar[dl] \\
		& Z,
		}
		\]
	where $g$ is \'etale. Then the following hold:
	\begin{enumerate}
		\item $f$ is a monomorphism if and only if $\red{f}$ is.
		\item Let $Z_0$ be a closed algebraic subspace of $Z$, and let $Z'$ be its complement. 
			Then $f$ is a monomorphism if and only if $f\times_Z Z_0$ and $f\times_Z Z'$ are.
	\end{enumerate}
\end{lemma}
\begin{proof}
	The only-if parts of both statements follow immediately from the fact that both closed and
	open immersions are monomorphisms.
	
	(a): It is enough to show that
	for any affine scheme $T$ and any maps $a,b\:T\to X$ such that $f\circ a=f \circ b$,
	we have $a=b$. Then we have $\red{f}\circ \red{a}=\red{f}\circ\red{b}$. Since $\red{f}$
	is assumed to be a monomorphism, we have $\red{a}=\red{b}$. Since $X$ is \'etale over $Z$,
	we have $a=b$ (EGA IV 18.1.3~\cite{EGA-no.32}).
	
	(b): Again, let $T$ be an affine scheme with maps $a,b\:T\to X$ 
	such that $f\circ a=f \circ b$.
	Let $\bar{T}$ denote the equalizer of $a$ and $b$. It is an algebraic subspace of $T$.
	By the assumptions on $f$, we have $\bar{T}\times_Z Z_0=T\times_Z Z_0$ and
	$\bar{T}\times_Z Z'=T\times_Z Z'$. 
	Therefore $\bar{T}$ is a closed subscheme of $T$
	defined by a nil ideal. As above, since $X$ is \'etale over $Z$,
	there is at most one extension of the $Z$-morphism $\bar{T}\to X$ to $T$.
	Therefore $a=b$.
\end{proof}

\begin{lemma}\label{lem:isom-criterion-for-etale-maps}
	Let $f\:X\to Y$ be an \'etale map of algebraic spaces, and let $Y_0$ be a closed algebraic
	subspace of $Y$ with complement $Y'$. Then $f$ is an isomorphism if and only if 
	$\red{(f\times_Y Y_0)}$	and $\red{(f\times_Y Y')}$ are.
\end{lemma}
\begin{proof}
	The only-if statement is clear. Now consider the converse. It follows from 
	\ref{lem:monomorphism-criterion-for-etale-spaces} that $f$ is a monomorphism, and so it is 
	enough to show that $f$ is an epimorphism. To do this, it is enough to show that
	any \'etale map $V\to Y$, with $V$ affine, lifts to a map $V\to X$. Thus, by changing base
	to $V$ and relabeling $Y=V$, 
	we may assume $Y$ is affine. (The property of being an isomorphism after applying
	$\red{(\vbl)}$ is stable under base change.) Now let $(U_i)_{i\in I}$ be an \'etale cover
	of $X$, where each $U_i$ is an affine scheme. Then each composition $U_i\to X\to Y$ 
	is an \'etale morphism of affine schemes, and the union of images of these maps covers $Y$.
	Therefore the induced map $\coprod_i U_i\to Y$ is an epimorphism, and hence so is $f$.
\end{proof}

\section{$\wnus{n}$ and algebraic spaces}
\label{sec:wnus-and-algebraic-spaces}

The purpose of this section is to give a number of useful
consequences of the inductive lemma in the previous section.
We continue with the notation of~\ref{subsec:supramaximal-def-global-base-S}.

\begin{theorem}\label{thm:W-main-geometric-finite-length(1)}
	Let $X$ be an algebraic space over $S$. Then $\wnus{n} (X)$ is an algebraic space. 
\end{theorem}

\begin{proof}
	By~(\ref{eq:geom-witt-iterate-upper-s}), we may assume $\ptst$ consists of
	one ideal, in which case 
	we can apply~\ref{lem:main-geometric-finite-length}(\ref{lem-mgfl-part1}).
\end{proof}

\begin{theorem}	\label{thm:W-main-geometric-finite-length(2)}
	Let $f\: X'\to X$ be an \'etale map of algebraic spaces over $S$.  Then
	the following hold.
	\begin{enumerate}
		\item \label{part:W-main-geometric-finite-length(2)-a}
			The induced map 
			$\wnus{n}(f)\:\wnus{n}(X')\to\wnus{n}(X)$ is \'etale.					
		\item \label{part:W-main-geometric-finite-length(2)-b}
			If $f$ is surjective, then so is $\wnus{n}(f)$.
		\item \label{part:W-main-geometric-finite-length(2)-c}
			For any algebraic space $Y$ over $X$, the map $(\wnus{n}(\pr_1),\wnus{n}(\pr_2))$
				\[
				\wnus{n}(X'\times_X Y) \longmap 
					\wnus{n}(X')\times_{\wnus{n}(X)}\wnus{n}(Y)
				\]
			is an isomorphism.
	\end{enumerate}
\end{theorem}
\begin{proof}
	By~(\ref{eq:geom-witt-iterate-upper-s}), we may assume $\ptst$ consists of one
	ideal. Then parts (\ref{part:W-main-geometric-finite-length(2)-a}) and
	(\ref{part:W-main-geometric-finite-length(2)-c}) follow
	from~\ref{lem:main-geometric-finite-length}(\ref{lem-mgfl-part3}).
	For part (\ref{part:W-main-geometric-finite-length(2)-b}), it is 
	enough, by passing to an \'etale cover of $X'$, to assume $X'\in\AlgSp_{-1}$. Then we can 
	apply~\ref{lem:main-geometric-finite-length}(\ref{lem-mgfl-part1}), because $f$ is
	$1$-representable.
\end{proof}

\begin{corollary}\label{cor:etale-chart-presentation-of-wnus}
	Let $(U_i)_{i\in I}$ be an \'etale cover of an algebraic space $X$ over $S$.
	Then $\bigl(\wnus{n}(U_i)\bigr)_{i\in I}$ is an \'etale cover of $\wnus{n}(X)$, and
	for each pair $(i,j)\in I^2$, the map
		\[
		\wnus{n}(U_i\times_X U_j) \longmap \wnus{n}(U_i)\times_{\wnus{n}(X)} \wnus{n}(U_j)
		\]
	given by $\bigl(\wnus{n}(\pr_1),\wnus{n}(\pr_2)\bigr)$ is an isomorphism.
\end{corollary}

In other words, $\wnus{n}(X)$ can be constructed by charts in the \'etale topology.

\begin{proof}
	Because $\wnus{n}$ is a left adjoint, it preserves disjoint unions.
	Then apply
	\ref{thm:W-main-geometric-finite-length(1)}(\ref{part:W-main-geometric-finite-length(2)-b})
	to the induced map $\coprod_i U_i\to X$ and 
	\ref{thm:W-main-geometric-finite-length(1)}(\ref{part:W-main-geometric-finite-length(2)-c}) 
	to $U_i\times_X U_j$.
\end{proof}

\begin{corollary}\label{cor:geometric-W-etale-base-change}
	Let $f\:X\to Y$ be an \'etale map of algebraic spaces.  Then the following
	diagrams are cartesian, where the horizontal maps are the ones defined 
	in~\ref{subsec:natural-maps}: 
	\begin{enumerate}
		\item for $i\in\nset$,
			\[
			\xymatrix{
			\wnus{n}(X) \ar^-{\psi_i}[r] \ar[d]
				& \wnus{n+i}(X) \ar[d]	\\
			\wnus{n}(Y) \ar^-{\psi_i}[r]	
				& \wnus{n+i}(Y); 				
			}		
			\]			
		\item for $i=\nset$,
			\[
			\xymatrix{
			\wnus{n}(X) \ar[d]\ar^-{\inclmap{n,i}}[r]
				& \wnus{n+i}(X) \ar[d] \\
			\wnus{n}(Y) \ar^-{\inclmap{n,i}}[r]
				& \wnus{n+i}(Y); 
			}		
			\]
		\item for $i\in[0,n]$,
			\[
			\xymatrix{
			X \ar^-{\gh{i}}[r] \ar[d]
				& \wnus{n}(X) \ar[d] \\
			Y \ar^-{\gh{i}}[r]
				& \wnus{n}(Y);
			}
			\]
		\item when $\ptst=\{\m\}$ and $i\in\bN$,
				\[
				\xymatrix{
				X \times_S S_{n} \ar[d] \ar^-{\rgh{i}}[r]
					& \wnus{n}(X) \ar[d] \\
				Y \times_S S_{n} \ar^-{\rgh{i}}[r]
					& \wnus{n}(Y),
				}
				\]
			where $S_{n}=\Spec\sO_S/\m^{n+1}$.
	\end{enumerate}	
\end{corollary}

\begin{proof}
	By~(\ref{eq:geom-witt-iterate-upper-s}), we may assume $\ptst$ consists of one 
	ideal $\m$.  All four parts are proved by the same method. 
	Let us give the details for (a) and leave the rest to the reader.

	We want to show that the induced map
		\[
		g\:\wnus{n}(X) \longmap \wnus{n}(Y)\times_{\wnus{n+i}(Y)}\wnus{n+i}(X)
		\]
	is an isomorphism. By~\ref{thm:W-main-geometric-finite-length(2)}, this is a map
	of \'etale algebraic spaces over $\wnus{n}(Y)$.
	Therefore, to show it is an isomorphism, it is enough 
	by~\ref{lem:isom-criterion-for-etale-maps} to show that $\red{(g\times_S S_0)}$
	and $\red{(g\times_S S')}$ are isomorphisms, where $S_0=\Spec \sO_S/\m$ and $S'=S-S_0$.

	It is easy check that $g\times_S S'$ is an isomorphism.
	Write $X_0=X\times_S S_0$, $Y_0=Y\times_S S_0$, and let $F$ denote the $q$-th power
	Frobenius map.
	Then the map $\red{(g\times_S S_0)}$, can be identified with $\red{h}$, where 
	$h\:X_0\to Y_0\times_{F^i,Y_0} X_0$ is the map induced by the diagram
		\[
		\xymatrix{
			X_0 \ar^{F^i}[r]\ar[d] & X_0 \ar[d] \\
			Y_0 \ar^{F^i}[r] & Y_0.
		}
		\]
	This diagram is cartesian (SGA 5 XV \S 1, Proposition 2(c)\cite{SGA5}), 
	and so $h$ and $\red{h}$ are isomorphisms.
\end{proof}

\begin{corollary}\label{cor:geometric-W-preserves-finite-etale-limits}
	Let $X$ be an algebraic space over $S$.
	\begin{enumerate}
		\item Let 
				\[
				\displayfork{U}{Y}{Z}
				\]
			be an equalizer diagram of algebraic spaces over $X$.
			If $Z$ is \'etale over $X$, then the induced diagram
				\[
				\displayfork{\wnus{n}(U)}{\wnus{n}(Y)}{\wnus{n}(Z)}
				\]
			is also an equalizer diagram.
		\item Let $(Y_i)_{i\in I}$ be a finite diagram of
			\'etale algebraic $X$-spaces.  Then the following natural map
			is an isomorphism:
				\[
				\wnus{n}(\limm_{i\in I} Y_i) \longisomap \limm_{i\in I} \wnus{n}(Y_i).
				\]
			Here the limits are taken in the category of $X$-spaces.
	\end{enumerate}
\end{corollary}

\begin{proof}
	(a): Since the structure map $Z\to X$ is \'etale,
	so is the diagonal map $Z\to Z\times_X Z$. And since $U$ is
	$Y\times_{Z\times_X Z}Z$, we have 
	by~\ref{thm:W-main-geometric-finite-length(2)}(\ref{part:W-main-geometric-finite-length(2)-c})
	\begin{align*}
		\wnus{n}(U) &= \wnus{n}(Y) \times_{\wnus{n}(Z\times_X Z)} \wnus{n}(Z) \\
					&= \wnus{n}(Y) \times_{(\wnus{n}(Z)\times_{\wnus{n}(X)} \wnus{n}(Z))} 
						\wnus{n}(Z)
	\end{align*}
	Thus $\wnus{n}(U)$ is the equalizer of the two induced maps 
	$\wnus{n}(Y)\rightrightarrows \wnus{n}(Z)$.	

	(b): To show a functor preserves finite limits, it is sufficient
	to show it preserves finite products
	and equalizers of pairs of arrows. The first follows 
	from~\ref{thm:W-main-geometric-finite-length(2)}(c), and the second from part (a) above. 
\end{proof}

\begin{corollary}\label{cor:W-preserves-open-immersions}
	Let $j\:U\to X$ be an open immersion of algebraic spaces.
	Then the map $\wnus{n}(j)\:\wnus{n}(U)\to\wnus{n}(X)$ is an open immersion.
	If $X$ is a scheme, then so is $\wnus{n}(X)$.
\end{corollary}

\begin{proof}
	An open immersion is the same as an \'etale monomorphism.
	By~\ref{thm:W-main-geometric-finite-length(2)}, $\wnus{n}(j)$ is \'etale, and so
	we only need to show it is a monomorphism or, equivalently, that its diagonal map
	is an isomorphism. By
	\ref{thm:W-main-geometric-finite-length(2)}(\ref{part:W-main-geometric-finite-length(2)-c}),
	the diagonal map of $\wnus{n}(j)$ agrees with $\wnus{n}(\Delta_j)$, where $\Delta_j$
	is the diagonal map $U\to U\times_X U$ of $j$. Because $j$ is a monomorphism, $\Delta_j$
	is an isomorphism, and hence so is $\wnus{n}(\Delta_j)$.
	Therefore the diagonal map of $\wnus{n}(j)$ is an isomorphism,
	and so $\wnus{n}(j)$ is a monomorphism.

	Now suppose $X$ is a scheme. Let $(U_i)_{i\in I}$ be an open cover of $X$.
	By~\ref{cor:etale-chart-presentation-of-wnus}, $\wnus{n}(X)$ is an algebraic space covered by 
	the $\wnus{n}(U_i)$, and by the above, each map $\wnus{n}(U_i)\to\wnus{n}(X)$ is an open 
	immersion. Therefore $X$ is a scheme.
\end{proof}

\begin{corollary}\label{cor:gh-is-surj-and-integral-geometric}
	Let $X$ be an algebraic space over $S$.
	\begin{enumerate}
		\item The map $\gh{\leq n}\:\coprod_{[0,n]}X \to \wnus{n}(X)$ is surjective and integral,
		and the kernel $I$ of the induced map 
			\[
			\sO_{\wnus{n}(X)}\to \gh{\leq n*}\bigl(\sO_{\coprod_{[0,n]}X}\bigr)
			\] 
		satisfies $I^{2^N}=0$, where $N=\sum_{\m}n_{\m}$.
		\item The map $\gh{0}\:X\to\wnus{n}(X)$ is a closed immersion.
	\end{enumerate}
\end{corollary}
\begin{proof}
	All the properties in question are \'etale-local on $\wnus{n}(X)$, and hence $S$.
	Therefore, we may assume that $S$ is affine,
	by~(\ref{eq:localization-of-S-commutes-with-wnus}), and that $X$ is affine,
	by~\ref{cor:geometric-W-etale-base-change}(b).	In this case, (a) was proved 
	in~\ref{lem:gh-integral-and-surj-multi-prime}, and (b) follows from the surjectivity
	of the map $\gh{0}\:W_n(A)\to A$, which follows from the existence of the Teichm\"uller
	section (1.21), say.
\end{proof}

\begin{corollary}\label{cor:wnus-is-faithful}
	The functor $\wnus{n}\:\AlgSp\to\AlgSp$ is faithful.
\end{corollary}
\begin{proof}
	The map $\gh{0}$ is easily seen to be equal to the composition
		\[
		X \longlabelmap{\eps} \wnls{n}\wnus{n} X \longlabelmap{\cgh{0}} \wnus{n} X,
		\]
	where $\eps$ is the unit of the evident adjunction, and 
	$\cgh{0}$ is as in~(\ref{map:wnls-change-of-n-projection-map}).
	Therefore by~\ref{cor:gh-is-surj-and-integral-geometric}(b), the map $\eps$ is a monomorphism.
	Equivalently, $\wnus{n}$ is faithful.
\end{proof}

\subsection{} \emph{$\wnus{n}$ is generally not full.}
For example, if we consider the usual $p$-typical Witt vectors over $\bZ$
of length $n$, and if $A$ and $B$ are $\bZ[1/p]$-algebras, then we have
\[
	\Hom_{W_n(\bZ)}(W_n(A),W_n(B)) = \Hom_{\bZ^{[0,n]}}(A^{[0,n]},B^{[0,n]})
		= \Hom(A,B)^{[0,n]},
\]
which is usually not the same as $\Hom(A,B)$.  To be sure, the entire point of the theory
is in applying $W$ to rings where $p$ is not invertible.

\begin{corollary}\label{cor:plethysm-base-change}
	Let $f\:X\to Y$ be an \'etale map of algebraic spaces over $S$.  Then the diagram
		\begin{equation} \label{diag:plethysm-base-change}
		\xymatrix{
			\wnus{m}\wnus{n}(X) \ar^-{\mu_X}[r]\ar_{\wnus{m}\wnus{n}(f)}[d]			
				& \wnus{m+n}(X)\ar^{\wnus{m+n}(f)}[d] \\
			\wnus{m}\wnus{n}(Y) \ar^-{\mu_Y}[r]					& \wnus{m+n}(Y)
		}
		\end{equation}
	is cartesian, where $\mu_X$ and $\mu_Y$ are the plethysm maps 
	of~(\ref{map:geometric-plethysm}).
\end{corollary}
\begin{proof}
	We use the usual method, as in \ref{cor:geometric-W-etale-base-change}.
	Let us assume that $\ptst$ consists of one ideal $\m$.
	This is sufficient by~\ref{thm:W-main-geometric-finite-length(2)},
	(\ref{eq:geom-witt-iterate-upper-s}), and a short argument we
	we leave to the reader.
	
	By~\ref{thm:W-main-geometric-finite-length(2)}, the map
		\begin{equation} \label{eq:local-4838}
		\wnus{m}\wnus{n}(X) \longlabelmap{g} \wnus{m+n}(X) \times_{\wnus{m+n}(Y)} 
				\wnus{m}\wnus{n}(Y)
		\end{equation}
	is \'etale, and so we only need to show $g\times_S S'$ and $\red{(g\times_S S_0)}$
	are isomorphisms, where $S_0=\Spec \sO_S/\m$ and $S'=S-S_0$.
	
	Consider $g\times_S S'$ first. By~(\ref{eq:localization-of-S-commutes-with-wnus}), we may
	assume $\m$ is the unit ideal. Then diagram~(\ref{diag:plethysm-base-change}) can be identified
	with the diagram
		\[
		\entrymodifiers={+!!<0pt,\fontdimen22\textfont2>}
		\def\objectstyle{\displaystyle}
		\xymatrix{
			\coprod_{[0,m]}\coprod_{[0,n]} X \ar^-{\mu_X}[r]\ar[d]			
				& \coprod_{[0,m+n]} X\ar[d] \\
			\coprod_{[0,m]}\coprod_{[0,n]} Y \ar^-{\mu_Y}[r]				
				& \coprod_{[0,m+n]} Y,
		}
		\]
	where each map $\mu$ sends component $(i,j)\in [0,m]\times[0,n]$ 
	identically to component $i+j\in[0,m+n]$.
	Since this diagram is cartesian, $g\times_S S'$ is an isomorphism.

	Now consider $\red{(g\times_S S_0)}$.
	Write $(\vbl)'$ for the functor $T\mapsto \red{(S_0\times_S T)}$;
	thus we want to show $g'$ is an isomorphism.
	Consider the commutative diagram	
		\[
		\xymatrix{
		\wnus{m}\wnus{n}(X)' \ar^-{g'}[r]\ar^b[dr]			
		& \bigl(\wnus{m+n}(X) \times_{\wnus{m+n}(Y)} \wnus{m}\wnus{n}(Y)\bigr)'
				\ar^{c}[d] \\
			& \wnus{m+n}(X)' \times_{\wnus{m+n}(Y)'} 
				\wnus{m}\wnus{n}(Y)',
		}
		\]
	where $c$ is the evident map induced by the universal property of products. 
	Since the functor $(\vbl)'$ sends \'etale maps to \'etale maps, $c$ is \'etale;
	also $\red{c}$ is an isomorphism, and so $c$ is an isomorphism.
 	Therefore it is enough to show that $b$ is isomorphism---in other words, 
	that diagram~(\ref{diag:plethysm-base-change}) becomes cartesian after
	applying $(\vbl)'$. 

	To do this, it is enough to show that for $T=X,Y$ (or any algebraic 
	space over $S$), the map $\mu'_T\:\wnus{m}\wnus{n}(T)' \to \wnus{m+n}(T)'$
	is an isomorphism. Consider the commutative diagram
		\[
		\xymatrix{
				& T \ar^{\gh{0}}[dr]\ar_{\gh{0}}[dl] \\
		\wnus{n}(T) \ar_-{\gh{0}}[r]
				& \wnus{m}\wnus{n}(T) \ar_-{\mu_T}[r]
				& \wnus{m+n}(T). \\
		}
		\]
	If we apply $S_0\times\vbl$ to this diagram, all maps labeled $\gh{0}$
	become closed immersions defined by square-zero ideals, 
	by~\ref{lem:main-geometric-finite-length}(\ref{lem-mgfl-part2}).
	Thus they all become isomorphisms after applying $(\vbl)'$, and therefore
	so does $\mu_T$.
\end{proof}

\section{Preservation of geometric properties by $\wnus{n}$}
\label{sec:geometric-properties-of-W-upper-star}

We continue with the notation of~\ref{subsec:supramaximal-def-global-base-S}.

\spacesubsec{Sheaf-theoretic properties}

\begin{proposition}\label{pro:W-preserves-qcom}
	Let $X$ be a quasi-compact object of $\Space_S$. Then $\wnus{n}(X)$ is quasi-compact.
\end{proposition}

\begin{proof}
Since $X$ is quasi-compact, it has a finite cover $(U_i)_{i\in I}$ by affine schemes.
Therefore $\wnus{n}(\coprod_i U_i)$ is affine, and since
$\wnus{n}(X)$ is covered by this space, it must be quasi-compact.
\end{proof}

\spacesubsec{General localization}

\begin{proposition}\label{pro:obj-general-permanence-properties}
	Let $P$ be an \'etale-local property of algebraic spaces $X$ over $S$.
	For $\wnus{n}$ to preserve property $P$, it is sufficient that it do so when 
	$\ptst$ consists of one principal ideal and both $X$ and $S$ are affine schemes.
\end{proposition}
\begin{proof}
	When $\ptst$ is empty, $\wnus{n}$ is the identity functor.
	Therefore by~(\ref{eq:geom-witt-iterate-upper-s}), it is enough
	to consider the case where $\ptst$ consists of one ideal $\m$.
	
	Let $X$ be an algebraic space satisfying property $P$, and let
	$(U_i)_{i\in I}$ be an \'etale cover of $X$ such that
	each $U_i$ is affine and lies over an affine open subscheme $S_i$ of $S$ over which 
	$\m$ is principal. Because $P$ is \'etale local, each $U_i$ satisfies $P$; and since we have
	$\wnus{S,n}(U_i)=\wnus{S_i,n}(U_i)$, by (\ref{eq:wnus-is-independent-of-S}),
	so does each $\wnus{n}(U_i)$.
	But the spaces $\wnus{n}(U_i)$ form an affine \'etale cover of $\wnus{n}(X)$,
	by~\ref{thm:W-main-geometric-finite-length(2)}.
	Therefore $\wnus{n}(X)$ satisfies $P$.
\end{proof}

\begin{proposition}\label{pro:maps-local-permanence}
	Let	$P$ be a property of maps $f\:X\to Y$ of algebraic spaces which is \'etale-local on the 
	target. For $\wnus{n}$ to preserve property $P$, it is sufficient that it do so
	when $\ptst$ consists of one principal ideal, $S$ is affine, and $Y$ is affine. 
	
	If property $P$ is also \'etale-local on the source, then 
	we may further restrict to the case where $X$ is affine.
\end{proposition}

\begin{proof}
	Let $f\:X\to Y$ be a map satisfying $P$.
	As in the proof of \ref{pro:obj-general-permanence-properties}, it is enough
	to consider the case where $\ptst$ consists of one ideal $\m$.
	
	Let us show the first statement. 
	Let $(V_j)_{j\in J}$ be an \'etale cover of $Y$ such that
	each $V_j$ is affine and lies over an affine open subscheme
	subscheme $S_j$ of $S$ on which $\m$ is principal.
	Then $\bigl(\wnus{n} (V_j)\bigr)_{j\in J}$ is an \'etale cover of $\wnus{n}(Y)$,
	by~\ref{thm:W-main-geometric-finite-length(2)}(\ref{part:W-main-geometric-finite-length(2)-b}).
	Therefore $\wnus{n}(f)$ satisfies $P$ if its base change to each $\wnus{n}(V_j)$ does.  
	By~\ref{thm:W-main-geometric-finite-length(2)}(\ref{part:W-main-geometric-finite-length(2)-c}), 
	this base change can be identified with
		\[
		\wnus{n}(f_{V_j})\:\wnus{n}(V_j\times_Y X)\longmap\wnus{n}(V_j).
		\]
	Since $f_{V_j}$ satisfies $P$, so does $\wnus{S',n}(f_{V_j})$, by the assumptions
	of the proposition. By~(\ref{eq:wnus-is-independent-of-S}), we have 
	$\wnus{n}(f_{V_j})=\wnus{S',n}(f_{V_j})$, and so 
	$\wnus{n}(f_{V_j})$ also satisfies $P$.
		
	Now suppose property $P$ is also \'etale-local on the source.
	By what we just proved, we may assume $Y$ and $S$ are affine. 
	Let $(U_i)_{i\in I}$ be an \'etale cover of $X$, with each $U_i$ affine.
	Then each composition $U_i\to X\to Y$ satisfies $P$.
	Therefore so does each composition $\wnus{n}(U_i)\to\wnus{n}(X)\to\wnus{n}(Y)$.
	But again 
	by~\ref{thm:W-main-geometric-finite-length(2)}(\ref{part:W-main-geometric-finite-length(2)-b}),
	the spaces $\wnus{n}(U_i)$ form an \'etale cover of $\wnus{n}(X)$.
	Since $P$ is local on the source, $\wnus{n}(f)$ satisfies $P$.
\end{proof}

\spacesubsec{Affine properties of maps}

\begin{proposition}\label{pro:W-preserves-affine-local-properties}
	The following (affine) 
	properties of maps of algebraic spaces are preserved by $\wnus{n}$:
		\begin{enumerate}
			\item \label{pro:W-preserves-affine-local-properties:affine}
				affine,
			\item a closed immersion,
			\item integral,
			\item finite \'etale.
		\end{enumerate}
\end{proposition}
\begin{proof}
	Let $f\:X\to Y$ be a map satisfying one of these properties.
	In particular, $f$ is affine.
	Because the properties are local on the target, it is enough,
	by~\ref{pro:maps-local-permanence}, to assume that both $Y$ and $S$ are affine
	and that $\ptst$ consists of one ideal $\m$.
	But because $f$ is affine, $X$ must also be an affine scheme.
	In other words, it is enough to show $\wnus{n}$ preserves
	these properties for maps of affine schemes.  For (a), there is nothing to prove, and
	(b) is true by~6.5.
	
	Let us prove (c). Write $S=\Spec R$.
	Let $A$ be an $R$-algebra, and let $B$ be an integral $A$-algebra.
	Consider the induced diagram
		\[
		\entrymodifiers={+!!<0pt,\fontdimen22\textfont2>}
		\def\objectstyle{\displaystyle}		
		\xymatrix{
		W_n(B) \ar^-{\gh{\leq n}}[r] & B^{[0,n]} \\
		W_n(A) \ar^-{\gh{\leq n}}[r]\ar[u] & A^{[0,n]}. \ar[u]
		}
		\]
	Since $B$ is integral over $A$, we know that $B^{[0,n]}$ is integral over $A^{[0,n]}$.
	By~8.2, 
	$A^{[0,n]}$ is integral over $W_n(A)$, and hence so is
	$B^{[0,n]}$, and hence so is the image of the ghost map $\gh{\leq n}$.
	But the kernel of $\gh{\leq n}$ is nilpotent; so $W_n(B)$ is integral
	over $W_n(A)$.
\end{proof}

\spacesubsec{Absolute properties and properties relative to $S$}

\begin{proposition}\label{pro:W-preserves-etale-local-properties}
	The following (\'etale local) properties of algebraic spaces over $S$
	are preserved by $\wnus{n}$:
		\begin{enumerate}
			\item locally of finite type over $S$,
			\item flat over $S$,
			\item flat over $S$ and reduced,
			\item of Krull dimension $d$.
		\end{enumerate}
\end{proposition}

\begin{proof}
	Because these are \'etale-local properties, by~\ref{pro:obj-general-permanence-properties}
	we can write $S=\Spec R$, $X=\Spec A$, and $\ptst=\{\pi R\}$ with $\pi\in R$.

	(a): 
 	Let $T$ be a finite subset of $A$ generating it
	as an $R$-algebra, and let $B$ denote the sub-$R$-algebra of
	$W_n(A)$ generated by the set $\{[t] : t\in T\}$ of Teichm\"uller lifts (3.9).
	It is enough to show that $W_n(A)$ is finitely generated as a $B$-module.
	By induction, we may assume $W_{n-1}(A)$ is finitely generated. Therefore
	it is enough to show that $V^nW_n(A)$ is finitely generated.
	We will do this by showing that the subset
		\begin{equation} \label{eq:W-absolute-fin-gen-1}
		T = \Bigl\{V_{\pi}^n\Bigl[\prod_{t\in T}t^{a_t}\Bigr] \,\mid\,
		0\leq a_t<q^n \text{ for all } t\in T\Bigr\} \subseteq V^nW_n(A),
		\end{equation}
	where $[x]$ denotes the Teichm\"uller lift of $x$,
	generates $V^n W_n(A)$ as a $B$-module.
	Indeed, for any monomial $\prod_t t^{c_t}$, write $c_t=b_t q^n + a_t$ with $0\leq a_t < q^n$.
	Then we have by~(3.9.1)
		\[
		V_{\pi}^n\Bigl[\prod_t t^{c_t}\Bigr] = \Bigl(\prod_t [t]^{b_t}\Bigr)
			\Bigl(V_{\pi}^n\Bigl[\prod_t t^{a_t}\Bigr]\Bigr).
		\]
	
	(b): Since $A$ is flat over $R$, the ghost
	map $\gh{\leq n}\:W_n(A)\to A^{[0,n]}$ is injective (2.7).
	Since $A^{[0,n]}$ is $\m$-flat (\ref{subsec:definition-of-E-flat}), so is
	$W_n(A)$. But $R[1/\pi]\tn_R W_n(A)$ is also flat, because
	it agrees with $(A[1/\pi])^{[0,n]}$, by 
	(\ref{eq:localization-of-S-commutes-with-wnus}). Therefore $W_n(A)$ is flat over $R$.

	(c): By (b), we only need to show if $A$
	is flat and reduced over $R$, then $W_n(A)$ is reduced.  Since $A$ is
	flat over $R$, the ghost map $\gh{\leq n}\:W_n(A)\to A^{[0,n]}$ is 
	injective~(2.7). And since $A$ is also reduced, so is
	$W_n(A)$. 

	(d): 
	The ghost map $\gh{\leq n}\:W_n(A)\to A^{[0,n]}$ is integral and surjective on spectra;
	so the Krull dimension of $W_n(A)$ agrees with that of $A^{[0,n]}$, which is $d$.
	(See EGA 0, 16.1.5~\cite{EGA-no.20}.)
\end{proof}

\subsection{} {\em Counterexamples with relative finite conditions and noetherianness.}
\label{subsec:W-does-not-preserve-relative-fg}
It is not true that $\wnus{n}$ preserves relative finite generation or presentation
in general.
For example, consider the usual $p$-typical Witt vectors.
Let $A=\bZ[x_1,x_2,\dots]$, and let $B=A[t]$.  It is then a short exercise
to show that $W_1(B)$ is not a finitely generated $W_1(A)$-algebra.

For another, perhaps more extreme example, let $C=A[t]/(t^2)$.  Then $C$ is 
finite free as an $A$-module, but $W_1(C)$ is not finitely generated as a 
$W_1(A)$-algebra.

Noetherianness is also not preserved.  If $k$ is a field of characteristic $p$,
then $W_1(k)$ is a local ring with residue field $k$ and maximal ideal isomorphic
to $k^{1/p}$.  Therefore $W_1(k)$ is noetherian if and only if $k$ has a finite $p$-basis.

\begin{corollary}\label{cor:qcom-and-ftype-over-S}
The following properties of algebraic spaces over $S$ are preserved by $\wnus{n}$:
	\begin{enumerate}
		\item quasi-compact over $S$, 
		\item finite type over $S$.
	\end{enumerate}
\end{corollary}
\begin{proof}
	Because these properties are \'etale local on $S$, we can assume $S$ is affine,
	by~(\ref{eq:localization-of-S-commutes-with-wnus}).
	
	(a): Since the structure map $X\to S$ is quasi-compact and $S$ is affine, $X$ is quasi-compact.
	Then $\wnus{n}(X)$ is quasi-compact, by~\ref{pro:W-preserves-qcom}.
	Therefore the structure map $\wnus{n}(X)\to S$ is quasi-compact,
	since $S$ is affine. (See SGA 4 VI 1.14~\cite{SGA4.2}, say.)

	(b): By (a) and \ref{pro:W-preserves-etale-local-properties}(a).
\end{proof}

\begin{proposition} \label{pro:wnus-preserves-separatedness}
	The following properties of algebraic spaces over $S$ are preserved by $\wnus{n}$:
		\begin{enumerate}
			\item quasi-separated over $S$,
			\item $0$-geometric over $S$ (see \ref{subsec:algebraic-spaces}),
			\item separated over $S$,
			\item separated.
		\end{enumerate}
\end{proposition}

\begin{proof}
	(a): Consider the diagram	
		\begin{equation} \label{eq:internal-diag-pink-house}
		\xymatrix{
		\coprod_{[0,n]} X \ar^-a[r]\ar_{c=\gh{\leq n}}[d]		
			& \bigl(\coprod_{[0,n]} X\bigr) \times_S \bigl(\coprod_{[0,n]} X\bigr)
				\ar^{b=\gh{\leq n}\times\gh{\leq n}}[d]\\
		\wnus{n}(X) \ar^-d[r]				
			& \wnus{n}(X) \times_S \wnus{n}(X),
		}
		\end{equation}
	where the horizontal maps are the diagonal maps.
	Because $X$ is quasi-separated, $a$ is quasi-compact, and
	by~\ref{cor:gh-is-surj-and-integral-geometric}, so is $b$.
	Therefore $b\circ a$ is quasi-compact, and hence
	so is $d\circ c$.  Now let $U$ be an affine scheme mapping to
	$\wnus{n}(X)\times_S \wnus{n}(X)$, and let $(V_i)_{i\in I}$ be
	an \'etale cover of the pull back $d^*(U)$.  Since $d\circ c$ is quasi-compact, 
	there is a finite subset $J\subseteq I$ such that $(c^*V_j)_{j\in J}$
	is a cover of $c^*d^*(U)$. In other words,
	the induced map $v\:\coprod_{j\in J}V_j\to d^*(U)$ 
	becomes surjective after base change by the map $c$. Because $c$ is surjective
	(\ref{cor:gh-is-surj-and-integral-geometric}), $v$ must also be.

	(b): 
	Recall that a map is $0$-geometric if and only if its diagonal map is affine  
	(\ref{subsec:algebraic-spaces}).
	This is equivalent to requiring the existence of an \'etale cover
	$(U_i)_{i\in I}$ of $X$, with each $U_i$ affine, such that $U_i\times_X U_j$ is affine.
	Fix such a cover of $X$. By~\ref{cor:etale-chart-presentation-of-wnus}, the family
	$\bigl(\wnus{n}(U_i)\bigr)_{i\in I}$ is an \'etale cover of $\wnus{n}(X)$.
	Therefore it is enough to show that 
	$\wnus{n}(U_i)\times_{\wnus{n}(X)}\wnus{n}(U_j)$ is affine, for all $i,j\in I$. 
	By~\ref{thm:W-main-geometric-finite-length(2)}(\ref{part:W-main-geometric-finite-length(2)-c}),
	this agrees with $\wnus{n}(U_i\times_X U_j)$.
	Because $X$ has affine diagonal, $U_i\times_X U_j$ is affine. Therefore, by
	\ref{pro:W_n-of-affine-is-affine}, so is $\wnus{n}(U_i\times_X U_j)$.
	
	(c):
	Let us first assume that $X$ is of finite type over $S$.
	Consider diagram~(\ref{eq:internal-diag-pink-house}) above.
	To show that $d$ is a closed immersion, it is enough to show that it is a finite monomorphism.
	(It is a general fact the a finite monomorphism of algebraic spaces is a closed immersion.
	To prove it, it is enough to work \'etale locally, which reduces us
	to the affine case, where it follows from Nakayama's lemma.)
	Since $d$ has a retraction, it is a monomorphism. Therefore it suffices to show that
	$d$ is finite.
	On the other hand, by~\ref{cor:qcom-and-ftype-over-S}(b), the structure map
	$\wnus{n}(X)\to S$ is of finite type, and hence so is $d$.
	Therefore it is enough to show that $d$ is integral.
	
	Since $X$ is separated, $a$ is a closed immersion and, in particular, is integral.
	Since $b$ is integral (by \ref{cor:gh-is-surj-and-integral-geometric}),
	$b\circ a$ is integral and, hence, so is
	$d\circ c$.  By part (b), the map $d$ is affine.
	Therefore, by~\ref{cor:gh-is-surj-and-integral-geometric}, 
	the maps $c$ and $d$ can be written
	\'etale-locally on $\wnus{n}(X)\times_S \wnus{n}(X)$ as
		\[
		\Spec C\longlabelmap{c} \Spec B\longlabelmap{d} \Spec A,
		\]
	where the induced ring map $B\to C$ has nilpotent kernel.
	We showed above that $C$ is integral over $A$.
	Therefore $B$ is integral over $A$. This proves $X$ is separated over $S$ when it is
	of finite type.

	Now consider the general case. Let us show that we can assume $X$ is quasi-compact. To
	prove that $d$ is a closed immersion, it is enough to work \'etale locally.
	Therefore, it is enough (by
	\ref{thm:W-main-geometric-finite-length(2)}(\ref{part:W-main-geometric-finite-length(2)-b}))
	to show that for any affine schemes $U,V$ with \'etale maps to $X$,
	the base change
		\[
		\wnus{n}(U)\times_{\wnus{n}(X)}\wnus{n}(V) \longlabelmap{d'}
			\wnus{n}(U)\times_S \wnus{n}(V)
		\]
	of $d$ is a closed immersion.
	By~\ref{thm:W-main-geometric-finite-length(2)}(\ref{part:W-main-geometric-finite-length(2)-c}),
	the source of $d'$ agrees with $\wnus{n}(U\times_X V)$.
	Therefore, $d'$
	does not change if we replace $X$ with the union of the images of $U$ and $V$.
	So in particular, we can assume $X$ is quasi-compact.
		
	Then there exists an affine $S$-map $h\:X\to X_0$, 
	where $X_0$ is some separated algebraic space of finite type over $S$. Indeed, 
	since $X$ is quasi-compact and separated, 
	by Conrad--Lieblich--Olsson~\cite{Conrad-Lieblich-Olsson},
	Theorem 1.2.2, there is an affine map $h'\:X\to X'_0$, where 
	$X'_0$ is a separated algebraic space of finite type over $\bZ$.
	Put $X_0=S\times_{\bZ}X'_0$.
	Then the induced map $h\:X\to X_0$ factors as 
		\[
		X \longmap S\times_{\bZ} X \longmap S\times_{\bZ} X'_0.
		\]
	The first map is a base change of the diagonal map $S\to S\times_{\bZ}S$, which is
	a closed immersion, since $S$ is separated. The second map is a base change of $h'$,
	which is affine. Therefore both maps in the factorization above 
	are affine, and hence so is the composition $h$.
	
	Since $X_0$ is of finite type over $S$, we can apply the argument above to see
	that $\wnus{n}(X_0)$ is separated. But $\wnus{n}(X)$ is affine over $\wnus{n}(X_0)$,
	by~\ref{pro:W-preserves-affine-local-properties}(\ref{pro:W-preserves-affine-local-properties:affine}). 
	Therefore $\wnus{n}(X)$ is also separated.
	
	(d): This follows from (c) because $S$ is assumed to be separated.
\end{proof}

\begin{proposition}\label{pro:W-preserves-target-local-rel-to-S}
	The following properties are preserved by $\wnus{n}$:
		\begin{enumerate}
			\item finite over $S$,
			\item faithfully flat over $S$.
		\end{enumerate}
\end{proposition}
\begin{proof}
	These are local properties on the target $S$. So,
	by~(\ref{eq:localization-of-S-commutes-with-wnus})
	and~(\ref{eq:geom-witt-iterate-upper-s}), we can write assume
	$S=\Spec R$ and $\ptst=\{\m\}$, for some ring $R$ and ideal $\m$. 
	Away from $\m$, the properties are clearly true.
	Therefore we only need to work locally
	near $\m$, and in particular we can assume $\m$ is not the unit ideal.  
	By~(\ref{eq:localization-of-S-commutes-with-wnus}) again,
	we may further assume $R$ agrees with $R_{\m}$, which is a discrete valuation ring.
	
	Let $X$ be an algebraic space over $S$ having the property in question.
	
	(a): Write $X=\Spec A$.	Then $W_n(A)$ is a subring of $A^{[0,n]}$,
	which is finite over $R$ because $A$ is.
	Since $R$ is a discrete valuation ring, this implies that $W_n(A)$ is finite over $R$.

	(b): The composition
		\[
		\entrymodifiers={+!!<0pt,\fontdimen22\textfont2>}
		\def\objectstyle{\displaystyle}		
		\xymatrix{
			\coprod_{[0,n]}X \ar^-{\gh{\leq n}}[r]
				& \wnus{n}(X) \ar[r]
				& S
		}
		\]
	is surjective because $X\to S$ is. Therefore the map $\wnus{n}(X)\to S$
	is surjective. It is flat by~\ref{pro:W-preserves-etale-local-properties}.
\end{proof}

\begin{corollary}\label{pro:W(S)-properties}
	The structure map $\wnus{n}(S)\to S$ finite and faithfully flat.
\end{corollary}

\spacesubsec{Relative properties}

\begin{proposition}\label{pro:W-preserves-target-local-properties}
	The following properties of maps (\'etale-local on the target)
	of algebraic spaces are preserved by $\wnus{n}$.
		\begin{enumerate}
			\item quasi-compact, 
			\item universally closed,
			\item quasi-separated,
			\item separated,
			\item surjective.
		\end{enumerate}
\end{proposition}
\begin{proof}
	Let $f\:X\to Y$ be the map in question.
	By~\ref{pro:maps-local-permanence}, 
	we may write $S=\Spec R$ and $\ptst=\{\m\}$ and we may assume that $Y$ is affine.
	
	(a): Since $f$ is quasi-compact and $Y$ is affine, $X$ is quasi-compact.
	Then $\wnus{n}(X)$ is quasi-compact, by~\ref{pro:W-preserves-qcom}.
	Since $\wnus{n}(Y)$ is affine (\ref{pro:W_n-of-affine-is-affine}), 
	$\wnus{n}(f)$ is quasi-compact. (SGA 4 VI 1.14~\cite{SGA4.2})
		
	(b): 	Consider the following square:
		\begin{equation} \label{diag:internal-heehaw}
			\xymatrix{
			\coprod_{[0,n]} X \ar^-{\gh{X}}[r]\ar_{\smcoprod f}[d]										
				& \wnus{n}(X) \ar^{\wnus{n}(f)}[d] \\
			\coprod_{[0,n]} Y \ar^-{\gh{Y}}[r]				
				& \wnus{n}(Y),
			}
		\end{equation}
	where $\gh{X}$ and $\gh{Y}$ denote the ghost maps $\gh{\leq n}$ for $X$ and $Y$.
	To show $\wnus{n}(f)$ is universally closed, it is enough to show
	that $\gh{X}$ is surjective and $\gh{Y}\circ \coprod f$ is universally closed.
	(See EGA II 5.4.3(ii) and 5.4.9~\cite{EGA-no.8}.)
	
	But we know $\gh{X}$ is surjective by~\ref{cor:gh-is-surj-and-integral-geometric};  
	and $\gh{Y}\circ f$ is universally closed because $f$ is universally
	closed and because $\gh{Y}$ is integral, by~\ref{cor:gh-is-surj-and-integral-geometric},
	and hence universally closed. (See EGA II 6.1.10~\cite{EGA-no.8}.)

	(c)--(d): Because $Y$ is affine, so is $\wnus{n}(Y)$, by~(\ref{eq:wnus-of-affine-is-W}). 
	Therefore being separated or quasi-separated over $\wnus{n}(Y)$ is equivalent to being
	so over $S$. Thus the results then follow from~\ref{pro:wnus-preserves-separatedness},
	
	(e): Consider diagram~(\ref{diag:internal-heehaw}).
	By~\ref{cor:gh-is-surj-and-integral-geometric}, the map $\gh{Y}$ is surjective.
	Since $f$ is too, so is $\coprod f$. Therefore $\gh{Y}\circ\coprod f$ and hence $\wnus{n}(f)$.
\end{proof}

\subsection{} {\em Flatness properties.}
\label{subsec:W-does-not-preserve-flatness-properties}
With the $p$-typical Witt vectors, say, $W_1(\bZ[x])$ is not flat over $W_1(\bZ)$. So if $P$ is a
property of morphisms which is stronger than flatness and which is satisfied by the map
$\bZ\to\bZ[x]$, then it is not generally preserved by $W_n$. Examples: flat, faithfully flat,
smooth, Cohen--Macaulay, and so on.

\begin{proposition}\label{pro:W-preserves-relative-fin-type}
	Let $f\:X\to Y$ be a map of algebraic spaces having one of the following properties:
		\begin{enumerate}
			\item locally of finite type,
			\item of finite type,
			\item finite,
			\item proper.
		\end{enumerate}
	Then $\wnus{n}(f)\:\wnus{n}(X)\to\wnus{n}(Y)$ has the same property, as long as
	$Y$ is locally of finite type over $S$.
\end{proposition}
\begin{proof}
	(a): The composition $X\to Y\to S$ is locally of finite type because the factors are.
	Therefore, $\wnus{n}(X)$ is locally
	of finite type over $S$, by~\ref{pro:W-preserves-etale-local-properties}.
	In particular, it is locally of finite type over $\wnus{n}(Y)$.

	(b): Part (a) above plus \ref{pro:W-preserves-target-local-properties}(a).
	
	(c): Part (b) above plus \ref{pro:W-preserves-affine-local-properties}(c).

	(d): Part (b) above plus \ref{pro:W-preserves-target-local-properties}(b),(e).
\end{proof}

\subsection{} {\em $\wnus{n}$ and properness.}
\label{subsec:W-does-not-preserve-properness}
Some hypotheses on $Y$ are needed in~\ref{pro:W-preserves-relative-fin-type}. For example, let $f$
be the canonical projection $\bP^1_Y \to Y$, where $Y=\Spec \bZ[x_1,x_2,\dots]$. Then $f$ is
proper, but $\wnus{1}(f)$ (with $p$-typical Witt vectors, say) is not,
because it is not of finite type. Indeed, the map
$\wnus{1}(\bA^1_Y)\to\wnus{1}(\bP^1_Y)$ is \'etale~(\ref{thm:W-main-geometric-finite-length(2)}),
and hence of finite type, but the map $\wnus{1}(\bA^1_Y)\to\wnus{1}(Y)$ is
not~(\ref{subsec:W-does-not-preserve-relative-fg}).

\spacesubsec{Depth properties}

\begin{proposition}\label{pro:W-of-local-ring}
	Suppose that $S=\Spec R$ for some ring $R$ 
	and that $\ptst$ consists of a single maximal ideal $\m$ of $R$.
	Let $A$ be a local $R$-algebra whose maximal ideal contains $\m$.
	Then $W_n(A)$ is a local ring with maximal ideal $\gh{0}^{-1}(\m)$, where
	$\gh{0}$ denotes the usual projection map $W_n(A)\to A$.
\end{proposition}
\begin{proof}
	Let $I$ denote $\gh{0}^{-1}(\m)$; it is a maximal ideal because
	$\gh{0}$ is surjective (1.21). Let us show that it is the unique maximal 
	ideal.
	
	Let $J$ be a maximal ideal of $W_n(A)$.
	By~8.2, the map
		\[
		\gh{\leq n}:W_n(A) \longmap A^{[0,n]}
		\]  
	is integral and its kernel is nilpotent. Therefore, $J$
	is the pre-image of a maximal ideal of $A^{[0,n]}$.
	But every maximal ideal of $A^{[0,n]}$ contains $\m$, because the maximal ideal of $A$ does.
	Therefore $J$ contains $\m$.
	But $I$ is the only maximal ideal of $W_n(A)$ containing $\m$, because
	by~6.8, every element of $I$ is nilpotent modulo $\m W_n(A)$.
	Therefore $J=I$.
\end{proof}

\begin{proposition}\label{pro:local-rings-of-W}
	Let $S,R,\ptst,\m$ be as in \ref{pro:W-of-local-ring}. Let $A$ be an $R$-algebra,
	and let $\p$ be a prime ideal of $W_n(A)$.  
	\begin{enumerate}
		\item If $\p$ does not contain $\m$,
			then there is a unique integer $i\in[0,n]$ and a unique prime ideal $\q$ of $A$ such 
			that $\gh{i}^{-1}(\q)=\p$. For this $i$ and $\q$, the map 
				\begin{equation} \label{map:local-rings-of-W-map-1}
					A_{\q}\longmap W_n(A)_{\p}
				\end{equation}
			induced by $\gh{i}$ is an isomorphism.
		\item If $\p$ does contain $\m$, then there is a unique prime ideal $\q$ of $A$
			such that $\gh{0}^{-1}(\q)=\p$. For this $\q$, there is a unique map of 	
			$W_n(A)$-algebras
				\begin{equation} \label{map:local-rings-of-W-map-2}
					W_n(A)_{\p}\longmap W_n(A_{\q}),
				\end{equation}
			and this map is an isomorphism.
	\end{enumerate}	
\end{proposition}
\begin{proof}
	(a): This holds because the map $\gh{\leq n}$ is an isomorphism away from $\m$, 
	by~6.1.
	
	(b): Any such prime ideal $\q$ contains $\m$. Therefore to show such a prime
	ideal $\q$ exists and is unique, it is enough to show the map
		\[
		\id\tn\gh{0}\:R/\m\tn_R W_n(A) \longmap R/\m\tn_R A
		\]
	induces a bijection on prime ideals. This holds because $\id\tn\gh{0}$
	is surjective with nilpotent kernel, by~8.2. 
	
	Now consider the diagram
		\[
		\xymatrix{
			A \ar[r] & A_{\q} \\
			W_n(A) \ar[r]\ar^{\gh{0}}[u] & W_n(A_{\q}). \ar_{\gh{0}}[u]
		}
		\]
	By~\ref{pro:W-of-local-ring}, the ring $W_n(A_{\q})$ is a local ring. 
	So, to show there is a unique map of $W_n(A)$-algebras as in 
	(\ref{map:local-rings-of-W-map-2}),	it is enough to that $\p$ is the 
	pre-image in $W_n(A)$ of the maximal ideal $\gh{0}^{-1}(\m)$ of $W_n(A_{\q})$. This holds
	by the commutativity of the diagram above.

	Now let us show that (\ref{map:local-rings-of-W-map-2}) is an isomorphism.
	By induction, we may assume that the map
		\[
		W_{n-1}(A)_{\p} \longmap W_{n-1}(A_{\q})
		\]
	is an isomorphism. By~4.4, it is therefore enough to show that the maps
		\[
		A_{(n)} \tn_{W_n(A)} W_n(A)_{\p} \longmap (A_{\q})_{(n)}
		\]
	are isomorphisms.  
	Write $T_{\p}=W_n(A)-\p$ and $T_{\q}=A-\q$.
	Then this map can be identified with the following map of localizations:
		\begin{equation} \label{eq:internal-moose}
			\gh{n}(T_{\p})^{-1} A \longmap T_{\q}^{-1}A.
		\end{equation}
	Thus it is enough to show that $T_{\q}$ becomes invertible in $\gh{n}(T_{\p})^{-1} A$.
	It is therefore enough to show $(T_{\q})^{q_\m^n}\subseteq \gh{n}(T_{\p})$, which
	holds by the following.
	
	We have $[T_{\q}]\subseteq T_{\p}$, where $[\vbl]$ denotes the Teichm\"uller section
	(of 1.21); indeed, if $[x]\in \p = \gh{0}^{-1}(\q)$, then
	$x=\gh{0}([x])\in\q$. Therefore, we have
		\[
		\gh{n}(T_{\p}) \supseteq \gh{n}([T_{\q}]) = (T_{\q})^{q_\m^n}.
		\]
\end{proof}

\subsection{} \emph{Remark.}
In either case of \ref{pro:local-rings-of-W}, the prime ideal $\q$ is the pre-image of $\p$ under
the Teichm\"uller map $t\:A\to W_n(A)$, $t(a)=[a]$.
In fact, the induced function $\text{``$\Spec t$''}\:\Spec
W_n(A)\to \Spec A$ is continuous in the Zariski topology. Even further,
if we view the structure sheaves on these schemes as sheaves of commutative monoids under
multiplication, then the usual Teichm\"uller map gives $\text{``$\Spec t$''}$ the structure of a
map of locally monoided topological spaces. If $R$ is an $\bF_p$-algebra, for some prime number
$p$, then $t$ is a ring map, and $\text{``$\Spec t$''}$ agrees with the scheme map $\Spec t$; but
otherwise $\text{``$\Spec t$''}$ will not generally be a scheme map. These statements can in
fact be promotied to the \'etale topology, but there it is necessary to work with maps of toposes 
instead of maps of topological spaces.

\begin{proposition}\label{pro:teich-of-regular-sequence}
	Let $S,R,\ptst,\m,A$ be as in \ref{pro:W-of-local-ring}.
	Let $a_1,\dots,a_d\in\m_A$ be a regular sequence for $A$.
	Then $[a_1],\dots,[a_d]$ lie in the maximal ideal of $W_n(A)$
	and form a regular sequence for $W_n(A)$.
\end{proposition}

\begin{proof}
	By \ref{pro:W-of-local-ring}, the maximal ideal of $W_n(A)$ is $\gh{0}^{-1}(\m_A)$, which
	contains the sequence $[a_1],\dots,[a_d]$. It remains to show that the sequence is regular.
	
	By~(\ref{eq:localization-of-S-commutes-with-wnus}), we can assume $R$ agrees with $R_{\m}$, and
	hence that the ideal $\m$ is generated by an element $\pi$.
	The argument will now go by induction on $n$.  For $n=0$, there is nothing to prove;
	so assume $n\geq 1$.
	For any $W_n(A)$-module $M$, let $K_{W_n(A)}(M)$ denote the Koszul complex of
	$M$ with respect the sequence $[a_1],\dots,[a_d]\in W_n(A)$.
	If this sequence is regular for $M$, then $H^{d-1}(K_{W_n(A)}(M))=0$, and the converse
	holds if $M$ is finitely generated and nonzero.
	(See Eisenbud~\cite{Eisenbud:Commutative-algebra-book}, Corollary 17.5 and Theorem 17.6.)

	In particular, it is enough to show $H^{d-1}(K_{W_n(A)}(W_n(A)))=0$.
	Considering the exact sequence (4.4.1) of $W_n(A)$-modules
		\[
		0 \longmap A_{(n)} \longlabelmap{V_{\pi}^n} W_n(A) \longmap W_{n-1}(A) \longmap 0,
		\]
	we see it is even sufficient to show 
		\[
		H^{d-1}(K_{W_n(A)}(A_{(n)})) = H^{d-1}(K_{W_n(A)}(W_{n-1}(A))) = 0.
		\]
	Observe that we have
		\[
		H^{d-1}\bigl(K_{W_{n}(A)}(W_{n-1}(A))\bigr)= H^{d-1}\bigl(K_{W_{n-1}(A)}(W_{n-1}(A))\bigr)= 0,
		\]
	by induction. Therefore, it is enough to prove
	$H^{d-1}\bigl((K_{W_n(A)}(A_{(n)}))\bigr)=0$, and hence that $[a_1],\dots,[a_d]$ is regular for 
	$A_{(n)}$.
	This is equivalent to the sequence $a_1^{q^n},\dots,a_d^{q^n}\in A$ being regular for 
	$A$---indeed, the product $[a]\cdot x$, where $x\in A_{(n)}$, is by definition
	$\gh{n}([a])x$, which equals $a^{q^n}x$.
	We complete the argument with the general fact that any power of a regular sequence
	for a finitely generated module is
	again regular (\cite{Eisenbud:Commutative-algebra-book}, Corollary 17.8).
\end{proof}

\begin{proposition}\label{pro:W-preserves-depth-bound}
	Let $k$ be an integer. Let $X$ be an algebraic space over $S$ such that both $X$ and 
	$\wnus{n}(X)$ are locally noetherian. Suppose $X$ satisfies one of the following properties:
	\begin{enumerate}
		\item Cohen--Macaulay,
		\item Cohen--Macaulay over $S$,
		\item $\text{S}_k$ (Serre's condition),
		\item $\text{S}_k$ over $S$.
	\end{enumerate} 
	Then $\wnus{n}(X)$ satisfies the same property.
\end{proposition}

\subsection{} \emph{Remark.}
See EGA IV (5.7.1), (6.8.1) \cite{EGA-no.24} for the definition of Cohen--Macaulay and $\text{S}_k$.
Typically, these concepts are discussed only for noetherian rings, but because $\wnus{n}$ does not
preserve noetherianness, we must assume $\wnus{n}(X)$ is noetherian. I do not know if it is possible
to remove this assumption by extending the concept of depth beyond the noetherian setting. If so,
maybe even the noetherian hypotheses on $X$ could be removed.

Note that, by~\ref{pro:W-preserves-etale-local-properties}(a), the assumptions hold if $S$ is
noetherian and $X$ is locally of finite type over $S$.

\subsection{} \emph{Proof of \ref{pro:W-preserves-depth-bound}.}
	The properties are all \'etale-local (EGA IV (6.4.2) \cite{EGA-no.24}).
	So by~\ref{pro:obj-general-permanence-properties}, 
	we can write $S=\Spec R$, $\ptst=\{\m\}$, $\m=\pi R$, and $X=\Spec A$.

	(a)--(b): These follow from (c) and (d).
	
	(c): At a prime ideal of $W_n(A)$ not containing $\pi$, the local ring agrees with 
	a local ring of $A$ (by \ref{pro:local-rings-of-W}), which satisfies $\text{S}_k$
	by assumption; so there is nothing to prove.
	Now let $\p$ be a prime ideal of $W_n(A)$ containing $\pi$.
	Let $\q$ be the corresponding prime ideal of $A$ given by~\ref{pro:local-rings-of-W}.
	Then we have $W_n(A)_{\p}=W_n(A_{\q})$, 
	so it suffices to assume that $A$ is a local ring with maximal ideal $\q$.
	We can therefore also assume that $R$ is a discrete valuation ring with maximal ideal 
	
	By~\ref{pro:teich-of-regular-sequence} and~\ref{pro:W-preserves-etale-local-properties},
	we have	
		\[
		\depth W_n(A)\geq \depth A, \quad \dim W_n(A)=\dim A.
		\]
	By the definition of $\text{S}_k$, $\depth A$ is at least $k$ or $\dim A$.
	Therefore $\depth W_n(A)$ is at least $k$ or $\dim W_n(A)$. In other words, $W_n(A)$
	also satisfies $\text{S}_k$.
		
	(d): By~\ref{pro:W-preserves-etale-local-properties}, $\wnus{n}(A)$ is flat over $R$.
	Away from $\m$, we have $\wnus{n}(A)=A^{[0,n]}$, the fibers over $S$ of which
	satisfy $\text{S}_k$. So it suffices to consider the fiber over $\m$.
	Therefore, by~\ref{pro:local-rings-of-W}, we can assume that
	$A$ is a local $R$-algebra whose maximal ideal contains $\m$.
	By~(\ref{eq:localization-of-S-commutes-with-wnus}), we can further assume that $R$ is a 
	discrete valuation ring. We need to show that $W_n(A)/\pi W_n(A)$ satisfies $\text{S}_k$.  

	Since $A$ and $W_n(A)$ are flat over $R$, 
	the element $\pi$ is not a zero divisor in $A$ or $W_n(A)$. Therefore we have
		\[
		\depth W_n(A)/\pi W_n(A) = \depth W_n(A) - 1 \geq \depth A - 1 = \depth A/\pi A,
		\]
	by~\ref{pro:teich-of-regular-sequence} (and EGA 0 (16.4.6)(ii)~\cite{EGA-no.20}, say).
	Further, by~\ref{pro:W-preserves-etale-local-properties}, we have
		\[
		\dim W_n(A)/\pi W_n(A) = \dim W_n(A) -1 =\dim A - 1 = \dim A/\pi A.
		\]
	Because $A/\pi A$ satisfies $\text{S}_k$, $\depth A/\pi A$ is at least $k$ or
	$\dim A/\pi A$. Therefore $\depth W_n(A)/\pi W_n(A)$ is at least $k$ or
	$\dim W_n(A)/\pi W_n(A)$. In other words, the fiber $W_n(A)/\pi W_n(A)$ satisfies $\text{S}_k$.
\qed

\subsection{} {\em Gorenstein, regular, normal.}
\label{subsec:W-of-Gorenstein}
Consider the $p$-typical Witt vectors.
Then we have
	\begin{equation} \label{eq:presentation-of-W(R)}
		W_n(\bZ) = \bZ[x_1,\dots,x_{n}]/(x_ix_j-p^ix_j\,|\,1\leq i\leq j\leq n),
	\end{equation}
with the element $x_i$ corresponding to $V_{p}^i(1)$, where $V_{p}$ denotes the
usual Verschiebung operator. 
(See 3.8.)

This presentation gives some easy counterexamples.
The ring $W_1(\bZ)$ agrees with $\bZ[x]/(x^2-px)$, which is not normal.
So $W_n$ does not generally preserve regularity or normality.

The property of being Gorenstein is also not preserved by $W_n$. Indeed, 
we have
	\[
	\bF_p\tn_{\bZ} W_n(\bZ) = \bF_p[x_1,\dots,x_{n}]/(x_ix_j\,|\,1\leq i\leq j\leq n).
	\]
Therefore the socle of $\bF_p\tn_{\bZ} W_n(\bZ)$ (that is, the 
annihilator of its maximal ideal) is the vector space
	\[
	\bF_p x_1 \oplus \cdots \oplus \bF_p x_{n} \quad \text{if } n\geq 1,
	\]
and is $\bF_p$ if $n=0$. Since the sequence $\{p\}$ of length $1$ is a system of
parameters in $W_n(\bZ)$ at the prime ideal $\p$ containing $p$, the ring
$W_n(\bZ)$ is Gorenstein at $\p$ if and only if the
dimension of the socle is $1$. This holds if and only if $n=0,1$. When $n=1$, it is even a
complete intersection, but it is not normal. (A basic
treatment of these concepts is in Kunz's book~\cite{Kunz:IntroCAAG} (VI 3.18), for example.)

\section{Ghost descent and the geometry of Witt spaces}
\label{sec:ghost-descent-and-the-geometry-of-Witt-spaces}

The purpose of this section is to describe the Witt space $\wnus{n}(X)$ of a flat algebraic space 
$X$ as a certain quotient, in the category of algebraic spaces, of the ghost space 
$\coprod_{[0,n]}X$.
We continue with the notation of~\ref{subsec:supramaximal-def-global-base-S}.

\subsection{} {\em Reduced ghost components.}
Suppose $\ptst$ consists of one ideal $\m$, consider the diagram
	\begin{equation} \label{eq:thm:presentation-of-W-of-alg-space}
		\displaylonglabelcofork{S_{n} \times_S X}
			{i_1\circ\rgh{n+1}}{\bar{\imath}_2}
			{\wnus{n}(X) \smcoprod X}
			{\alpha_n}{\wnus{n+1}(X),}
	\end{equation}
where 
	\begin{enumerate}
		\item[] $\rgh{n+1}\:S_n\times_S X \to \wnus{n}(X)$ is as 
			in~(\ref{map:general-reduced-ghost-map}), 
		\item[] $i_1\:\wnus{n}(X)\to\wnus{n}(X) \smcoprod X$ is the inclusion into the first 
			component,
		\item[] $\bar{\imath}_2$ is the closed immersion of 
			$S_n\times_S X$ into the second component, and
		\item[] $\alpha_n$ is $\inclmap{n,1}$ on $\wnus{n}(X)$ and $\gh{n+1}$ on $X$,
			in the notation of~(\ref{map:wnus-change-of-n-inclusion-map}).
	\end{enumerate}
When $X$ is affine, this is the same as the diagram in~(8.1.1).

\begin{proposition}\label{pro:ghost-descent-for-spaces}
	Let $X$ be an algebraic space over $S$.
	Then the map $\alpha_n$ is an effective descent	map for the fibered category of
	algebraic spaces which are both \'etale and affine over their base.
	In this case, descent data is equivalent to gluing data with respect to the diagram
	(\ref{eq:thm:presentation-of-W-of-alg-space}).
\end{proposition}
\begin{proof}
	Given an \'etale map $U\to X$ with $U$ affine, consider the following three categories:
	the category of affine \'etale algebraic spaces over
	$\wnus{n+1}(U)$, that of affine \'etale algebraic spaces over $\wnus{n}(U)\smcoprod U$
	with descent data with respect to $\alpha_n$, and that of affine \'etale 
	algebraic spaces over $\wnus{n}(U)\smcoprod U$ with gluing 
	data with respect to the diagram (\ref{eq:thm:presentation-of-W-of-alg-space}). As $U$ varies, 
	there are obvious transition functors, and these give rise to
	three fibered categories over the small \'etale topology of $X$.

	There are also evident morphisms between these fibered categories,
	and the statement of the corollary is that, for $U=X$, these morphisms are equivalences.

	By~\ref{thm:W-main-geometric-finite-length(1)}, all these fibered categories satisfy effective 
	descent in the \'etale topology. Thus it is enough (by 
	Giraud~\cite{Giraud:Cohomologie-non-abelienne}, II 1.3.6, say)
	to assume $X$ is affine, in which case the equivalence follows 
	from~8.3.
\end{proof}

\begin{theorem}\label{thm:presentation-of-W-of-alg-space}
	If $X$ is $\m$-flat (\ref{subsec:definition-of-E-flat}),
	then (\ref{eq:thm:presentation-of-W-of-alg-space}) 
	is a coequalizer diagram in the category of algebraic spaces.
\end{theorem}

\begin{proof}
	For any space $Z$ over $S$, write $Z_n=S_{n} \times_S Z$.

	Let us first reduce to the case where $X$ is affine.
	Write
		\begin{equation} \label{eq:internal-dkd-1}
			X=\colim_{i\in I} U_i,
		\end{equation}
	where $(U_i)_{i\in I}$ is a diagram of affine schemes mapping by \'etale maps to $X$.
	Then $\bigl((U_i)_n\bigr)_{i\in I}$ is a diagram of affine schemes mapping by \'etale
	maps to $X_n$. We also have 
		\begin{equation} \label{eq:internal-dkd-2}
			X_n=\colim_i (U_i)_n,
		\end{equation}
	because the functor 
	$S_n\times_S \vbl\:\Space_S\to\Space_{S_n}$ has a right adjoint, and hence
	preserves colimits.
	In particular, both colimit formulas (\ref{eq:internal-dkd-1})
	and (\ref{eq:internal-dkd-2}) hold in the category of algebraic spaces, as well as in 
	$\Space_S$. Therefore, assuming the theorem in the affine case, we can make the
	following formal computation in the category of algebraic spaces:
	\begin{align*}
		\coeq \bigl[X_n\rightrightarrows \wnus{n}(X) \smcoprod X \bigr] 
			&= \coeq \bigl[\colimm_i(U_i)_n\rightrightarrows 
				\colimm_i\wnus{n}(U_i) \smcoprod \colimm_i U_i \bigr] \\
			&= \coeq \bigl[\colimm_i(U_i)_n\rightrightarrows 
				\colimm_i\left(\wnus{n}(U_i) \smcoprod U_i \right) \bigr] \\
			&= \colimm_i\coeq \bigl[(U_i)_n\rightrightarrows 
				\wnus{n}(U_i) \smcoprod U_i \bigr] \\
			&= \colimm_i\wnus{n+1}(U_i) \\
			&= \wnus{n+1}(\colimm_i U_i) \\
			&= \wnus{n+1}(X).
	\end{align*}
	Hence we can assume $X$ is affine.

	Let $Y$ be an algebraic space, and let
	$d\:\wnus{n-1}(X)\smcoprod X\to Y$ be a map such that the two compositions
	in the diagram
		\[
			\displaylonglabelcofork{S_{n} \times_S X}
				{i_1\circ\rgh{n+1}}{\bar{\imath}_2}
				{\wnus{n}(X) \smcoprod X}
				{d}{Y}
		\]
	agree. We want to show that $d$ factors through $\alpha_n$.  
	Because $\wnus{n}(X) \smcoprod X$ is affine, by~\ref{pro:W_n-of-affine-is-affine}, there is
	a quasi-compact open algebraic subspace containing the image of $d$.
	Since we can replace $Y$ with it, we may assume $Y$ is quasi-compact.
	
	Take $m\geq-1$ such that $Y\in\AlgSp_m$. We will argue by induction on $m$.
	When $m=-1$, the space $Y$ is a quasi-compact disjoint union of affine schemes;
	therefore it is affine.
	The result then follows because (\ref{eq:thm:presentation-of-W-of-alg-space})
	is a coequalizer diagram in the category of affine schemes, by~8.1.
	
	Now suppose $m\geq 0$.
	Let $e\:Y'\to Y$ be an \'etale surjection, where $Y'$ is an affine scheme.
	Then there is a \'etale surjection $g\:X'\to X$, with $X'$ affine, such that $d$ lifts 
	to a map $d'$ as follows:
		\[
		\xymatrix{
		\wnus{n}(X')\smcoprod X \ar@{-->}^-{d'}[r]\ar@{->>}_{\wnus{n}(g)\smcoprod g}[d] 
			& Y' \ar^{e}[d] \\
		\wnus{n}(X)\smcoprod X \ar^-{d}[r] 
			& Y.
		}
		\]
	Indeed, the existence of such a map $d'$
	is equivalent to the existence of a lift $f'$
		\[
		\xymatrix{
		X' \ar@{-->}^-{f'}[r]\ar@{->>}^{g}[d] 
			& \wnls{n}(Y')\times_S Y' \ar^{\wnls{n}(e)\times e}[d] \\
		X  \ar^-{f}[r] 
			& \wnls{n}(Y)\times_S Y,
		}
		\]
	where $f$ is the left adjunct of $d$.
	This exists because $\wnls{n}(e)\times e$ is an epimorphism of spaces, which is true
	by~\ref{pro:wnls-preserves-epis-and-monos}.

	Now let us construct the diagram in figure~\ref{fig:ghost-descent}. The
	rows are diagrams of the form (\ref{eq:thm:presentation-of-W-of-alg-space}); 
	to get $d''$, take the product of $d'$ with itself over $\wnus{n}(X)\smcoprod X$
	and then apply
	\ref{thm:W-main-geometric-finite-length(2)}(\ref{part:W-main-geometric-finite-length(2)-c}). 
	The maps $a,a',a''$ have not been constructed yet.
	
	\begin{figure}[t] 
		\begin{center}
			\[
			\xymatrix@C=70pt@R=40pt@!0{
			X'_n\times_{X_n}X'_n \ar@<0.7ex>[rr]\ar@<-0.7ex>[rr] 
							\ar@<0.7ex>[dd]^-{\pr_2}\ar@<-0.7ex>_-{\pr_1}[dd] &
				&
				\wnus{n}(X'\times_X X')\smcoprod (X'\times_X X') \ar^-{c''}[rr]\ar^{d''}[dr]						
						\ar@<0.7ex>[dd]^(.69){\wnus{n}(\pr_2)\smcoprod\pr_2}
						\ar@<-0.7ex>_(.69){\wnus{n}(\pr_1)\smcoprod\pr_1}[dd] &
				&
				\wnus{n+1}(X'\times_X X') \ar@{-->}^{a''}[dl]	
							\ar@<0.7ex>[dd]^(.69){\wnus{n+1}(\pr_2)}
							\ar@<-0.7ex>_(.69){\wnus{n+1}(\pr_1)}[dd] \\
			& & & Y'\times_Y Y' \ar@<0.7ex>[dd]^(0.37){\pr_2}\ar@<-0.7ex>_(0.37){\pr_1}[dd]\\
			X'_n 	\ar@<0.7ex>[rr]\ar@<-0.7ex>[rr] \ar^{g_n}[dd] &
				&
				\wnus{n}(X')\smcoprod X' 
					\ar_{\wnus{n}(g)\smcoprod g}[dd]\ar^{d'}[dr]\ar'[r][rr]_-{c'} &
				&
				\wnus{n+1}(X') \ar^{\wnus{n+1}(g)}[dd]\ar@{-->}^{a'}[dl] \\
			& & & Y'\ar_(.3){e}[dd] \\
			X_n \ar@<0.7ex>[rr]\ar@<-0.7ex>[rr] &
				&
				\wnus{n}(X)\smcoprod X \ar^{d}[dr]\ar'[r][rr]^-{c} &
				&
				\wnus{n+1}(X) \ar@{-->}^{a}[dl] \\
			& & & Y 
			}
			\]	
			\end{center}
		\caption{}
		\label{fig:ghost-descent}
	\end{figure}

	Since $X'$ and $X$ are affine, so is $X'\times_X X'$; and since $X'$ is \'etale over $X$,
	which is $\m$-flat, $X'$ and $X'\times_X X'$ are also $\m$-flat.
	Also, since $Y\in\AlgSp_m$, we have
		\[
		Y',Y'\times_Y Y'\in\AlgSp_{m-1}.
		\]
	So, by induction, there are unique maps $a'$ and $a''$ such that $d'=a'\circ c'$
	and $d''=a''\circ c''$.
	
	Now let us show $a'\circ\wnus{n+1}(\pr_i)=\pr_i\circ a''$, for $i=1,2$. 
	It is enough to show
		\[
		a'\circ \wnus{n+1}(\pr_i)\circ c'' = \pr_i\circ a''\circ c''.
		\]
	Indeed, by induction the coequalizer universal property holds for the top row.
	Showing this equality is a straightforward diagram chase.

	Therefore we have 
		\[
		e\circ a'\circ\wnus{n+1}(\pr_1) = e\circ \pr_1\circ a'' = 
			e\circ \pr_2\circ a'' = e\circ a'\circ\wnus{n+1}(\pr_2)
		\]
	So, by the universal property of coequalizers applied to the rightmost column,
	there exists a unique map $a$ such that 
		\begin{equation} \label{eq:internal-wjdj}
			a\circ \wnus{n+1}(g) = e\circ a'.
		\end{equation}
	Finally, let us verify the equality $d=a\circ c$.
	Because $\wnus{n}(g)\smcoprod g$ is an epimorphism, it is enough to show
		\[
		d\circ \bigl(\wnus{n}(g)\smcoprod g\bigr) = a\circ c \circ \bigl(\wnus{n}(g)\smcoprod g\bigr).
		\]
	This follows from (\ref{eq:internal-wjdj}) and a diagram chase which is
	again left to the reader.
\end{proof}

\subsection{} \emph{Remark.}
It is typically not true that~(\ref{eq:thm:presentation-of-W-of-alg-space}) 
is a coequalizer diagram in the category $\Space_S$.
For example, if we take $X=\Spec\bZ_p[\sqrt{p}]$ and consider the usual, $p$-typical Witt vectors, 
then $\alpha_1$ is not an epimorphism in $\Space_S$.

\section{The geometry of arithmetic jet spaces}
\label{sec:geometry-of-arithmetic-jet-spaces}

The main purpose of the section is to prove~\ref{thm:blow-up-description-of-w-lower-s}.
We continue with the notation of~\ref{subsec:supramaximal-def-global-base-S}. Let $X$ be an
algebraic space over $S$.

\subsection{} \emph{Single-prime notation.}
\label{subsec:affine-blow-up-notation}
Suppose that $\ptst$ consists of one maximal ideal $\m$.  Let 
	\[
	\wnls{n+1}(X)\longlabelmap{f}\wnls{n}(X)\times_S X
	\]
denote the map $(\projmap{n,1},\cgh{n+1})$ (in the notation of~\ref{subsec:natural-maps}), and let 
$I$ denote the ideal sheaf of $\sO_{\wnls{n}(X)\times_S X}$ defining the closed immersion
	\begin{equation} \label{eq:co-ghost-map}
		(\pr_2,\rcgh{n})\:S_n\times_S \wnls{n}(X) \longmap \wnls{n}(X)\times_S X.
	\end{equation}
Let $\sB$ denote the sub-$\sO_{\wnls{n}(X)\times_S X}$-algebra 
of $\sO_{S'}\tn_{\sO_S}\sO_{\wnls{n}(X)\times_S X}$ generated by the subsheaf
$\m^{-n-1}\tn_{\sO_S} I$. Observe that $\sB$ is $\m$-flat and satisfies 
	\begin{equation} \label{eq:B-congruence-condition}
		\m^{n+1}\sB \supseteq \m^{n+1}(\m^{-n-1}\tn_{\sO_S}I)\sB = I \sB.
	\end{equation}
When necessary, we will write $f_X,I_X,\sB_X$ to be clear.

\begin{proposition}\label{pro:blow-up-description-of-w-lower-s}
	Suppose that $\ptst$ consists of one maximal ideal $\m$.
	Let $T$ be an $\m$-flat algebraic space over $\wnls{n}(X)\times_S X$.
	Then there exists at most one map $\tilde{g}$ 
		\[
		\xymatrix{
		T \ar^{g}[dr] \ar@{-->}^-{\tilde{g}}[r]
			& \wnls{n+1}(X) \ar^{f}[d] \\
			& \wnls{n}(X)\times_S X
		}
		\]
	lifting the structure map $g$. Such a lift
	exists if and only if $I\sO_T\subseteq \m^{n+1}\sO_T$. 
\end{proposition}
\begin{proof}
	Giving a map $\tilde{g}\:T\to\wnls{n+1}(X)$ is equivalent to giving
	a map $\wnus{n+1}(T)\to X$. Such maps can be described using the diagram
		\[
		\displaycofork{S_n\times_S T}{\wnus{n}(T)\smcoprod T}{\wnus{n+1}(T),}
		\]
	because it is a coequalizer diagram in the category
	of algebraic spaces, by~\ref{thm:presentation-of-W-of-alg-space}.
	Therefore giving a map $T\to\wnls{n+1}(X)$ is equivalent to
	giving maps $a\:T\to\wnls{n}(X)$ and $b\:T\to X$ such that
	the diagram
		\begin{equation*} \label{diag:proof-of-geom-desc-of-wnls-3}
			\xymatrix@C=35pt{
			S_n\times_S T \ar^-{\pr_2}[r] \ar_{\rgh{n}}[d] 
				& T \ar^{b}[d] \\
			\wnus{n}(T) \ar^-{a'}[r] 
				& X,
			}
		\end{equation*}
	where $a'$ is the left adjunct of $a$, commutes. 
	The commutativity of this diagram is equivalent to that of  
		\begin{equation} \label{diag:proof-of-geom-desc-of-wnls-4}
			\xymatrix@C=40pt{
			S_n\times_S T \ar^-{\pr_2}[r] \ar_{\id\times a}[d] 
				& T \ar_{b}[d] \\
			S_n\times_S \wnls{n}(X) \ar^-{\rcgh{n}}[r] 
				& X.
			}
		\end{equation}
	This is because the following diagram commutes:
		\[
		\xymatrix@C=35pt{
		S_n\times_S T \ar^-{\id\times a}[r]\ar^{\rgh{n}}[d] 
			& S_n\times_S \wnls{n}(X) \ar^{\rgh{n}}[d] \ar^{\rcgh{n}}[ddr] \\
		\wnus{n}(T) \ar^-{\wnus{n}(a)}[r] \ar^{a'}[drr]
			& \wnus{n}\wnls{n}(X) \ar^{\varepsilon}[dr] \\
		&	& X,
		}
		\]
	where $\varepsilon$ is the counit of the evident adjunction.
	(And this diagram commutes by the naturalness of $\rgh{n}$ and the definitions of $a'$
	and $\rcgh{n}$.)
	
	Let us now apply this in the case where we take $(a,b)$ to be $g$.
	Then the map $\tilde{g}$ required by the lemma is unique, and it
	exists if and only if (\ref{diag:proof-of-geom-desc-of-wnls-4}) commutes.
	The commutativity
	of (\ref{diag:proof-of-geom-desc-of-wnls-4}) is equivalent to that of
		\begin{equation} \label{diag:proof-of-geom-desc-of-wnls}
		\xymatrix{
		S_n\times_S T \ar^-{\pr_2}[r] \ar_{\id\times (\pr_1\circ g)}[d] 
			& T \ar_{g}[d] \\
		S_n\times_S \wnls{n}(X) \ar^-{h}[r] 
			& \wnls{n}(X)\times_S X,
		}
		\end{equation}	
	where $h$ denotes the map $(\pr_2,\rcgh{n})$ of~(\ref{eq:co-ghost-map}).
	Because $h$ is a closed immersion,
	the commutativity of (\ref{diag:proof-of-geom-desc-of-wnls}) is
	equivalent to requiring 
	that the ideal $I$ defining $h$
	pull back to the zero ideal on $S_n\times_S T$, which is equivalent to the
	containment $I\sO_T\subseteq \m^{n+1}\sO_T$.
\end{proof}

\begin{theorem}\label{thm:blow-up-description-of-w-lower-s}
	Suppose that $\ptst$ consists of one maximal ideal $\m$ and that
	$\wnls{n+1}(X)$ is $\m$-flat (\ref{subsec:definition-of-E-flat}). 
	Let $\sB$ be as in~\ref{subsec:affine-blow-up-notation}.
	Then the unique $(\wnls{n}(X)\times_S X)$-map
		\[
		\underSpec \sB \longlabelmap{\tilde{g}} \wnls{n+1}(X),
		\]
	of \ref{pro:blow-up-description-of-w-lower-s} is an isomorphism.
\end{theorem}

\subsection{} \emph{Remark.}
In other words, we have
	\begin{equation}
		\wnls{n+1}(X) = \underSpec \sO_{\wnls{n}(X)\times_S X}[\m^{-n-1}\tn_{\sO_S}I],
	\end{equation}
which gives a concrete recursive description of $\wnls{n+1}(X)$ when it is $\m$-flat.
Note that, by~\ref{cor:smooth-implies-wnls-flat}, this flatness condition is satisfied when 
$X$ is $\ptst$-smooth.

\subsection{} \emph{Proof of \ref{thm:blow-up-description-of-w-lower-s}.}
	Fix, for the moment, an \'etale algebraic $X$-space $U$.
	Let $Y_U$ denote $\wnls{n}(U)\times_S U$, and let $Z_U$ denote $\underSpec\sB_U$.
	Let $\cat_U$ denote the full subcategory of algebraic spaces over $Y_U$ consisting
	of objects $T$ which are $\m$-flat and satisfy
		\begin{equation} \label{eq:internal-ckwk}
			I_U\sO_T\subseteq\m^{n+1}\sO_T,
		\end{equation}
	where $I_U$ is the ideal sheaf defined in~\ref{subsec:affine-blow-up-notation};
	let $\cat^{\mathrm{aff}}_U$ denote the full subcategory of $\cat_U$ consisting
	of objects which are affine over $Y_U$. 
	
	First, observe that $\wnls{n+1}(U)$ is the terminal object of $\cat_U$. Indeed,
	by~\ref{pro:blow-up-description-of-w-lower-s}, it is enough to show 
	that $\wnls{n+1}(U)$ is $\m$-flat; this is true because, 
	by~\ref{pro:wnls-preserves-formally-etale-etc}, it is \'etale over $\wnls{n+1}(X)$,
	which is $\m$-flat by assumption.

	Second, observe that $Z_U$ is the terminal object of $\cat^{\mathrm{aff}}_U$: it is an object of 
	$\cat^{\mathrm{aff}}_U$
	by~\ref{eq:B-congruence-condition}, and it is terminal by the definition of \emph{generated}.
	
	Because of these two terminal properties, the theorem is
	equivalent to the statement that there exists a map $\wnls{n+1}(X)\to Z_X$ of $Y_X$-spaces,
	which is what we will prove.
	
	Let $\mathsf{D}$ be a diagram of \'etale algebraic spaces $U$ over $X$ (as above)
	such that each space $U$ in the diagram is an object of $\affrel_S$ and such that the induced 
	map 
		\begin{equation} \label{eq:internal-cwisl}
			\colim_{U\in\mathsf{D}}\wnls{n+1}(U)\longmap\wnls{n+1}(X)
		\end{equation}
	is an isomorphism. The existence of $\mathsf{D}$ follows from~\ref{rem:wnls-good-cover}.
	(One can in fact take $\mathsf{D}$ to consist of all such spaces $U$.)
	Then, for any map $b\:V\to U$ of $\mathsf{D}$, the space $\wnls{n+1}(V)$ is an object 
	of $\cat^{\mathrm{aff}}_U$. Indeed, the induced map $\wnls{n+1}(V)\to Y_U$ is affine, because 
	both the source
	and the target are affine schemes (\ref{eq:wnls-of-affine-is-lambda-circle});
	and (\ref{eq:internal-ckwk}) is satisfied because we have $I_U\sO_{Y_V}\subseteq I_V$.
	
	Therefore, by the terminal property of $Z_U$,
	for any such map $b\:V\to U$, there is a unique map 
	$F(b)\:\wnls{n+1}(V)\to Z_U$ of $Y_U$-spaces. In particular,
	the induced diagram
		\[
		\xymatrix@C=40pt{
		\wnls{n+1}(V) \ar^-{F(\id_V)}[r]\ar_{\wnls{n+1}(b)}[d] & Z_V \ar[d] \\
		\wnls{n+1}(U) \ar^-{F(\id_U)}[r] & Z_U 
		}
		\]
	commutes. In particular, the compositions 
		\[\wnls{n+1}(U)\xright{F(\id_U)} Z_U \longmap Z_X\]
	form
	a compatible family of $Y_X$-maps, as $U$ runs over $\mathsf{D}$. This induces a $Y_X$-map
		\[
		\colim_{U\in\mathsf{D}}\wnls{n+1}(U)\to Z_X.
		\]
	On other hand, (\ref{eq:internal-cwisl}) is an isomorphism of $Y_X$-spaces.
	Thus there exists a map $\wnls{n+1}(X)\to Z_X$ of $Y_X$-spaces,
	which completes the proof.
\qed

\subsection{} {\em Non-smooth counterexample.}
We cannot remove the assumption above that $\wnls{n+1}(X)$ is $\m$-flat.
Indeed, the example in~\ref{subsec:wnls-does-not-preserve-flatness} shows
that the locus of $\wnls{n}(X)$ over the complement of $\Spec\sO_S/\m$ can fail to be
dense in $\wnls{n}(X)$.

\begin{corollary}\label{cor:co-ghost-map-is-affine}
	If $X$ is $\ptst$-smooth, then the co-ghost map
		\[
		\entrymodifiers={+!!<0pt,\fontdimen22\textfont2>}
		\def\objectstyle{\displaystyle}
		\xymatrix@C=25pt{
			\wnls{n}(X)\ar^-{\cgh{\leq n}}[r] &  X^{[0,n]}
		}
		\]
	is affine.  It is an isomorphism away from $\ptst$.  
\end{corollary}

It would be interesting to know whether this is true for arbitrary algebraic spaces $X$ over $S$.

\begin{proof}
	If $\ptst$ is empty, then $\cgh{\leq n}$ is an isomorphism. If not, 
	write $\ptst=\ptst'\smcoprod\ptst''$, where $\ptst''$ consists of a single element.
	Let $n'$ and $n''$ denote the projections
	of $n$ onto $\bN^{(\ptst')}$ and $\bN^{(\ptst'')}$.
	Then by~(\ref{eq:geom-witt-iterate-lower-s}), the map $\cgh{\leq n}$ can be
	identified with the composition
		\[
		\entrymodifiers={+!!<0pt,\fontdimen22\textfont2>}
		\def\objectstyle{\displaystyle}
		\xymatrix@C=45pt{
		\wnls{n''}\bigl(\wnls{n'}(X)\bigr) \ar^-{\cgh{\leq n''}}[r] &
			\bigl(\wnls{n'}(X)\bigr)^{[0,n'']} \ar^(0.55){(\cgh{\leq n'})^{[0,n'']}}[r] &
			\bigl(X^{[0,n']}\bigr)^{[0,n'']}.
		}
		\]
	By~\ref{pro:wnls-preserves-formally-etale-etc} and (\ref{eq:global-base-jet-localization}),
	the space $\wnls{n'}(X)$ is	$\ptst$-smooth. Therefore, by induction, the second map above is 
	affine and is an isomorphism away from $\ptst$. Thus to show the first map is 
	affine and is an isomorphism away from $\ptst$, it is enough to prove the corollary itself in
	the case where $\ptst$ consists of a single element.

	In that case, $\cgh{\leq n}$ factors as follows
		\[
		\wnls{n}(X) \longlabelmap{f} 
			\wnls{n-1}(X)\times_S X \longlabelmap{f\times\id_X} 
			\bigl(\wnls{n-2}(X)\times_S X\bigr)\times_S X
			\longmap \cdots \longmap X^{[0,n]}.
		\]
	Since $X$ is $\ptst$-smooth, each $\wnls{i}(X)$ is $\ptst$-smooth and hence $\ptst$-flat.
	Thus, by~\ref{thm:blow-up-description-of-w-lower-s}, each of these maps is affine
	and an isomorphism away from $\ptst$. Therefore so is their composition $\cgh{\leq n}$.
\end{proof}

\bibliography{references}
\bibliographystyle{plain}

\end{document}